\documentclass[review]{elsarticle}

\usepackage{lineno,hyperref}
\modulolinenumbers[5]

\journal{Journal of Differential Equations}
\usepackage{amsmath}
\usepackage[toc,page]{appendix}
\usepackage{amssymb}
\usepackage{bbm}
\usepackage{amsthm}
\usepackage{enumerate}
\usepackage{graphicx,epstopdf}
\usepackage{bold-extra}
\usepackage[margin=1.97cm]{geometry}
\usepackage{hyperref}
\usepackage{color}
\usepackage{tikz,pgfplots}
\usepackage{caption}
\newtheorem{theorem}{Theorem}[section]
\newtheorem{remark}{Remark}[section]
\newtheorem{definition}{Definition}[section]

\newtheorem{lemma}{Lemma}[section]
\newtheorem{proposition}{Proposition}[section]

\newtheorem{Op}{Open Problem}[section]
\numberwithin{equation}{section}
\renewcommand{\div}{\operatorname{div\,}}
\newcommand{\curl}{\operatorname{curl}}
\newcommand{\eps}{\varepsilon}
\newcommand{\R}{\mathbb R}

\begin{document}

\begin{frontmatter}

\title{Control at a distance of the motion of a rigid body immersed in a two-dimensional viscous incompressible fluid}

\author{J\'ozsef J. Kolumb\'an  \corref{mycorrespondingauthor}}
\address{CEREMADE, UMR CNRS 7534, Universit\'e Paris-Dauphine, PSL Research University,
Place du Mar\'echal de Lattre de Tassigny, 75775 Paris Cedex 16, France}

\cortext[mycorrespondingauthor]{Corresponding author}
\ead{jozsi\_ k91@yahoo.com }

\begin{abstract}
We consider the motion of a rigid body immersed in a two-dimensional viscous incompressible fluid with Navier slip-with-friction conditions at the solid boundary. The fluid-solid system occupies the whole plane. We prove the small-time exact controllability of the position and velocity of the solid when the control takes the form of a distributed force supported in a compact subset (with nonvoid interior) of the fluid domain, away from the body. The strategy relies on the introduction of a small parameter: we consider fast and strong amplitude controls for which the ``Navier-Stokes+rigid body'' system behaves like a perturbation of the ``Euler+rigid body'' system. 
By the means of a multi-scale asymptotic expansion we construct a controlled solution to the ``Navier-Stokes+rigid body'' system thanks to some controlled solutions to ``Euler+rigid body''-type systems and to a detailed analysis of the influence of the boundary layer on the solid motion.
\end{abstract}

\begin{keyword}
Fluid-solid interaction; impulsive control; vanishing viscosity; Navier-Stokes equations; asymptotic expansion; Navier ``slip-with-friction'' boundary conditions; boundary layers; coupled ODE-PDE system; control problem. 
\end{keyword}

\end{frontmatter}

\newpage
\tableofcontents
\newpage
%
%
%
%
%
%%%%%%%%%%%%%%%%%%%%%%%%%%%%%%%%%%%%%%%%%%%%%%%%%%%%%%%%%%%%%%%%%%%%%%%%%%%%%%%%%%%%%%%%%%%%%%%%%%%%%%%%
%
%
%
%
\section{Introduction}
\label{INTRO}

In this section we present the fluid-solid model we consider and we state our main result.

\subsection{The mathematical model}

We split the two dimensional plane into two disjoint parts: the closed part ${\mathcal S}(t)$ representing the solid and the open part ${\mathcal F}(t)=\mathbb{R}^2\setminus\mathcal{S}(t)$ filled with fluid. These parts depend on time $t\in[0,T]$, where $T>0$. Furthermore, we assume that ${\mathcal S}(t)$ is smooth and simply connected. 
On the fluid part ${\mathcal F}(t)$, the velocity field ${u}:[0,T]\times\overline{\mathcal{F}(t)} \rightarrow \mathbb{R}^2$ and the pressure field ${\pi}:[0,T]\times\overline{\mathcal{F}(t)} \rightarrow \mathbb{R}$ satisfy the incompressible Navier-Stokes equations with an added source term $\xi:[0,T]\times\overline{\mathcal{F}(t)} \rightarrow \mathbb{R}^2$, that is
\begin{eqnarray}\label{eu}
\begin{split}
\frac{\partial u}{\partial t}+(u\cdot\nabla)u + \nabla \pi -\Delta u= \xi \   \text{ and } \ \div u = 0
\ \text{ for } \ t\in[0,T]   \text{ and  } x \in \mathcal{F}(t) .
\end{split}
\end{eqnarray}
Furthermore, we assume that the support of $\xi$ is in a smooth, 
compact, simply 
connected 
 set $\Omega_c \subset \mathbb{R}^2$ with non-empty interior.

We consider an impermeability boundary condition and a Navier slip-with-friction condition on the solid boundary, namely,
\begin{equation} \label{bc1}
u \cdot n = u_{S}\cdot n\ ,\ \ (D(u)n)_{\text{tan}} = - \mu (u-u_S)_{\text{tan}}\ \text{on}\ \partial \mathcal{S}(t),
\end{equation}
where $u_{S}$ denotes the solid velocity described below, ${n}$ is the unit outward normal vector on $\partial\mathcal{S}(t)$, $\mu\geq 0$ is the coefficient of friction, and for any vector field $f$, we have
\begin{align*}
D(f)=\frac{1}{2}\left(\nabla f +(\nabla f)^T\right)\text{ and }(f)_{\text{tan}}=f-(f\cdot n) n.
\end{align*}
Furthermore, we consider a zero limit condition at infinity, namely,
\begin{equation} \label{bc2}
|u| \to 0 \text{ as } |x|\to+\infty.
\end{equation}
The solid ${\mathcal S}(t)$ is obtained by a rigid movement from ${\mathcal S}(0)=\mathcal{S}_0$, and one can describe its position by the center of mass, ${h}(t)$, and the angle variable with respect to the initial position, ${\vartheta}(t)$. 
Consequently, we have
\begin{align}
\mathcal{S}(t)=h(t)+R(\vartheta(t)) \mathcal{S}_0,
\end{align} 
where the center of mass at initial time is assumed to be $h_0=0$ without loss of generality,  and
$${R}(\vartheta)= \left( \begin{array}{cc}
\cos \vartheta & -\sin \vartheta\\
\sin \vartheta & \cos \vartheta\end{array} \right) .$$
Moreover the solid velocity is hence given by
\begin{align}
u_S (t,x)={h}'(t)+\vartheta'(t) (x-h(t))^{\perp},
\end{align} 
where for $x=(x_1,x_2)$ we denote $x^\perp  = (-x_2,x_1).$

The solid evolves according to Newton's law, and is influenced by the Cauchy stress tensor on the boundary:
\begin{align} \label{newt}
\begin{split}
m  h'' (t) =& - \int_{ \partial \mathcal{S} (t)} (-\pi \text{Id}+2D(u)) \, n \, \, d\sigma,\\
\mathcal{J}  \vartheta'' (t) =&    -\int_{ \partial   \mathcal{S} (t)} (x-  h (t) )^\perp  \cdot (-\pi \text{Id} +2D(u) )n \, d \sigma.
\end{split}
\end{align}
Here the constants $m>0$ and $\mathcal{J}>0$ denote respectively the mass and the
moment of inertia of the body, where the 
fluid is supposed to be homogeneous of density 1, without loss of generality.

We consider the following initial conditions: 
\begin{eqnarray}\label{ic}
u|_{t=0}=u_{0}\text{ for } x\in\mathcal{F}_0,\
h(0)=0,\ h'(0)=h'_0,\ \vartheta(0)=0,\ \vartheta'(0)=\vartheta'_0.
\end{eqnarray}
For the initial data we will asume $u_0\in H^4(\mathcal{F}_0)$ and $\curl u_0 \in L^1(\mathcal{F}_0)$, satisfying the compatibility conditions $\div u_0=0$ in $\mathcal{F}_0$, $u_0\cdot n =(h'_0 + \vartheta'_0 x^\perp)\cdot n$ and $(D(u_0)n)_{\text{tan}} = - \mu (u_0-(h'_0 + \vartheta'_0 x^\perp))_{\text{tan}}$ on $\partial\mathcal{S}_0$. The integrability of the initial vorticity is assumed in order to guarantee that the circulation at infinity is well-defined.

 \begin{figure}[ht]
\centering
\resizebox{1\linewidth}{!}{\includegraphics[clip, trim=0.5cm 17cm 0.5cm 4cm, width=1.00\textwidth]{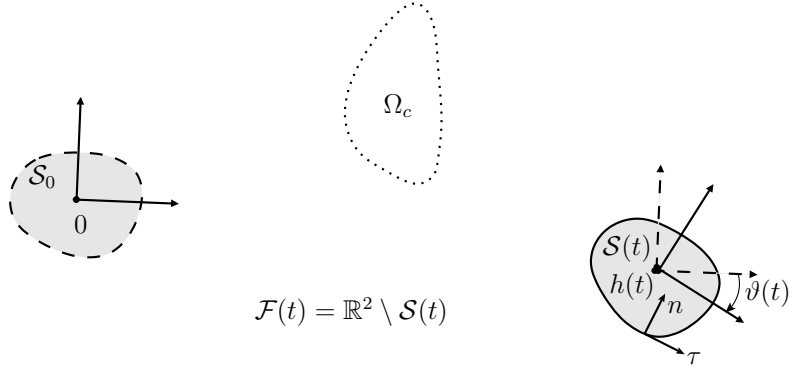}}
\caption{The setting of the control problem}
\end{figure}

 Throughout this paper we will only consider solid trajectories which stay away from the control zone, therefore our construction will satisfy the following condition:
 \begin{align}\label{czone}
 \text{supp }\xi(t,\cdot) \cap \mathcal{S}(t) = \emptyset,\ \forall t\in[0,T].
 \end{align}

\subsection{Definition of weak solutions}

We will now present a notion of Leray-type weak solution to the fluid-solid system.

Let  $\xi \in L^2((0,T)\times\Omega_c)$ be
fixed. 

In order to define a notion of weak solution to \eqref{eu}, \eqref{bc1}, \eqref{bc2}, \eqref{newt}, \eqref{ic}, we introduce for each $t\in[0,T]$ the following spaces:
\begin{align}\label{lerayspaces}
\begin{split}
\mathcal{H}(t):=&\left\{\phi \in L^2(\mathbb{R}^2;\mathbb{R}^2):\ \div \phi =0 \text{ in }\mathbb{R}^2,\ D(\phi)=0\text{ in }\mathcal{S}(t)\right\},\\
\mathcal{V}(t):=&\left\{\phi \in \mathcal{H}(t):\ \nabla \phi \in L^2(\mathcal{F}(t))\right\}.
\end{split}
\end{align}
Note that the above spaces depend on the solution itself through the domains $\mathcal{F}(t)$ and $\mathcal{S}(t)$. Furthermore, according to Lemma 11 in \cite{tem}, for any $\phi\in\mathcal{H}(t)$, there exists $(l_\phi,r_\phi)\in\mathbb{R}^3$ (which may depend on $t$) such that $\phi(x)=\phi_S(x):=l_\phi + r_\phi (x-h(t))^\perp$ in $\mathcal{S}(t)$. Therefore, we extend the initial data $u_0$ by $h'_0 + \vartheta'_0 x^\perp$ in $\mathcal{S}_0$.

We further define the following scalar product, which endows $\mathcal{H}(t)$ with a Hilbert space structure,
\begin{align*}
(u,v)_{\mathcal{H}(t)}:=\int_{\mathcal{F}(t)} u\cdot v\, dx + m l_u \cdot l_v + \mathcal{J} r_u r_v.
\end{align*}

We give the following definition of a weak solution ``\`a la Leray''.
\begin{definition}\label{weak}
We say that
$$u \in C([0,T]; \mathcal{H}(t)) \cap L^2((0,T);\mathcal{V}(t))$$
is a weak solution to the system \eqref{eu}, \eqref{bc1}, \eqref{bc2}, \eqref{newt}, \eqref{ic} if  \eqref{czone} holds, and for all $\phi \in C^\infty([0,T];\mathcal{H}(t))$ such that $\phi|_{\overline{\mathcal{F}(t)}} \in C^{\infty}([0,T]; C_0^{\infty}(\overline{\mathcal{F}(t)};\mathbb{R}^{2}))$ the following holds on $[0,T]$,
\begin{align*}
(u(t,\cdot),\phi(t,\cdot))_{\mathcal{H}(t)}-(u_0,\phi(0,\cdot))_{\mathcal{H}(0)} = \int_0^t (u(s,\cdot),\partial_t\phi(s,\cdot))_{\mathcal{H}(s)}\, ds 
+\int_0^t \int_{\Omega_c} \xi\cdot \phi \, dx \, ds
\\
+ \int_0^t \int_{\mathcal{F}(s)} u \cdot \nabla  \phi \cdot u \, dx \, dt - 2 \int_0^t \int_{\mathcal{F}(s)} D(u) : D(\phi) \, dx \, ds \\ 
- 2\mu \int_0^t \int_{\partial\mathcal{S}(s)}(u(s,\cdot)-u_S(s,\cdot)) \cdot  (\phi(s,\cdot)-\phi_S(s,\cdot))\, d\sigma \, ds,
\end{align*}
where the spaces $\mathcal{H}(t)$ are associated with $\mathcal{S}(t)$ as in \eqref{lerayspaces}, and we have that $\mathcal{S}$ is transported by the flow of $u_S$.
\end{definition}

It can be easily checked, by performing some integration by parts, that a strong solution to \eqref{eu}, \eqref{bc1}, \eqref{bc2}, \eqref{newt}, \eqref{ic}  is also a weak solution in the above sense (see for instance \cite{PS} for further details).

Note that there is a slight abuse of notation in writing $C([0,T]; \mathcal{H}(t))$
for instance, since the domain in which the fluid evolves is also time-dependent. Furthermore, the test functions also depend on the solution $u$ as noted above in the definition of the spaces $\mathcal{H}(t)$ and $\mathcal{V}(t)$.

\subsection{Main result}

Our goal is to investigate the possibility of controlling the solid by the means of prescribing an interior control $\xi$ acting on the fluid. 
In particular we raise the question of driving the solid from a given position and a given velocity to some other prescribed position and velocity. 

In order to ensure that \eqref{czone} holds, we note that we do not need to control on the whole of $\Omega_c$. Instead, we will introduce a set of admissible positions for the solid, such that as long as the final position of the solid is in this set, there exists a fixed open subset of $\Omega_c$, which does not touch the solid neither in its initial nor in its final position, and which we will use as the support of the controls we construct.

More precisely, we set
\begin{align}\label{czoneq}
\mathcal{Q}=\left\{ q := (h,\vartheta)\in\mathbb{R}^3:\ \text{int}(\Omega_c) \setminus \left\{(h+R(\vartheta)\mathcal{S}_0) \cup \mathcal{S}_0 \right\}\neq \emptyset\right\}.
\end{align}
Since $\mathcal{S}_0$ is simply connected, it can be easily checked that $\mathcal Q$ is path-connected.

We are now in position to state our main result regarding the small-time global exact controllability of the solid position and velocity.
\begin{theorem} \label{main}
Consider $\mathcal{S}_0\subset\mathbb{R}^2$ bounded, closed, simply connected with smooth boundary, which is not a disk, and
 $u_{0}\in  H^4(\mathcal{F}_0;\mathbb{R}^{2})$, $\curl u_0 \in L^1(\mathcal{F}_0;\mathbb{R}^{2})$, 
$q_0=0, q_f =(h_{f},\vartheta_f) \in \mathcal{Q}$, 
$h'_0,h'_f\in\mathbb{R}^2,\vartheta'_0,\vartheta'_f\in\mathbb{R}$, such that \begin{gather*}
\div u_0=0 \text{ in }\mathcal{F}_0, \ 
\lim_{|x|\to+\infty}|u_0(x)|=0,\\
u_0 \cdot n = (h'_0+\vartheta'_0 x^{\perp}) \cdot n,\ (D(u_0)n)_{\text{tan}} = - \mu (u_0- (h'_0+\vartheta'_0 x^{\perp}))_{\text{tan}} \text{ on } \partial \mathcal{S}_0.
\end{gather*}
Then there exists $\tilde{T}>0$ such that, for any $T\in(0,\tilde{T}]$, there exists a control $\xi\in L^2((0,T)\times\Omega_c)$, compactly supported in time, and a weak solution $u \in C([0,T]; \mathcal{H}(t)) \cap L^2((0,T);\mathcal{V}(t))$ in the sense of Definition \ref{weak} to the system \eqref{eu}, \eqref{bc1}, \eqref{bc2}, \eqref{newt}, \eqref{ic} with \eqref{czone},
such that we have
$(h,h',\vartheta,\vartheta')(T)=(h_f,h'_f,\vartheta_f,\vartheta'_f)$. 
\end{theorem}

Note that the condition that $\mathcal{S}_0$ is not a disk is essential because a significant step during our proof will rely heavily on a strategy similar to the one presented in \cite{GKS}. As mentioned in the aforementioned paper, it is possible to treat the case of a homogeneous disk with a similar strategy, controlling only the center of mass $h$.

Furthermore, we will present in Section \ref{end} a possible strategy for passing to arbitrary time controllability, given a sufficiently strong autoregularization property of the system (that would also allow for less regular initial data), which up to our best knowledge is currently still an open problem in the literature for this type of systems. 

We also note that the reason for working in the whole plane instead of a bounded domain will be given in Remark \ref{plane} in Section \ref{prel} below, once the main ideas behind our proof have been presented.

\ \par \noindent
{\bf References.} 
Our result can be contrasted with the result from \cite{CMS} (see also \cite{CMS2}  for a gentle exposition) regarding the controllability of the fluid velocity alone in the case of the incompressible Navier-Stokes equation with Navier slip-with-friction boundary conditions. However, in our case we are only interested in controlling the solid position and velocity, and unlike in \cite{CMS} we achieve this without any control on the solid boundary.
Note that the proof relies on the previous results for the Euler equations by means of a rapid and strong control which drives the system in a high Reynolds regime, 
a strategy which originates from  \cite{Coron:NS},  where an interior controllability result was already established. For this purpose in our case we generalize the control result regarding 2D ``perfect fluid + rigid body'' systems from \cite{GKS} to treat the viscous case with a similar strategy, but without using the return method, in contrast to  \cite{Coron:NS,CMS}. Note also that the main result in \cite{GKS} was presented in the absence of vorticity, however, it was mentioned that the effect of the vorticity could be handled by some appropriate control strategy. In this paper there will indeed be some vorticity created at the solid boundary, which we will handle by the means of an asymptotic boundary layer expansion. 

For ``viscous fluid + rigid body'' control systems (with Dirichlet boundary conditions), local null-controllability results have already been obtained in both 2D and 3D, see e.g. \cite{BG, BO, IT}. These results rely on Carleman estimates on the linearized equation, and consequently on the parabolic character of the fluid equation. 
A similar result has been established in \cite{LTT} for the case of the 1D viscous Burgers equation with a new strategy introduced by the authors without the use of any Carleman estimates, and as noted in the aforementioned article, those methods can be extended to other nonlinear parabolic systems.
However, note that the results mentioned above concern local null-controllability for the solid position and the velocities of both the solid and the fluid, whereas
 in this paper we achieve global exact controllability for both the solid position and velocity directly.
 
 Let us also mention some stabilization results regarding ``viscous fluid + rigid body'' systems in a bounded domain with Dirichlet boundary conditions, see \cite{BT1} for the 2D and 3D cases, respectively \cite{BT2} for a simplified model in the 1D case. In the aforementioned articles the authors stabilize the position and velocity of the solid and the velocity of the fluid using a feedback control on the exterior boundary of the fluid domain, assuming that the initial data of the system is close to a stationary state (which is not necessarily assumed to be zero).

A different type of problem regarding fluid-solid interactions is that of a deformable body in a fluid, regarding the dynamics of swimming, see for instance \cite{fish1, fish2, fish3, fish4, fish5} for the viscous case, respectively \cite{ChambrionMunnier} for the inviscid case. In the case of such problems, the control is no longer at a distance, rather it consists of the deformation of the body itself.

\ \par \noindent
{\bf Generalizations and open problems.}
A natural generalization of the problems above would be the passage from the two-dimensional case to the three-dimensional one. The main difficulty in adapting our methods to the 3D case  (apart from the Cauchy theory of the 3D system) is the use of complex analysis to explicitly construct the spacial part of the control (see the Appendices of the paper). However, one could replace these arguments by a Cauchy-Kovalevskaya type construction and a higher dimensional generalization of Runge's theorem (c.f. \cite{For}), so that a similar result could be established in the three-dimensional case.

Furthermore, one could also be interested in controlling several solids. Indeed, one can see in the Appencides of the paper that the construction of the spacial part of the vector field associated with control is quite local around the solid, so the arguments in our proof should be adaptable to the case of multiple solids, also guaranteeing that there is no collision between the solids.

Another interesting open problem is that of the motion planning of a rigid body immersed in a viscous incompressible fluid. Namely, suppose that we have a fixed curve $\Gamma\in C^{2}([0,T];\mathbb{R}^3)$, and the conditions of Theorem \ref{main} are satisfied with $(h_0,\vartheta_0,h'_0,\vartheta'_0)=(\Gamma(0), \Gamma'(0))$
, $(h_f,\vartheta_f,h'_f,\vartheta'_f)=(\Gamma(T), \Gamma'(T))$. Does there exist a control and a solution to the fluid-solid system as described in Theorem \ref{main}, satisfying in addition $\Gamma = (h,\vartheta)$ on $[0,T]$?
Even the approximate motion planning in $C^2$, i.e. the same statement as above but with 
$  \|  \Gamma - (h,\vartheta) \|_{C^{2}([0,T])} \leq \eps$ (with $\varepsilon>0$ arbitrary) instead of $\Gamma = (h,\vartheta)$, is an open problem. Furthermore, as mentioned in \cite{GKS}, the motion planning for a rigid body in an inviscid fluid is also open. However, there might be some hope to adapt certain techniques presented in this article (specifically in Section \ref{seu0} to pass from approximate controllability to exact controllability), in order to tackle this problem in the future.

Finally, we mention that the global exact controllability of a ``viscous fluid + rigid body'' system with Dirichlet ``no slip'' boundary conditions is completely open, and a very challenging problem due to the fact that the Dirichlet boundary conditions create boundary layers with a larger amplitude than in the case of Navier slip-with-friction boundary conditions. Note that even the problem of controlling only the fluid velocity in such a context is open and similarly challenging, some recent advances have been made in \cite{CMSZ} in the very particular case when the domain is assumed to be a rectangle, using an added distributed phantom force.

\paragraph{Plan of the paper} 
 The paper is organized as follows.

 In Section \ref{prel} we give some preliminary results, such as reducing Theorem \ref{main} to the case of a fixed domain and small viscosity, which we further reduce to constructing an appropriate asymptotic expansion with respect to the viscosity.
 
In Section \ref{seu0} we construct the terms of order $\mathcal{O}(1)$ in the asymptotic expansion by a generalization of the geodesic method used in \cite{GKS}.

In Section \ref{BL} we construct the boundary layer profiles associated with $u^0$, which appear in particular at order $\mathcal{O}(\sqrt{\varepsilon})$ and $\mathcal{O}(\varepsilon)$.

In Section \ref{seu1} we construct the linearized terms of $\mathcal{O}(\varepsilon)$ in the asymptotic expansion by a different impulsive control strategy.

In Section \ref{remain} we construct and estimate the remainder in the asymptotic expansion and prove that it converges to zero in an appropriate space.

We conclude the article in Section \ref{end} with some visual representations of the controls constructed during our strategy, as well as some remarks regarding the passage to arbitrary time controllability.

In Appendix \ref{comproofs} we explicitly construct the controls by the means of complex analysis.

%
%
%%%%%%%%%%%%%%%%%%%%%%%%%%%%%%%%%%%%%%%%%%%%%%%%

\section{Preliminary reductions}\label{prel}

In this section we will prove that Theorem \ref{main} can be reduced to the case where the fluid domain is fixed and the fluid viscosity is small. Furthermore, we will show that in this case one can introduce a vanishing viscosity asymptotic expansion for the solid trajectory in order to prove our main result.

\subsection{A reduction of Theorem \ref{main} to the case of a fixed domain and small viscosity}\label{previsc}

The goal of this section is to prove that Theorem \ref{main} can be deduced from a controllability result for the following system:
\begin{eqnarray}\label{eucs}
\begin{split}
\frac{\partial u^\varepsilon}{\partial t}+(u^\varepsilon&-u^\varepsilon_S)\cdot\nabla u^\varepsilon + r^\varepsilon (u^{\varepsilon})^\perp+ \nabla \pi^\varepsilon -\varepsilon\Delta u^\varepsilon  =0 \   \text{ and } \ \div u^\varepsilon = g^\varepsilon
\ \text{ for } x \in \mathcal{F}_0, \\
u^\varepsilon\cdot n &= u^\varepsilon_S \cdot n,\ \ (D(u^\varepsilon)n)_{\text{tan}} = - \mu (u^\varepsilon-u^\varepsilon_S)_{\text{tan}}\ \text{ for } x \in \partial\mathcal{S}_0,\ \lim_{|x|\to+\infty}|u^\varepsilon|=0,\\
m   (l^\varepsilon)'&= - \int_{ \partial \mathcal{S}_0} (-\pi^\varepsilon \text{Id}+2\varepsilon D(u^\varepsilon)) \, n \, \, d\sigma - m  r^{\varepsilon}  (l^{\varepsilon})^\perp,\\
\mathcal{J}  \ (r^\varepsilon)'  &=   -\int_{ \partial   \mathcal{S}_0} x^\perp  \cdot (-\pi^\varepsilon\text{Id} +2\varepsilon D(u^\varepsilon) )n \, d \sigma,
\end{split}
\end{eqnarray}
where $u^\varepsilon_S (t,x)= l^\varepsilon(t) + r^\varepsilon(t) x^\perp,$
for  $t\in[0,T],$ with $\varepsilon>0$,
$u^\varepsilon(0,\cdot)=\varepsilon u_0(\cdot), (l^\varepsilon,r^\varepsilon) (0) =\varepsilon (h'_0,\vartheta'_0),$ and the control term is now $g^\varepsilon\in C_{0}^{\infty}((0,T)\times\mathcal{F}_0)$ such that $\text{supp }g^\varepsilon (t,\cdot) \subset R(\vartheta^\varepsilon(t))^T (\Omega_c -h^\varepsilon(t))$ and $\int g^\varepsilon (t,\cdot) \, dx =0$, for any $t\in[0,T]$. 

Note that in \eqref{eucs} the viscosity coefficient $\varepsilon$ appears in front of the term $D(u^\varepsilon)$ in the solid equations (as part of the Cauchy stress tensor), but not in the Navier slip-with-friction boundary condition on $\partial \mathcal{S}_0$. This will be essential in the asymptotic expansion presented in Section \ref{fluidassexp}.

Furthermore, we may associate the solid position $(h^\varepsilon,\vartheta^\varepsilon)$, which no longer plays a role directly in solving system \eqref{eucs}, through
\begin{align}\label{frame2}
h^\varepsilon(t)=\int_0^{t} R(\vartheta^\varepsilon(s)) l^\varepsilon(s) \, ds,\quad
\vartheta^\varepsilon(t)=\int_0^{t} r^{\varepsilon}(s) \, ds,
\end{align}
for $t\in[0,T].$

We have the following adaptation of \eqref{czone} for $g^\varepsilon$,
 \begin{align}\label{gczone}
 \text{supp }g^\varepsilon(t,\cdot) \cap \mathcal{S}_0 = \emptyset,\ \forall t\in[0,T].
 \end{align}

Let us first give a definition of so-called ``very weak solutions'' to viscous fluid-solid models as \eqref{eucs} with non-zero divergence (similarly to the same notion for the fluid alone, as done for instance in \cite{veryw}). To do so, we further introduce the following spaces:
\begin{align*}
\mathcal{H}:=\left\{\phi \in L^2(\mathbb{R}^2;\mathbb{R}^2):\ D(\phi)=0\text{ and }\div \phi =0 \text{ in }\mathcal{S}_0\right\},\
\mathcal{V}:=\ \mathcal{H} \cap H^1(\mathcal{F}_0).
\end{align*}
Once more, note that for any $\phi\in\mathcal{H}$, there exists $(l_\phi,r_\phi)\in\mathbb{R}^3$ such that $\phi(x)=\phi_S(x):=l_\phi + r_\phi x^\perp$ in $\mathcal{S}_0$. Therefore, we may once again extend the initial data $u_0$ by $h'_0 + \vartheta'_0 x^\perp$ in $\mathcal{S}_0$.
We consider the following scalar product, which endows $\mathcal{H}$ with a Hilbert space structure,
\begin{align*}
(u,v)_{\mathcal{H}}:=\int_{\mathcal{F}_0} u\cdot v\, dx + m l_u \cdot l_v + \mathcal{J} r_u r_v.
\end{align*}
Note that $(\cdot,\cdot)_{\mathcal{H}}$ coincides with the scalar product for the Hilbert space $\mathcal H (0)$ defined in \eqref{lerayspaces}, however, $\mathcal H (0)$ is a strict subspace of $\mathcal H$.
\begin{definition}\label{eweak}
We say that
$$u^\varepsilon \in C([0,T]; \mathcal{H}) \cap L^2((0,T);\mathcal{V})$$
is a very weak solution to the system \eqref{eucs} if \eqref{gczone} holds, if
we have
$$\div u^\varepsilon = g^\varepsilon,\text{ for a.e. }t\in[0,T],$$
and if, for all $\phi \in C^\infty([0,T];\mathcal{H}(0))$ such that $\phi|_{\overline{\mathcal{F}_0}} \in C^{\infty}([0,T]; C_0^{\infty}(\overline{\mathcal{F}_0};\mathbb{R}^{2}))$, the following holds on $[0,T]$,
\begin{align*}
(u^\varepsilon(t,\cdot),\phi(t,\cdot))_{\mathcal{H}}-\varepsilon(u_0,\phi(0,\cdot))_{\mathcal{H}} = \int_0^t (u^\varepsilon(s,\cdot),\partial_t\phi(s,\cdot))_{\mathcal{H}}\, ds 
+\int_0^t \int_{\mathcal{F}_0} g^\varepsilon \, u^\varepsilon \cdot \phi \, dx \, ds
\\
+ \int_0^t \int_{\mathcal{F}_0}\left[ (u^\varepsilon -u^\varepsilon_S)\cdot \nabla  \phi \cdot u^\varepsilon  -r^\varepsilon (u^\varepsilon)^\perp \cdot \phi \right]\, dx \, ds -\int_0^t m r^\varepsilon (l^\varepsilon)^\perp  \cdot l_\phi \, ds \\- 2 \varepsilon \int_0^t \int_{\mathcal{F}_0} D(u^\varepsilon) : D(\phi) \, dx \, ds 
- 2\varepsilon \mu \int_0^t \int_{\partial\mathcal{S}_0}(u^\varepsilon-u^\varepsilon_S) \cdot  (\phi-\phi_S)\, d\sigma \, ds .
\end{align*}
\end{definition}
Note that in fact the above definition can be extended to less regular divergence terms $g^\varepsilon$ (as done in \cite{veryw} for example). But for our purposes the case of smooth $g^\varepsilon$ will suffice, since our construction for such solutions will rely on a linear decomposition into a smooth term which has divergence $g^\varepsilon$, and the remaining term defined as the weak solution of a Navier-Stokes type system with zero divergence (in the Leray sense, since the divergence is considered to be zero).

We claim that Theorem \ref{main} follows from the following result.

\begin{theorem} \label{maindiv}
Consider $T>0$, $\mathcal{S}_0\subset\mathbb{R}^2$ bounded, closed, simply connected with smooth boundary, which is not a disk, and $u_{0}\in  H^4( \mathcal{F}_0;\mathbb{R}^{2})$, $\curl u_{0}\in L^1( \mathcal{F}_0;\mathbb{R}^{2})$, 
$q_0=0, q_f =(h_{f},\vartheta_f) \in \mathcal{Q}$, 
$h'_0,h'_f\in\mathbb{R}^2,\vartheta'_0,\vartheta'_f\in\mathbb{R}$, such that 
\begin{gather*}
\div u_0=0 \text{ in }\mathcal{F}_0, \ 
\lim_{|x|\to+\infty}|u_0(x)|=0,\\
u_0 \cdot n = (h'_0+\vartheta'_0 x^{\perp}) \cdot n,\ (D(u_0)n)_{\text{tan}} = - \mu (u_0- (h'_0+\vartheta'_0 x^{\perp}))_{\text{tan}} \text{ on } \partial \mathcal{S}_0.
\end{gather*}
Then there exists $\bar{\varepsilon}\in(0,1)$ such that, for any $\varepsilon\in(0,\bar{\varepsilon}]$, there exists a control $g^\varepsilon\in C_{0}^{\infty}((0,T)\times\mathcal{F}_0)$ with $\int g^\varepsilon =0$, $\text{supp }g^\varepsilon (t,\cdot) \subset R(\vartheta^\varepsilon(t))^T (\Omega_c -h^\varepsilon(t))$, and a very weak solution $u^\varepsilon \in C([0,T]; \mathcal{H}) \cap L^2((0,T);\mathcal{V})$ in the sense of Definition \ref{eweak} to \eqref{eucs}%, \eqref{frame2}
such that we have
\begin{align}\label{divex}
(h^\varepsilon,l^\varepsilon,\vartheta^\varepsilon,r^\varepsilon)(T)=(h_f,\varepsilon R(\vartheta_f)^T h'_f,\vartheta_f,\varepsilon\vartheta'_f),
\end{align}
 with $(h^\varepsilon,\vartheta^\varepsilon)$ given by \eqref{frame2}. 
\end{theorem}

\begin{proof}[Proof of Theorem \ref{main} from Theorem \ref{maindiv}]

We will show that given $T>0$, $g^\varepsilon$ as above, there exists some $\xi \in L^2((0,\varepsilon T)\times\Omega_c)$ and some appropriate transformations such that 
we may deduce the existence of a solution to system \eqref{eu}, \eqref{bc1}, \eqref{bc2}, \eqref{newt}, \eqref{ic} on $[0,\varepsilon T]$ from the existence of a  solution to system \eqref{eucs}, \eqref{frame2} on $[0,T]$.
We will introduce a change of variables for passing from small viscosity to viscosity $1$ and for passing from a fixed domain to a moving domain, and we will switch from a control on the divergence to a control on the evolution equation via the Bogovskii operator (see for instance \cite{Bog} or \cite{GHH}), defined as follows.
\begin{definition}\label{bogovski}
Given a smooth, bounded 
domain $\Omega\subset\mathbb{R}^2$, there exists an operator $B:C_0^\infty(\Omega)\to C_0^\infty(\Omega)^2$ such that, for any $g\in C_0^\infty(\Omega)$ with $\int g =0$, we have $\div B g =g$.
Furthermore, $B\in\mathcal{L}(W^{s,p}_0(\Omega),W^{s+1,p}_0(\Omega)^2)$, for any $1<p<+\infty$, $s\geq 0$. Also observe that we may extend $B g$ by $0$ outside of $\Omega$.
\end{definition}
From now on we consider $B$ to be the Bogovskii operator associated with $\Omega_c$.
Given $g^\varepsilon\in C_{0}^{\infty}((0,T)\times\mathcal{F}_0)$ with $\text{supp }g^\varepsilon (t,\cdot) \subset R(\vartheta^\varepsilon(t))^T (\Omega_c -h^\varepsilon(t))$ and $u^\varepsilon \in C([0,T]; \mathcal{H}) \cap L^2((0,T);\mathcal{V})$  as in Theorem \ref{maindiv}, we define
\begin{align}\label{chov}
\begin{split}
g(t,x)&:=\varepsilon^{-1} g^\varepsilon\left(\varepsilon^{-1} t,R\left(\vartheta^\varepsilon\left(\varepsilon^{-1} t\right)\right)^T\left(x-h^\varepsilon\left(\varepsilon^{-1} t\right)\right)\right),\\
U(t,x)&:=\varepsilon^{-1} R\left(\vartheta^\varepsilon\left(\varepsilon^{-1} t\right)\right)u^\varepsilon\left(\varepsilon^{-1} t,R\left(\vartheta^\varepsilon\left(\varepsilon^{-1} t\right)\right)^T\left(x-h^\varepsilon\left(\varepsilon^{-1} t\right)\right)\right),\\
h'(t)&:=\varepsilon^{-1} R\left(\vartheta^\varepsilon\left(\varepsilon^{-1} t\right)\right)l^\varepsilon\left(\varepsilon^{-1} t\right),\quad
\vartheta'(t):=\varepsilon^{-1} r^\varepsilon\left(\varepsilon^{-1} t\right),
\end{split}
\end{align}
for $t\in [0,\varepsilon T]$ and $x\in\overline{\mathcal{F}(t)}$.

We set
\begin{align}\label{xig}
\xi:=-\frac{\partial B g}{\partial t} + B g \cdot \nabla B g -  U \cdot \nabla B g - B g \cdot \nabla U +\Delta B g \in L^2((0,\varepsilon T)\times\Omega_c)
\end{align} 
and $u:=U-B g$. It can  be easily checked that $(u,h,\vartheta,\xi)$ defined in this way allow us to deduce a weak solution in the sense of Definition \ref{weak}.

Finally, we may conclude the proof by
 noting that
we have 
$(h,\vartheta)(\varepsilon T)=(h^\varepsilon,\vartheta^\varepsilon)(T)=(h_f,\vartheta_f)$ and $(h',\vartheta')(\varepsilon T)=\varepsilon^{-1} (R(\vartheta^\varepsilon (T))l^\varepsilon(T),r^\varepsilon(T))=(h'_f,\vartheta'_f)$,  by using \eqref{divex}.
Therefore,
the conclusions of Theorem \ref{main} are satisfied with $\varepsilon T$ instead of $T$.
\end{proof}

\begin{remark}
  \label{smallflux}
In Theorem~\ref{maindiv} the control $g^\varepsilon$ can be chosen with an arbitrarily small total flux through its support, that is 
 for any $\delta_c > 0$,  there exists a control $g^\varepsilon$ and a very weak solution $u^\varepsilon$ satisfying the properties of Theorem  \ref{maindiv}  and such that moreover 
$$\left\vert \int_0^T\displaystyle \int_{R(\vartheta^\varepsilon(t))^T (\Omega_c -h^\varepsilon(t))} \, \left(g^\varepsilon(t,x)\right)_{-}  \, \, dx\,  dt \right\vert < \delta_c.$$
See Section \ref{flux} for more explanations.
Let us mention that such a small flux condition cannot be guaranteed in the results \cite{Coron:NS,CMS} regarding the controllability of the Navier-Stokes equations.
\end{remark}

%%%%

\subsection{Proving Theorem \ref{maindiv} by the means of an asymptotic expansion for the solid trajectory}
\label{SAE}

We introduce the following asymptotic expansion for the solid trajectory:
\begin{align}\label{scales}
\begin{split}
h^\varepsilon = h^0 + \varepsilon h^1 + \varepsilon h^\varepsilon_R,\
\vartheta^\varepsilon = \vartheta^0 + \varepsilon \vartheta^1 + \varepsilon \vartheta^\varepsilon_R,\\
l^\varepsilon = l^0 + \varepsilon l^1 + \varepsilon l^\varepsilon_R,\
r^\varepsilon = r^0 + \varepsilon r^1 + \varepsilon r^\varepsilon_R,
\end{split}
\end{align}
with $(l^0,r^0),(l^1,r^1),(l^\varepsilon_R,r^\varepsilon_R)\in L^\infty (0,T)$.

Now suppose that we could exactly drive $(h^0,\vartheta^0)(T)$ to $(h_f,\vartheta_f)$, while we had $(h^1,\vartheta^1)(T)$ and $(h^\varepsilon_R,\vartheta^\varepsilon_R)(T)$ bounded in $\varepsilon>0$. It would follow that we have approximately driven $(h^\varepsilon,\vartheta^\varepsilon)(T)$ to $(h_f,\vartheta_f)$, for $\varepsilon>0$ small enough.

Furthermore, suppose that at the same time we managed to exactly drive $(l^0,r^0)(T)$ to $0$ and (approximately) drive $(l^1,r^1)(T)$ to $(R(\vartheta_f)^T h'_f,\vartheta'_f)$, while $(l^\varepsilon_R,r^\varepsilon_R)(T)\to 0$ as $\varepsilon\to 0^+$. It would follow that we have approximately driven $(l^\varepsilon,r^\varepsilon)(T)$ to $\varepsilon(R(\vartheta_f)^T h'_f,\vartheta'_f)$, for $\varepsilon>0$ small enough.

However, we want to prove the exact controllability in \eqref{divex}. To do so, we may pass from the above mentioned approximate controllability to exact controllability by the means of a topological argument of Brouwer-type, as done in \cite{GKS}. This further requires some continuity property for our whole construction with respect to the target data for the solid trajectory. Therefore, we will realize the above construction not only for $(h_f,\vartheta_f,R(\vartheta_f)^T h'_f,\vartheta'_f)$ as given in Theorem \ref{maindiv}, but for any $(h_1,\vartheta_1,l_1,r_1)$ in a small enough ball centered at $(h_f,\vartheta_f,R(\vartheta_f)^T h'_f,\vartheta'_f)$, such that the construction depends continuously on $(h_1,\vartheta_1,l_1,r_1)$.

More precisely, we claim that Theorem \ref{maindiv} follows from the following controllability result.
\begin{theorem} \label{mainsc}
Suppose that the conditions of Theorem \ref{maindiv} are verified.
Let $\kappa>0$ such that 
$$\text{int}(\Omega_c) \setminus \bigcup_{(h,\vartheta)\in \overline{B}((h_f,\vartheta_f),\kappa)} \left\{(h+R(\vartheta)\mathcal{S}_0) \cup \mathcal{S}_0 \right\}\neq \emptyset.$$
For any  $\nu>0$ there exists $\varepsilon_0=\varepsilon_0(\nu)>0$, which only depends on $\nu>0$, such that for any $\varepsilon\in(0,\varepsilon_0]$, for any
 $(h_1,\vartheta_1,l_1,r_1)\in\overline{B}((h_f,\vartheta_f,R(\vartheta_f)^T h'_f,\vartheta'_f),\kappa)$,
there exists a control $g^\varepsilon\in C_{0}^{\infty}((0,T)\times\mathcal{F}_0)$ with $\int g^\varepsilon =0$, $\text{supp }g^\varepsilon (t,\cdot) \subset R(\vartheta^\varepsilon(t))^T (\Omega_c -h^\varepsilon(t))$, and a very weak solution $u^\varepsilon \in C([0,T]; \mathcal{H}) \cap L^2((0,T);\mathcal{V})$ to \eqref{eucs}, \eqref{frame2},
such that \eqref{scales} holds, and we have
\begin{align}\label{mainsc1}
\begin{split}
(h^0,\vartheta^0,l^0,r^0)(T)=(h_1,\vartheta_1,0,0),&\quad
|(l^1,r^1)(T)-(l_1,r_1)| \leq \nu,\\
|(h^1,\vartheta^1)(T)| \leq C,&\quad
|(l^\varepsilon_R,r^\varepsilon_R)(T)| \leq C \varepsilon^{1/8},
\end{split}
\end{align}
where $C>0$ can depend on $\nu>0$, but is independent of $\varepsilon>0$ and $(h_1,\vartheta_1,l_1,r_1)\in\overline{B}((h_f,\vartheta_f,R(\vartheta_f)^T h'_f,\vartheta'_f),\kappa)$. Furthermore, the map 
\begin{align}\label{mainsc2}
(h_1,\vartheta_1,l_1,r_1)\in\overline{B}((h_f,\vartheta_f,R(\vartheta_f)^T h'_f,\vartheta'_f),\kappa)\mapsto (h^\varepsilon,\vartheta^\varepsilon,l^\varepsilon
,r^\varepsilon)(T) 
\end{align}
is continuous.
\end{theorem}

\begin{remark}\label{remsc}
We note that we settle for approximate controllability for $(l^1,r^1)$ because it simplifies elements of the proof of Theorem \ref{mainsc} and it will be sufficient to prove Theorem \ref{maindiv}. Furthermore, the introduction of $\varepsilon_0(\nu)>0$ serves the purpose of making the trajectory $(h^\varepsilon,\vartheta^\varepsilon)$ stay sufficiently close to $(h^0,\vartheta^0)$ such that in order to guarantee condition \eqref{gczone}, it suffices to guarantee a similar condition for $(h^0,\vartheta^0)$, which we will detail during the proof of Theorem \ref{mainsc}.
\end{remark}

\begin{proof}[Proof of Theorem \ref{maindiv} from Theorem \ref{mainsc}]

We will conclude the proof of Theorem \ref{maindiv} 
by a topological argument based on the result below borrowed from \cite[pages 32-33]{G-R}.
\begin{lemma}\label{top}
Let $w_0 \in\mathbb{R}^n,$ $\kappa>0$, $f:\overline{B}(w_0,\kappa)\to\mathbb{R}^n$ a continuous map such that we have
$|f(w)-w|\leq \frac{\kappa}{2}$ for any $x$ in $\partial B(w_0,\kappa).$
Then
$B(w_0,\frac{\kappa}{2})\subset f(\overline{B}(w_0,\kappa)).$
\end{lemma}
We set $w_0=(h_f,\vartheta_f,R(\vartheta_f)^T h'_f,\vartheta'_f)$  as in Theorem \ref{maindiv}, and $\kappa>0$ as in Theorem \ref{mainsc}.

Let $T>0$, $\nu>0$ and $(h_1,\vartheta_1,l_1,r_1)\in \overline{B}(w_0,\kappa)$. We apply Theorem \ref{mainsc} to deduce that, for $\varepsilon>0$ small enough, we have
\begin{align}\label{topest}
\begin{split}|(h^\varepsilon,\vartheta^\varepsilon) (T)-(h_1,\vartheta_1)|=&|(h^\varepsilon,\vartheta^\varepsilon) (T)-(h^0,\vartheta^0) (T)|=\varepsilon |(h^1,\vartheta^1) (T)+(h^\varepsilon_R,\vartheta^\varepsilon_R) (T)|\leq C \varepsilon,\\
|\varepsilon^{-1}(l^\varepsilon,r^\varepsilon) (T)- (l_1,r_1)|&\leq\varepsilon^{-1}|(l^\varepsilon,r^\varepsilon) (T)-(l^0,r^0) (T)-\varepsilon (l^1,r^1) (T) |+ \nu= |(l^\varepsilon_R,r^\varepsilon_R) (T)|+\nu\leq C \varepsilon^{\frac{1}{8}}+ \nu,
\end{split}
\end{align}
where, $C=C(\nu)>0$ is independent of  $\varepsilon$ and $(h_1,\vartheta_1,l_1,r_1)$.

Therefore, we fix $\nu=\frac{\kappa}{2\sqrt{5}}$ and  $\bar{\varepsilon}\in(0,\varepsilon_0(\nu)]$ such that
$C\bar{\varepsilon}^{\frac{1}{8}}\leq \frac{\kappa}{2\sqrt{5}},$ and for any $\varepsilon\in(0,\bar{\varepsilon}]$ set $$f(h_1,\vartheta_1,l_1,r_1)=(h^{\varepsilon},\vartheta^{\varepsilon},\varepsilon^{-1}l^{\varepsilon},\varepsilon^{-1}r^{\varepsilon})(T).$$
It follows from \eqref{topest} that $|f(h_1,\vartheta_1,l_1,r_1)-(h_1,\vartheta_1,l_1,r_1)|\leq\frac{\kappa}{2}$, uniformly for $(h_1,\vartheta_1,l_1,r_1)\in \overline{B}(w_0,\kappa)$, and $f$ is continuous due to Theorem \ref{mainsc}.
We may conclude the exact controllability result of Theorem \ref{maindiv}
by setting applying Lemma \ref{top} to deduce that $w_0=(h_f,\vartheta_f,R(\vartheta_f)^T h'_f,\vartheta'_f)$ is in $\text{Range}(f)$.

\end{proof}

\subsection{Proving Theorem \ref{mainsc} - Constructing the control via the asymptotic expansion for the fluid velocity}\label{fluidassexp}

In order to achieve an expansion as in \eqref{scales} for the solid trajectory, we consider controls $g^\varepsilon$ in the form of $g^\varepsilon =g^0 + \varepsilon g^1,$ in the style of \cite{CMS}, and we look for 
the following asymptotic expansion for the fluid velocity and pressure:
\begin{align}\label{scale}
\begin{split}
u^\varepsilon &= u^0 +\sqrt{\varepsilon}\{v\}+ \varepsilon u^1 +\varepsilon\nabla\theta^\varepsilon+\varepsilon\{w\}+ \varepsilon u^\varepsilon_R,\\
\pi^\varepsilon &= \pi ^0 +\varepsilon \{Q\}+ \varepsilon \pi ^1 + \varepsilon \rho^\varepsilon+ \varepsilon \pi ^\varepsilon_R,
\end{split}
\end{align}
where, for $f=f(t,x,z)$, we denote $\{f\}$ its evaluation at $z=\frac{\varphi(x)}{\sqrt{\varepsilon}}$, with some function $\varphi$ to be specifyd in Section \ref{BL1}.
Therefore, proving Theorem \ref{mainsc} reduces to constructing the terms in the right-hand side of \eqref{scales}, \eqref{scale} in an appropriate way.
Note that we will use an energy estimate to prove the smallness of $(l^\varepsilon_R, r^\varepsilon_R)$ as stated in Theorem \ref{mainsc}, which is the main reason for investigating not only the terms in the asymptotic expansion for the solid trajectory, but also the terms in the expansion for $u^\varepsilon$, which include certain boundary layer profiles $v$, $w$ and $\nabla\theta^\varepsilon$. 

Furthermore, let us emphasise that our whole construction will be done in order to have that the fluid and solid velocities $(u^\varepsilon_R,l^\varepsilon_R,r^\varepsilon_R)$ associated with the remainder satisfy a Navier-Stokes-type fluid-solid system with small initial data, some added small source term and small viscosity (``small'' with respect to $\varepsilon>0$), such that an appropriate energy estimate can be achieved.

Therefore, our strategy will be the following: 
\begin{itemize}
\item 
We construct $g^0$ and a smooth solution $u^0$ to the Euler equation (with control $g^0$), with zero initial data, hence zero vorticity and zero circulation around the solid, such that we have an exact controllability result for $(h^0,\vartheta^0)(T)$ with $(l^0,r^0)(T)=0$. However, we note that contrary to \cite{CMS}, our strategy will not rely on a return method for $u^0$, we will rather just use $u^0$ to control the solid position $(h^0,\vartheta^0)(T)$. See Theorem \ref{maineu0} in Section \ref{seu0}.
\item
Due to the Navier slip-with-friction boundary conditions, the fluid velocity  boundary layer $v$ will appear near the solid at order $\mathcal{O}(\sqrt{\varepsilon})$, together with its pressure $Q$ at order $\mathcal{O}(\varepsilon)$. Furthermore, at order $\mathcal{O}(\varepsilon)$ we introduce a boundary corrector $w$ and an inner domain corrector $\nabla\theta^\varepsilon$, together with its pressure $\rho^\varepsilon$, as done in \cite{CMS}. 
Note that there will be no contribution at order $\mathcal{O}(\sqrt{\varepsilon})$ in the solid equations due to the boundary layers, however, at order  $\mathcal{O}(\varepsilon)$ the solid position and velocity $(h^1,\vartheta^1,l^1,r^1)$, and therefore the fluid velocity $u^1$, will depend on $\varepsilon$ in a subtle manner due to $\rho^\varepsilon$. However, we mention that, for simplicity of notation, we will not write this dependence explicitly in the notations $u^1$, $h^1$, etc. 
Furthermore, we stress that since we do not control the fluid velocity $u^\varepsilon$, there is no need to control $v$ (contrary to \cite{CMS}), it will suffice to prove some regularity estimates to handle the effect of the boundary layers at $\mathcal{O}(\varepsilon)$ in the solid equations for $(l^1,r^1)$, and in the equations of the remainder.
Note that $v$, $w$ and $\nabla\theta^\varepsilon$ only depend on $(u^0,l^0,r^0)$ and their existence is immediate from the existence of $(u^0,\pi^0,l^0,r^0,g^0)$.
See Section \ref{BL} for some regularity estimates for these boundary layer profiles, which we will use in the energy estimate in Section \ref{remain}, but also to estimate the above-mentioned impact of $\rho^\varepsilon$ on $(l^1,r^1)$.
\item
We construct $g^1$ and a smooth solution $u^1$ to a linearized Euler equation around $u^0$ (with control $g^1$) such that we have an have an approximate controllability result for $(l^1,r^1)(T)$ (we settle for an approximate controllability here because it simplifies our construction). 
It would be natural to assume that the initial data of the original system, that is \eqref{ic}, would be the initial data in the equation of the linearized term $(u^1,l^1,r^1)$. 
However, we only have $u_0\in H^4$, and for simplifying reasons, we would like to work with  an initial data which has smooth and compactly supported curl, so that the vorticity associated with the linearized equation for $u^1$ stays smooth and compactly supported at all times.

This can be achieved with the following modification. We construct a family of divergence-free $u^*\in C^\infty(\overline{\mathcal{F}_0})\cap L^2(\mathcal{F}_0)$ which is bounded in $C^2$ with respect to $\varepsilon>0$ such that we have $\curl u^*\in C_0^\infty(\overline{\mathcal{F}_0})$, $u^*\cdot n= u_0\cdot n$ on $\partial\mathcal{S}_0$, $|u^*(x)|\to 0$ as $|x|\to+\infty$, 
and
\begin{align}\label{newid}
\| u_0-u^*\|_{2}\leq \varepsilon^{1/8}.
\end{align}
Indeed such a $u^*$ can be straightforwardly constructed using a Helmholtz decomposition and an appropriate mollification of $\curl u_0$.

We then set the initial data for $(u^1,l^1,r^1)$ to be $(u^*,h'_0,\vartheta'_0)$, and implicitly leave the remaining $u_0-u^*$ in the initial data for $u^\varepsilon_R$, namely we consider $(u^\varepsilon_R,l^\varepsilon_R,r^\varepsilon_R)(0)=(u_0-u^*,0,0)$. Note that $u^*$ implicitly depends once more on $\varepsilon>0$, and this gives rise to further dependence of $u^1$ with respect to $\varepsilon>0$, which we once again omit from the notations for the sake of simplicity. However, since 
we have that $u^*$ is bounded in $C^2$ uniformly with respect to $\varepsilon>0$,
 we expect this dependence to be slight enough
such that we can have some uniform estimates for $u^1$ with respect to $\varepsilon>0$.
See Theorem \ref{maineu1} in Section \ref{seu1}.
\item

We construct $(u^\varepsilon_R,l^\varepsilon_R,r^\varepsilon_R)$ as a weak solution (in the sense of Leray) of a system which we deduce from the equations verified by all the other terms in \eqref{scales}, \eqref{scale} (note that at this point we have not yet proven the existence of $(u^\varepsilon,l^\varepsilon,r^\varepsilon)$, but we know that it should be a very weak solution in the sense of Definition \ref{eweak} with $g^\varepsilon =g^0 + \varepsilon g^1$).
We prove by the means of an energy estimate that  $(u^\varepsilon_R,l^\varepsilon_R,r^\varepsilon_R)$ is small in $L^\infty((0,T);L^2(\mathcal{F}_0)\times \mathbb{R}^3)$, when $\varepsilon>0$ is small. In particular, we have $|(l^\varepsilon_R,r^\varepsilon_R)(T)| \leq C \varepsilon^{1/8},$ and we may conclude the estimates \eqref{mainsc1}  in Theorem \ref{mainsc}. See Proposition \ref{EEfinal} in Section \ref{remain}.
\end{itemize}

From the above construction we may define $(u^\varepsilon,l^\varepsilon,r^\varepsilon)$ as the right-hand sides of \eqref{scales}, \eqref{scale}, since now all the respective terms are constructed and well-defined. Furthermore, in order to ensure the continuity of the map \eqref{mainsc2}, we make sure that the terms on the right-hand side of \eqref{scales} at time $T$ are constructed continuously with respect to $(h_1,\vartheta_1,l_1,r_1)$ in the steps above. In particular it is sufficient to guarantee that $(l^0,r^0),(l^1,r^1),(l^\varepsilon_R,r^\varepsilon_R)\in L^\infty (0,T)$ depend continuously on $(h_1,\vartheta_1,l_1,r_1)$, which gives the continuity of the map \eqref{mainsc2} by using \eqref{frame2}.

However, as mentioned in Remark \ref{remsc}, $(u^\varepsilon,l^\varepsilon,r^\varepsilon)$ defined in such a way will only qualify as a very weak solution in the sense of Definition \ref{eweak} 
if \eqref{gczone} is verified as well, so we proceed in the following manner. 

We fix an open ball $B_c \subset \text{int}(\Omega_c) \setminus \displaystyle \bigcup_{(h,\vartheta)\in \overline{B}((h_f,\vartheta_f),\kappa)} \left\{(h+R(\vartheta)\mathcal{S}_0) \cup \mathcal{S}_0\right\}$ such that $d(B_c,\Omega_c)>0$.
During our construction we make sure that $g^0(t,\cdot)$ and $g^1(t,\cdot)$  are supported in $ R(\vartheta^0(t))^T (B_c -h^0(t))$ and that we have
\begin{align}\label{0czone}
B_c \cap (R(\vartheta^0(t))\mathcal{S}_0+h^0(t))=\emptyset,\ \forall t\in[0,T].
\end{align}
Since $(h^1,\vartheta^1)$ and $(h^\varepsilon_R,\vartheta^\varepsilon_R)$ are bounded in $L^\infty(0,T)$ by $C=C(\nu)>0$, there exists $\varepsilon_0=\varepsilon_0(\nu)>0$ such that \eqref{0czone} implies \eqref{gczone}, for any $\varepsilon\in(0,\varepsilon_0]$, with $g^\varepsilon(t,\cdot)=g^0(t,\cdot)+\varepsilon g^1(t,\cdot)$ supported in $ R(\vartheta^\varepsilon(t))^T (\Omega_c -h^\varepsilon(t))$, for any $t\in[0,T]$.

This concludes the proof of Theorem \ref{mainsc}.

\begin{remark}\label{plane}
Let us now explain why we chose to work in the whole plane $\mathbb{R}^2$ instead of a bounded domain. The key technical difficulty in handling the case of a bounded domain with a similar strategy would be the step of transforming the moving domain $\mathcal{F}(t)$ into a fixed domain. As mentioned above, in the case of the plane this can be done through a simple rigid movement. However, in the bounded case one would also need to account for the outer boundary $\partial\Omega$, and construct a diffeomorphism which is a rigid movement in a neighbourhood of the solid, but leaves the boundary $\partial\Omega$ intact. This diffeomorphism would clearly depend on the solid position, as well as contribute more complicated nonlinear terms in the PDE  (see for instance \cite{Bra} or \cite{GS-Uniq} for such a construction). The main problem then is to investigate what happens to these terms when we look for an asymptotic expansion of the form of \eqref{scales} for the solid trajectory, moreover separating them in terms of orders of $\varepsilon$. To properly do this, one would need to establish a rigorous asymptotic expansion of the diffeomorphism (and the associated terms in the PDE) with respect to the solid position, which is rather difficult.
\end{remark}

\subsection{Regarding Remark \ref{smallflux}}\label{flux}

In order to prove that the small flux condition mentioned in Remark \ref{smallflux} can indeed be achieved, we ensure during our construction that $g^0$ satisfies in addition the small flux condition
\begin{align}\label{g0flux}
\left\vert \int_0^T\displaystyle \int_{R(\vartheta^0(t))^T (B_c -h^0(t))} \, \left(g^0(t,x)\right)_{-}  \, \, dx\,  dt \right\vert <\frac{ \delta_c }{2} .
\end{align}
This can be achieved by similar arguments as in \cite{GKS} which we will detail in Remark \ref{g0smallflux} at the end of Section \ref{proof-approx}.

Consequently, we have
\begin{align*}
\left\vert \int_0^T\displaystyle \int_{R(\vartheta^\varepsilon(t))^T (\Omega_c -h^\varepsilon(t))} \, \left(g^\varepsilon(t,x)\right)_{-}  \, \, dx\,  dt \right\vert \leq
\left\vert \int_0^T\displaystyle \int_{R(\vartheta^0(t))^T (B_c -h^0(t))} \, \left(g^0(t,x)\right)_{-}  \, \, dx\,  dt \right\vert \\ + \varepsilon \left\vert \int_0^T\displaystyle \int_{R(\vartheta^0(t))^T (B_c -h^0(t))} \, \left(g^1(t,x)\right)_{-}  \, \, dx\,  dt \right\vert  < \frac{ \delta_c }{2}  + \varepsilon C ,
\end{align*}
where $C\in (0,+\infty)$ is independent of $\varepsilon>0$, 
therefore we may in fact further reduce the $\bar{\varepsilon}>0$ in the proof of Theorem \ref{maindiv} from Theorem \ref{maindiv} from Section \ref{SAE}, so that it satisfies $C \bar{\varepsilon}< \frac{ \delta_c }{2}$. This allows us to prove the small flux condition from Remark \ref{smallflux}.

%%%%%

\section{The inviscid term $u^0$}\label{seu0}

In this section we construct a controlled solution to the inviscid terms appearing in the asymptotic expansion \eqref{scales}, \eqref{scale}.

At order $\mathcal{O}(1)$, we look for $(u^0,\pi^0,l^0,r^0,g^0)$ satisfying the following system.
\begin{eqnarray}\label{eu0}
\begin{split}
\frac{\partial u^0}{\partial t}+(u^0&-u^0_S)\cdot\nabla u^0 + r^0 (u^0)^\perp+ \nabla \pi^0  =0 \   \text{ and } \ \div u^0 = g^0
\ \text{ for } x \in \mathcal{F}_0, \\
u^0\cdot n &= u^0_S \cdot n\ \text{ for } x \in \partial\mathcal{S}_0,\ \lim_{|x|\to+\infty}|u^0|=0,\\
m   (l^0)' &=  \int_{ \partial \mathcal{S}_0} \pi^{0} \, n \, d\sigma - m  r^{0}  (l^0)^ \perp \ \text{ and } \
\mathcal{J}  \ (r^{0})'  =    \int_{ \partial   \mathcal{S}_0} \pi^{0} \, x^{\perp}  \cdot n \, d \sigma, \end{split}
\end{eqnarray}
where $u^0_S (t,x)= l^0(t) + r^0(t) x^\perp$,
for  $t\in[0,T],$ with
$u^0(0,\cdot)=0, (l^0,r^0) (0) =0.$
The position of the solid is associated through the system
\begin{align}\label{frameeu0}
h^0(t)=\int_0^{t} R(\vartheta^0(s)) l^0(s) \, ds,\quad
\vartheta^0(t)=\int_0^{t} r^{0}(s) \, ds,
\end{align}
for $t\in[0,T].$

We introduce the following notation.
\begin{definition}
The space $C^k_\infty$, $k\geq0$ is defined as follows: for any bounded set $A\subset\mathbb{R}^2$, we define
$$C^k_\infty(\mathbb{R}^2\setminus A):=\left\{ f\in C^k(\mathbb{R}^2\setminus A)\text{ such that } \lim_{|x|\to+\infty} \left| \nabla^i f(x) \right| =0,\ \forall i\in\{0,\ldots,k\} \right\}.$$
Note that, for any $f\in C^k_\infty(\mathbb{R}^2\setminus A)$, we have $\displaystyle\|f\|_{C^k_\infty}=\sup_{x\in\mathbb{R}^2\setminus A} \max_{0\leq i \leq k} |\nabla^i f(x)|<+\infty$.
\end{definition}

In the sequel we will use some regularity results with respect to the position $q$ for certain integral terms, which only hold on a bounded set of admissible positions. Therefore, we consider $q_f$ and $\kappa>0$ as in Theorem \ref{mainsc}, and we 
pick some open ball $B\subset\mathbb{R}^3$ such that
$\overline{B}(q_f,\kappa)\subset \mathcal{Q}\cap B$, and
for $\delta>0$ we introduce the set
\begin{align}\label{qd}
Q_\delta =\{q=(h,\vartheta)\in B:\  d(h+R(\vartheta)\mathcal{S}_0,B_c)> \delta\}.
\end{align}
Consequently, as long as $q^0$ stays in $Q_\delta$, condition \eqref{0czone} will hold. For $\delta>0$ small enough, $\mathcal{Q}_\delta$ is clearly path-connected.

Furthermore, we will look for solutions $(l^0,r^0,u^0)$ satisfying the following additional condition, which will not be needed for the construction of $u^0$, however it will be helpful in the construction of $u^1$, which as already mentioned, is a linearized solution around $u^0$. We will therefore look for solutions satisfying 
\begin{align}\label{H01}
\text{span}\left\{(n(x), x^\perp \cdot n(x)),\ x\in\partial\mathcal{S}_0 \cap \text{supp}\left\{u^0(T/2,\cdot)-u^0_S(T/2,\cdot)\right\} \right\}=\mathbb{R}^3.
\end{align}

We have the following exact controllability result for $(h^0,\vartheta^0,l^0,r^0)$.

\begin{theorem} \label{maineu0}
Let  $T>0$, $\delta>0$ small enough such that $\mathcal{Q}_\delta$ is path-connected, $\mathcal{S}_0\subset\mathbb{R}^2$ bounded, closed, simply connected with smooth boundary, which is not a disk, and $q_0,q_1\in \mathcal{Q}_\delta$ with $q_0=0$ and $q_1 =(h_{1},\vartheta_1)$.
There exists a control $g^0\in C_0^{\infty}((0,T)\times\mathcal{F}_0)$ and a solution
 $(h^0,\vartheta^0,u^0)\in C^{\infty}([0,T];\mathcal{Q}_\delta)\times C^{\infty}([0,T]\times\overline{\mathcal{F}_0};\mathbb{R}^{2})$ 
 to  
 \eqref{eu0}, \eqref{frameeu0} with zero initial conditions for $(h^0,\vartheta^0,l^0,r^0,u^0)$, such that \eqref{H01} holds, $u^0\in C([0,T];L^2(\mathcal{F}_0))$ and
$$(h^0,\vartheta^0,l^0,r^0)(T)=(h_1,\vartheta_1,0,0)\text{, } \text{supp }g^0(t,\cdot) \subset R(\vartheta^0(t))^T (B_c -h^0(t)),\ \forall t\in [0,T].$$
Furthermore, one may define a continuous map $(h_1,\vartheta_1)\mapsto (l^0,r^0,u^0)\in C^3([0,T];\mathbb{R}^3\times C^5_\infty(\overline{\mathcal{F}_0}))$.\end{theorem}
The proof will be given in Section \ref{proof-approx}. Note that the specific regularity $C^3(C_\infty^5)$ will serve to establish appropriate higher order energy estimates for the boundary layer profiles constructed in Section \ref{BL}.

%%%%%%%%%%%%%%%%%%%%%%%%%%%%%%%%%%%%%%%%%%%%%%%%%%%%%%%%
%
%
%
%
\subsection{Reformulation of the solid's equation into an ODE}
\label{NODE}

In this section we establish a reformulation of the solid equations from \eqref{eu0} as an ODE for the three degrees of freedom of the rigid body with coefficients obtained by solving some elliptic-type problems.

To simplify notations, we denote the positions and velocities
$q^0=(h^0,\vartheta^0)$, $p^0=(l^0,r^0)$. Observe that a smooth solution $u^0$ of \eqref{eu0} satisfies the following div/curl type system:
\begin{equation}
\label{zozo}
  \left\{
      \begin{aligned}
&\div u^0 = g^0,\, \curl u^0= 0 \ \text{ in } \mathcal{F}_0 ,\\
& u^0 \cdot n =  \left( l^0+r^0 x^\perp \right)\cdot n\ \text{on}\ \partial \mathcal{S}_0, \lim_{|x|\to+\infty}|u^0|=0,\\
&\int_{\partial\mathcal{S}_0} u^0 \cdot \tau \, d\sigma =0,
\end{aligned} \right.
\end{equation}
for $t\in[0,T]$,
the last equation above coming from Kelvin's theorem regarding the conservation of the circulation around the body.

We observe that the unique smooth solution of the above system can be uniquely decomposed in a linear manner.
We introduce the Kirchhoff potentials, which in our case are simplified due to the fact that we have moved our evolution PDE onto a cylindrical domain.
Let 
 $\mathbf{\Phi} =(\Phi_1,\Phi_2,\Phi_3)\in C^\infty(\overline{\mathcal{F}_0})\cap C^6_{\infty}(\overline{\mathcal{F}_0})$
be the solution (up to a constant) of the elliptic problems
\begin{align} \label{kir}
\begin{split}
&\Delta\Phi_i(x)=0\ \text{in}\ \mathcal{F}_0, \ \lim_{|x|\to+\infty}|\nabla\Phi_i(x)|=0,\ \text{ for } i\in\{1,2,3\},\\
&\partial_{n}\Phi_i(x)=
\left\{ \begin{array}{l}
n_i\ \text{on}\ \partial\mathcal{S}_0,\text{ for }i\in\{1,2\},\\
x^\perp \cdot n\ \text{on}\ \partial\mathcal{S}_0,\text{ for }i=3.
\end{array} \right.
\end{split}
\end{align} 
Note that $\nabla \Phi(x)= \mathcal{O}(1/|x|^2)$ as $|x|\to+\infty$, implying that $\nabla \Phi$ is in fact square-integrable (see for instance Section 2.3 in \cite{GLS-Mass}).

Furthermore, we will also be looking for potential flows to handle the term in the decomposition of $u^0$ due to the control. In order to satisfy the condition $\text{supp }g^0(t,\cdot) \subset R(\vartheta^0(t))^T (B_c -h^0(t)), \text{ for all } t\in [0,T]$ in Theorem \ref{maineu0}, we
introduce, for  any $q=(h,\vartheta)\in\mathcal{Q}_\delta$ ,the set 
$$ \mathcal C(q) := \left\{  g \in    C_{0}^{\infty}( R(\vartheta)^T (B_c -h)  ;\mathbb{R}) \  \text{ such that }  \, \int g =0 \right\} ,$$
and
we only consider potential flows of the following type.

\begin{definition}\label{defA}
With  any $q=(h,\vartheta)\in\mathcal{Q}_\delta$ and $g \in \mathcal C(q)$ we associate the unique solution
$\overline{\alpha} := \mathcal A[q,g] \in C^\infty (\overline{\mathcal{F}_0};\mathbb{R})$ which vanishes at infinity  to the following elliptic problem:
\begin{equation}\label{pot}
\Delta \overline{\alpha} =g\mathbbm{1}_{R(\vartheta)^T (B_c -h)}(x)\ \text{in}\  \mathcal{F}_0,\  \lim_{|x|\to+\infty}|\nabla\overline{\alpha}|=0,\ \text{ and }
\partial_{n} \, \overline{\alpha}=0\ \text{on}\ \partial\mathcal{S}_0 .
\end{equation}
\end{definition}
Note that since $\overline{\alpha}$ is harmonic outside of a compact set and $\displaystyle\lim_{|x|\to+\infty}|\nabla\overline{\alpha}|=0$, in particular we also have $\overline{\alpha}\in C^6_\infty$, by using a Laurent series development to investigate its behaviour at infinity.
Furthermore, noting that $\int_{\partial\mathcal{S}_0} \nabla \overline{\alpha} \cdot \tau \ d\sigma=\int_{\partial\mathcal{S}_0} \nabla \overline{\alpha} \cdot n \ d\sigma=0$, we may in fact conclude as in Lemma A8 from \cite{GLS-Light} that $\nabla \overline{\alpha} (x)= \mathcal{O}(1/|x|^2)$ as $|x|\to+\infty$, implying that $\nabla \overline{\alpha}$ is in fact square-integrable.
Finally, we observe that the map $q\mapsto  \mathcal A[q,g] $ is smooth.

The following statement is an immediate consequence of the definitions above.
\begin{lemma}\label{ldecomp0}
For any $q=(h,\vartheta)$ in $\mathcal{Q}_\delta$,
for any $p=(l,r)$ in $\mathbb R^2 \times \mathbb R$, and $g \in \mathcal C(q)$, the unique solution $u$ in $C^\infty ( \overline{\mathcal{F}_0})$ to the following system:
\begin{equation}
\label{zozoFormal}
  \left\{
      \begin{aligned}
&\div u = g\mathbbm{1}_{R(\vartheta)^T (B_c -h)},\ \, \curl u= 0 \ \text{ in } \mathcal{F}_0 ,\\
& u \cdot n =  \left(l+ r x^\perp \right)\cdot n\ \text{on}\ \partial \mathcal{S}_0,\  \lim_{|x|\to+\infty}|u|=0,\\\
&\int_{\partial\mathcal{S}_0} u \cdot \tau \, d\sigma =0,
\end{aligned} \right.
\end{equation}
 is given by the following formula, for $x$ in $\overline{\mathcal{F}_0}$, 
\begin{equation}
  \label{praud}
u(x)=\nabla (p \cdot \Phi (x))+\nabla  \mathcal A[q,g](x).
\end{equation}
\end{lemma}
Above $p \cdot \Phi  (x)$ denotes the inner product $p \cdot \Phi (x) = \sum_{i=1}^3 \, p_i  \Phi_i  (x)$.
\bigbreak
Let us now address the solid dynamics. We aim for a reformulation as in \cite{GKS}, however due to the fact that we are now in a domain which does not depend on $q^0$, there will be terms that become simplified.
We introduce the following notations.
\begin{definition}\label{christ0}
 We respectively define the genuine and added mass $3 \times 3$ matrices by
$$\mathcal{M}_{g}= \left( \begin{array}{ccc}
m & 0 & 0\\
0 & m & 0\\
0 & 0 & \mathcal{J}
\end{array} \right) ,$$
and,
$$ \mathcal{M}_{a}  =\left(  \int_{  \mathcal{F}_0}   \nabla\Phi_i (x) \cdot \nabla\Phi_j (x) \, dx \right)_{1 \leqslant i,j  \leqslant 3} ,
$$
and we denote their sum by $\mathcal{M}$. 

We define the symmetric bilinear map $\Gamma $ for any $p=(l,r),\tilde{p}=(\tilde{l},\tilde{r})\in\mathbb{R}^3$ by
$$\langle \Gamma, p,\tilde{p} \rangle=\frac{1}{2}\int_{\partial\mathcal{S}_0} (\nabla (p\cdot\Phi (x)))\cdot (\nabla (\tilde{p}\cdot\Phi (x))) \partial_n \Phi(x) \, d\sigma+\langle\tilde{\Gamma}^{\text{sym}},p,\tilde{p}\rangle \in \mathbb{R}^3,$$
where $\tilde{\Gamma}^{\text{sym}}$ denotes the symmetric part of the bilinear map $\tilde{\Gamma}$ defined by
$$\langle\tilde{\Gamma},p,\tilde{p}\rangle=-\int_{\partial\mathcal{S}_0} (l+r x^\perp) \cdot \nabla (\tilde{p}\cdot\Phi (x)) \partial_n \Phi(x) \, d\sigma + (m\tilde{r} l^\perp,0) \in \mathbb{R}^3.$$
Note that $\tilde{\Gamma}$ is no longer symmetric, however, we have $2\langle\tilde{\Gamma}^{\text{sym}},p,\tilde{p}\rangle=\langle\tilde{\Gamma},p,\tilde{p}\rangle+\langle\tilde{\Gamma},\tilde{p},p\rangle$, for all $p,\tilde{p}\in\mathbb{R}^3$.
\end{definition}

Let us first give the reformulation of the model as an ODE when there is no control.

\begin{proposition}\label{reform}
Given 
$p^0=(l^0,r^0) \in C^{\infty}([0,T];\mathbb{R}^3)$, $u^0\in C^{\infty}([0,T]\times\overline{\mathcal{F}_0};\mathbb{R}^{2})$, 
 we have that $(u^0,p^0)$  is a solution to 
 (\ref{eu0})  with $g^0=0$ and zero initial conditions if and only if $p^0$ satisfies the following ODE on $[0,T]$
 \begin{align} \label{tout}
\mathcal{M}(p^0)'+\langle \Gamma,p^0,p^0\rangle =  0
\end{align}
and $u^0$ is the unique smooth solution to 
the system (\ref{zozo}) with $g^0=0$. 
\end{proposition}
Observe that the position of the solid plays no role in this case.
\begin{proof}
The proof is straightforward based on the above definitions,
and due to the fact that \eqref{eu0} is on a fixed domain, contrary to the case of \cite{GKS} or \cite{GMS}.

We recall Lamb's form: for any differentiable functions $v_1,v_2$ defined on a subset of $\mathbb{R}^2$ with values in $\mathbb{R}^2$ we have
\begin{align}\label{lamb}
\nabla (v_1\cdot v_2)=v_1 \cdot \nabla v_2 + v_2 \cdot \nabla v_1 - \curl v_1 (v_2)^\perp - \curl v_2 (v_1)^\perp,
\end{align}
and we use it to obtain
that the gradient of the pressure $\pi^0$ in \eqref{eu0} with $g^0=0$ can be expressed as
$\nabla\pi^0=-\partial_t u^0 - \frac{1}{2}\nabla( |u^0|^2 ) +\nabla( u^0 \cdot u^0_S).$ Note that the solid equations can be rewritten as
\begin{align*}
\mathcal{M}(p^0)'-(m r^0(l^0)^\perp,0)&=\int_{\partial\mathcal{S}_0} \pi \, \partial_n \Phi \, d\sigma = \int_{\mathcal{F}_0} \nabla\pi \cdot \nabla\Phi \, dx=\int_{\mathcal{F}_0} \left(-\partial_t u^0 - \frac{1}{2}\nabla( |u^0|^2 ) +\nabla( u^0 \cdot u^0_S) \right)\cdot \nabla\Phi \, dx,
\end{align*}
since even though $u^0_S$ grows like $x$ at infinity, the integration by parts above is justified since $\nabla\Phi$ (and implicitly $u^0$) behaves like $1/|x|^2$ as $|x|\to+\infty$. We may conclude by using Lemma \ref{ldecomp0}, integrating once more by parts and rearranging the appropriate terms to get \eqref{tout}.
\end{proof}

Now we may move to the case with control. We introduce the following force terms.
\begin{definition}
\label{def-forces}
We define, for any $p=(l,r)$ in $\mathbb{R}^3$, $\alpha $ in $C^{\infty}(\overline{\mathcal{F}_0}; \mathbb R)$, 
$F_1  (p) [ \alpha]  $ and $F_2 [\alpha]   $ in $\mathbb R^3$ by $F_1(p)[\alpha]= F_{1,a} [ \alpha]  +F_{1,b} (p) [ \alpha]$, where
\begin{align} \label{ef2}
\begin{split}
 F_{1,a} [ \alpha]  =& -  \frac12 \int_{ \partial \mathcal{S} _0}  |\nabla\alpha(x)|^{2}   \, \partial_n \Phi (x) \, d \sigma,
\\  F_{1,b} (p) [ \alpha]  =&-\int_{ \partial \mathcal{S} _0}  \nabla\alpha(x) \cdot \nabla (p\cdot \Phi(x))\,  \partial_n \Phi (x)d \sigma
+\int_{\partial\mathcal{S}_0} (l+r x^\perp) \cdot \nabla \alpha(x)\partial_n \Phi(x) \, d\sigma,
\\   F_2   [\alpha]  
=& - \int_{ \partial \mathcal{S}_0}   \alpha(x)  \, \partial_n \Phi (x) \, d \sigma .
\end{split}
\end{align}
\end{definition}

 Observe that Formulas \eqref{ef2}  only require  $\alpha $ and $\nabla \alpha $ to be defined on $\partial \mathcal{S}_0$. 
 Moreover when these formulas  are applied to 
$\alpha = \mathcal A[q,g] $  for some $g $ in $\mathcal C$, 
then only the trace of $\alpha$ and the tangential derivative $\partial_\tau \alpha$ on $\partial \mathcal{S}_0$
are involved, since the normal derivative of $\alpha$ vanishes on $\partial \mathcal{S}_0$ by definition, cf. \eqref{pot}. 

We define our notion of controlled solution of the ``fluid+solid'' system as follows. 
\begin{definition}\label{CS}
We say that $(q^0,p^0,g^0)$ in $C^{\infty} ([0,T];\mathcal{Q}_\delta\times \mathbb{R}^3) \times C^{\infty}_0 ((0,T); \mathcal C(q^0(t)))$ is a controlled solution associated with  \eqref{eu0}, \eqref{frameeu0}, if the following ODE holds true on $[0,T]$:
  \begin{align}\label{renform}
  \begin{split}
\mathcal{M}(p^0)'+\langle \Gamma,p^0,p^0\rangle &=  F_1 (p) [ \alpha^0] +   F_2   [\partial_t\alpha^0] ,\\
(q^0)'&=\mathcal{R}(q^0) p^0,
\end{split}
\end{align}
 where $\alpha^0(t,\cdot) := \mathcal A[q^0(t),g^0(t,\cdot)] $ and 
 $$\mathcal{R} (q^0) = \mathcal{R} (\vartheta^0) 
:= \left( \begin{array}{ccc}
R(\vartheta^0)  & 0\\
0 & 1 
\end{array} \right) . $$
\end{definition}

We have the following result for reformulating the model as an ODE.
\begin{proposition} \label{reformC}
Given 
$(q^0,p^0)\in C^{\infty}([0,T];\mathcal{Q}_\delta \times\mathbb{R}^3)  ,\ u^0\in C^{\infty}([0,T]\times\overline{\mathcal{F}_0};\mathbb{R}^{2})  $ and $  g^0\in  C^{\infty}_0 ((0,T); \mathcal C(q^0(t)))  ,$
 we have that $(u^0,q^0,p^0,g^0)$
 is a solution to \eqref{eu0},  \eqref{frameeu0},  if and only if $(q^0,p^0,g^0)$
 is a controlled solution
 and $$u^0(t,x)=\nabla (p^0(t) \cdot \Phi (x))+\nabla  \mathcal A[q^0(t),g^0(t,x)](x),$$
 for any $(t,x)\in[0,T]\times \overline{\mathcal{F}_0}$, with $\Phi$ and $\mathcal{A}$ given in \eqref{kir}, respectively \eqref{pot}.
 \end{proposition} 
The proof is a straightforward adaptation of the proof of Proposition \ref{reform}, noting that the regularity of the functions involved allows us to perform all the integration by parts, and is therefore omitted.

%%%%%%%%%%%%%%%%%%%%%%%%%%
%%%%%%%%%%%%%%%%%%%%%%%%%%%%%%%%%%%%%%%%%%%%%%%%%%%%%%%%%%%%%%%%%%%%%%%%%%%%%%%%%%%%%%%%%%%%%%%%%%%%%%%%

%
\subsection{Proof of the exact controllability result Theorem~\ref{maineu0}}
\label{proof-approx}
In this section we observe that it is possible to prove Theorem \ref{maineu0} by  exploiting the geodesic feature of the uncontrolled system and an impulsive control strategy, as done in \cite{GKS}. We will skip certain straightforward parts of the proof and refer the reader to the aforementioned article for more details. We will instead focus on presenting the strategy and highlighting the main differences between our case and the bounded case in \cite{GKS}.

We observe that Theorem  \ref{maineu0} can be deduced from a simpler approximative controllability result, namely Theorem \ref{approxeuzero}, where  the solid displacement is assumed to be small. 
The proof of passing from small solid displacement to an arbitrary one is a straightforward adaptation of Section 4 in \cite{GKS}, therefore will be omitted.

As mentioned in the proof of Theorem \ref{maindiv} in this article, respectively in Section 4 of \cite{GKS}, it is possible to pass from approximate controllability of the final position and velocity to exact controllability using a topological argument (see Lemma \ref{top}). In this section however we will proof that this argument could be replaced by a local inversion argument, both in our case and in the case of \cite{GKS}.
For simplicity, we will present the method only in the case of $(u^0,q^0,p^0)$, where the fluid is irrotational and the circulation around the body is zero, however, it can be extended to the general case through some straightforward but technical modifications. 

We have the approximate controllability result below, which can be seen as a generalization of Theorem 5 from \cite{GKS}.
\begin{theorem} \label{approxeuzero}
Consider $\delta>0$, $\mathcal{S}_0\subset\mathbb{R}^2$ bounded, closed, simply connected with smooth boundary,
which is not a disk, $q_0$ in $\mathcal{Q}_\delta$ and $T>0$. Then
there exists $\tilde{r}>0$ such that $B(q_0 , \tilde{r}) \subset \mathcal{Q}_\delta$, and 
for any $\nu >0$, there exists a mapping
$$
\mathcal{T}: \overline{B}\big( (q_0 , 0), \tilde{r} \big) \rightarrow C^\infty ([0,T]; \mathcal{Q}_\delta\times\mathbb{R}^3)
$$
which with $ ({q}_1 , {q}'_1) $ associates $(q^0,p^0)$ where $(q^0,p^0,g^0) $ is a controlled solution associated with \eqref{eu0},  \eqref{frameeu0}, and the initial data $(q_0 , 0)$,
such that we have the following:
\begin{itemize}
\item
\eqref{H01} holds and
 the map $q_1\mapsto (l^0,r^0,u^0)\in C^3([0,T];\mathbb{R}^3)\times C^3([0,T];C^5_\infty(\overline{\mathcal{F}_0}))$ is continuous, where $u^0$ is the associated fluid velocity  given by Proposition \ref{reformC};
 \item
 the 
 mapping
$$({q}_1 , {q}'_1)  \in \overline{B}\big( (q_0 , 0), \tilde{r} \big) \mapsto \mathcal{T}({q}_1 , {q}'_1 )  (T)   \in \mathcal{Q}_\delta \times \R^3 $$
is $C^1$, and
 for any $ w $ in $ \overline{B}\big( (q_0 , 0), \tilde{r} \big)$, we have
\begin{equation}\label{difft}
\frac{\partial}{\partial w} \mathcal{T}(w)(T)=\text{Id}+\mathcal{O}(\nu).
\end{equation}
\end{itemize}
\end{theorem}

Using this result, we may prove Theorem \ref{maineu0} in the case of small solid displacement. Indeed, we set $\mathbb{T}(w)=\mathcal{T}(w)(T)$, for any $ w $ in $ \overline{B}\big( (q_0 , 0), \tilde{r} \big)$.
Taking $\nu$ sufficiently small, we see that $\frac{\partial \mathbb{T}}{\partial w }(w)$ is invertible.
Consequently one can use the inverse function theorem on $\mathbb{T}$: there exists $r>0$ small enough such that if $q_1\in\overline{B}(q_0,r)$, then $\mathbb{T}$ is invertible at $w=(q_1,0)$. Therefore, $\mathcal{T}(\mathbb{T}^{-1}(q_1,0))$ is a trajectory associated with a controlled solution which at time $T>0$ takes exactly the value $(q_1,0)$. This concludes the proof of Theorem \ref{maineu0}.

\subsubsection{Proof of Theorem \ref{approxeuzero} without \eqref{H01}}\label{step1t5}

For simplicity, we will organise the proof of Theorem \ref{approxeuzero} into two steps. In this subsection we first give our construction without trying to satisfy \eqref{H01}. In Section \ref{th5h}, we show how our construction can be straightforwardly modified to ensure that \eqref{H01} holds, without interfering with the conclusions presented in this subsection.

\paragraph{The choice of the form of the control - constructing the operator $\mathcal{T}$}

Let us first present our impulsive control strategy based on the intuition that we want to exploit the underlying geodesic structure of \eqref{renform}. See \cite{Br} and the references 
therein for examples of impulsive control strategies.

Let $ w=(q_1,q'_1)\in\overline{B}\big( (q_0 , 0), \tilde{r} \big)$.
We consider controls of the form
\begin{equation} \label{gcontrol}
g^0_{\eta} (t,x) :=  \beta_{\eta}(t) g_0 (x)  +  \beta_{\eta} (T-t) g_1(x) ,
\end{equation}
where $ \beta_{\eta}\geq 0$, $( \beta_{\eta}^2 )_\eta $  is compactly supported in $(0,2\eta)$ and is an approximation of the unity when $\eta \to 0^{+}$. Furthermore, we denote by $(q^0,p^0)$ the solution of \eqref{renform} with control $g^0_\eta$, dropping the dependence with respect to $\eta>0$ from the solid trajectory for simplicity of notation.

It can be proven that, for a right choice of $g_0,g_1$, the trajectory $(q^0,p^0)$ will be close to the solution  $(\tilde{q},\tilde{p})$ of the following toy system:
\begin{align}\label{toy}
\begin{split}
 \mathcal{M} \tilde{p}'+\langle\Gamma, \tilde{p}, \tilde{p}\rangle =   \beta_{\eta}^2 (\cdot)\,  v_0  +  \beta_{\eta}^2 (T-\cdot)  v_1,\
\tilde{q}'=\mathcal{R}(\tilde{q})\tilde{p},
\end{split}
\end{align}
with zero initial conditions and some $v_0,v_1\in\mathbb{R}^3$ given. 

On the other hand, we claim that there exists a local solution in $\mathcal{Q}_\delta$ to the ODE system  
  \begin{align}\label{geo}
  \begin{split}
\mathcal{M}(\bar{p})'+\langle \Gamma,\bar{p},\bar{p}\rangle = 0 ,\
(\bar{q})'=\mathcal{R}(\bar{q}) \bar{p},\ \text{on}\ [0,T],\ \text{with}\ \bar{q}(0)=q_0,\ \bar{q}(T)=q_1,
\end{split}
\end{align}
where the map $q_1\in\overline{B}\big( q_0, \tilde{r} \big)\mapsto (c_0,c_1)\in\mathbb{R}^6$ given by $c_0(q_1)=\bar{p}(0)$, $c_1(q_1)=\bar{p}(T)$, is $C^1$. Note that, contrary to \cite{GKS}, $\mathcal{M}$ and $\Gamma$ do not depend on $\bar{q}$ here, so proving the existence of $(\bar{q},\bar{p})$ at first glance does not seem to follow as in \cite{GKS}, at least not directly.
 
However, a simple way to prove the existence of a solution to \eqref{geo} is to first introduce a change of coordinates which corresponds to changing back from a fixed domain to a domain which moves according to the trajectory $(\bar{q},\bar{p})$. Namely, setting
\begin{align}
\bar{\tau}(t,x)=R(\bar{\vartheta}(t))^T(x-\bar{h}(t)),\ \bar{u}(t,x)=R(\bar{\vartheta}(t)) \nabla\left( \bar{p}(t)\cdot \Phi(\bar{\tau}(t,x))\right),
 \end{align}
 allows us to switch back to the setting where the solid position (and implicitly the fluid domain) evolves and the control zone stays fixed in time. More precisely, $(\bar{u},\bar{q})$ will satisfy an ``inviscid fluid + rigid body'' system as in \cite{GKS}, but in the whole plane, without any control, with zero initial data for $\bar{u}$, and the endpoints of $\bar{q}$ fixed as $q_0$ and $q_1$. For a geodesic reformulation of the equation of $\bar{q}$ in this setting, see Section 4.1 in \cite{GMS}, and note that the existence of $\bar{q}$ then follows in the same way as the proof of Lemma 4 from \cite{GKS}.

Setting 
\begin{align}\label{vi}
v_i=(-1)^i\mathcal{M}(c_i(q_1)-q'_i),\ i=0,1,
\end{align} 
we obtain, similarly as in Section 5.2 of \cite{GKS}, that the solution $(\tilde{q},\tilde{p})$ of \eqref{toy} satisfies 
\begin{align}\label{tglim}
\lim_{\eta\to 0^+}\| (\tilde{q},\tilde{p})-(\bar{q},\bar{p})\|_{C([2\eta,T-2\eta])}=0
\end{align}
 and $(\tilde{q},\tilde{p})(T)\to (q_1,q'_1)$ as $\eta\to 0$.

The construction of the controls $g_0,g_1$ relies on the observation that when we approximate the solution $(q^0,p^0)$ with $(\tilde{q},\tilde{p})$, the term from \eqref{renform} which will behave like $\beta_\eta^2\, v_i$ is $F_{1,a}$, and we use a complex analysis argument to prove that it can attain any direction $v\in\mathbb{R}^3$.
Since we are in the unbounded case, the construction of potential flows needs to be adapted in order to guarantee that the flow velocity is in $ C^5_{\infty}(\overline{\mathcal{F}_0})$ for instance. 

We consider $\tilde{B}_c\subset B_c$ closed such that $d(\tilde{B}_c,\partial B_c)>0$ and set
\begin{align}\label{tilmc}
 \tilde{\mathcal{C}}(q) := \left\{  g \in    C_{0}^{1}( R(\vartheta)^T(\tilde{B}_c-h)  ;\mathbb{R}) \  \text{ such that }  \, \int g =0 \right\},
 \end{align}
for any $q\in\mathcal{Q}_\delta$. Note that the purpose of introducing  $\tilde{\mathcal{C}}(q)$ is twofold: 
on one hand it allows us to construct our controls more robustly with respect to $q\in\mathcal{Q}_\delta$, since if $g\in\tilde{\mathcal{C}}(q)$ and $q$ is close to $\hat{q}$, it follows that $g\in \mathcal{C}(\hat{q})$, in particular this will be useful to prove that the support of the control we construct will stay sufficiently far away from $\mathcal{S}_0$; and on the other hand we will construct a $C^1$ map with respect to $g$ below, which is easier to accomplish on the Banach space  $\tilde{\mathcal{C}}(q)$.
Furthermore, we will extend $g\in\tilde{\mathcal{C}}(q)$ by $0$ in $B_c\setminus \tilde{B}_c$.

We have the following generalization of Proposition 2 from \cite{GKS}, which will be proved in Appendix \ref{comproof1}. Recall that $\mathcal{A}$ was defined in Definition \ref{defA}.
\begin{proposition} \label{farcontr0}
There exists a $C^1$ mapping $\overline{g}^0:\mathcal{Q}_\delta \times \mathbb R^3 \rightarrow \tilde{\mathcal{C}}(\mathcal{Q}_\delta)$ 
such that  for any $q\in\mathcal{Q}_\delta$ we have $\text{Range}(\overline{g}^0(q,\cdot))\subset \mathcal C(q) \cap \tilde{\mathcal{C}}(q)$, and for any $(q,v) $ in $\mathcal{Q}_\delta \times \mathbb R^3$ 
the function $\overline{\alpha}^0 := \mathcal A [q,\overline{g}^0 (q,v)]$ in $C^\infty (\overline{\mathcal{F}_0};\mathbb{R})\cap C^6_{\infty}(\overline{\mathcal{F}_0};\mathbb{R})$  satisfies:
\begin{gather}
\Delta \overline{\alpha}^0 =0\ \text{in}\  \mathcal{F}_0 \setminus R(\vartheta)^T (\tilde{B}_c -h) , \lim_{|x|\to+\infty}|\nabla\overline{\alpha}^0|=0  \text{ and } 
\partial_{n} \, \overline{\alpha}^0=0\ \text{on}\ \partial\mathcal{S}_0 ,\\
\int_{ \partial \mathcal{S}_0} \left|  \nabla\overline{\alpha}^0\right|^{2} \, \partial_n \Phi \, d \sigma = v , \\
\int_{\partial\mathcal{S}_0} \overline{\alpha}^0 \, \partial_n \Phi \, d\sigma = 0 ,\\
\text{span}\left\{(n(x),x^\perp \cdot n(x)),\ x\in\text{supp } \nabla\overline{\alpha}^0(q,\cdot) \cap \partial\mathcal{S}_0\right\} = \mathbb{R}^3. \label{fcsp}
\end{gather}
\end{proposition}
Using Proposition \ref{farcontr0}, we set $$g_0=\bar{g}^0(q_0,-2v_0),\quad g_1=\bar{g}^0(q_1,-2v_1)$$ in \eqref{gcontrol}, with $v_i$ as in \eqref{vi}.

It follows that for $\eta>0$ small enough, since the solid displacement during the control phase $[0,2\eta]\cup[T-2\eta,T]$ is small enough, the interior control $g^0_\eta$ given by \eqref{gcontrol} is truly supported inside the control zone $R(\vartheta^0(\cdot))^T (B_c -h^0(\cdot))$.

With the choice of control presented above, we set $\mathcal{T}(w):=(q^0,p^0)$ and $\tilde{\mathcal{T}}:=(\tilde{q},\tilde{p})$ the solutions of \eqref{renform}, respectively \eqref{toy}. We will prove that given any $\nu>0$, for $\eta>0$ small enough, the following hold
\begin{itemize}
\item[(i)] The toy model satisfies a similar condition to \eqref{difft} which we want to prove for $\mathcal{T}$, namely
\begin{equation}\label{tdifft}
\frac{\partial}{\partial w} \tilde{\mathcal{T}}(w)(T)=\text{Id}+\mathcal{O}(\nu).
\end{equation}
\item[(ii)] The differentials with respect to $w\in\overline{B}\big( (q_0 , 0), \tilde{r} \big)$ of the toy model and the real trajectory are close in $C([0,T];\mathbb{R}^{6\times6})$, uniformly with respect to $w\in\overline{B}\big( (q_0 , 0), \tilde{r} \big)$, namely we have
\begin{align}\label{trkul}
\left\|\frac{\partial}{\partial w} \mathcal{T}(w)(\cdot)-\frac{\partial}{\partial w} \tilde{\mathcal{T}}(w)(\cdot)\right\|_{C([0,T])} \leq \nu,\ \forall w\in\overline{B}\big( (q_0 , 0), \tilde{r} \big).
\end{align}
\end{itemize}
The result of Theorem \ref{approxeuzero} without \eqref{H01} clearly follows from (i) and (ii), by fixing $\eta=\eta_0>0$ small enough, and by observing that once $\eta$ is fixed, we may easily deduce the continuity result with respect to $q_1$ for $(p^0,u^0)$ mentioned in Theorem \ref{approxeuzero}. Indeed, it follows from \eqref{kir}, \eqref{pot}, \eqref{renform} and Proposition \ref{reformC} that the only issue which needs further investigation in order to establish such a continuity result is what happens when we differentiate $  \mathcal A[q^0(t),g^0_\eta(t,x)](x)$ with respect to $t$. Using \eqref{pot}, the fact that $g^0_\eta$ defined in \eqref{gcontrol} does not depend on $q^0$, we obtain that
\begin{align*}
\partial_t A[q^0(t),g^0_\eta(t,x)](x) &= \partial_q A[q^0(t),g^0_\eta(t,x)](x) \cdot (q^0)'(t) +A[q^0(t),\partial_t g^0_\eta(t,x)](x)=A[q^0(t),\partial_t g^0_\eta(t,x)](x).
\end{align*}
From here it is straightforward to conclude that the map
$q_1\mapsto (l^0,r^0,u^0)\in C^3([0,T];\mathbb{R}^3)\times C^3([0,T];C^5_\infty(\overline{\mathcal{F}_0}))$ is continuous.

\paragraph{Differentiating the toy model with respect to $w\in\overline{B}\big( (q_0 , 0), \tilde{r} \big)$} 

We will prove the following Lemma, which exactly implies \eqref{tdifft}.
\begin{lemma}
Let $w=(q_1,q'_1)\in\overline{B}\big( (q_0 , 0), \tilde{r} \big)$, and consider the solutions  $(\tilde{q},\tilde{p})$ of \eqref{toy} with $v_i=(-1)^i\mathcal{M}(c_i(q_1)-q'_i),\ i=0,1$, respectively $(\bar{q},\bar{p})$ of \eqref{geo}. Then $(\tilde{q},\tilde{p})$ and $(\bar{q},\bar{p})$ are differentiable with respect to $w$, and
we have that
\begin{align}\label{tglimd}
\lim_{\eta\to 0^+}\| (\partial_w\tilde{q},\partial_w\tilde{p})-(\partial_w\bar{q},\partial_w\bar{p})\|_{C([2\eta,T-2\eta])}=0
\end{align}
and 
$(\partial_w \tilde{q},\partial_w\tilde{p})(T)\to \text{Id}$ as $\eta\to 0^+$.
\end{lemma}

\begin{proof}
The existence of the differentials of  $(\tilde{q},\tilde{p})$ and $(\bar{q},\bar{p})$ with respect to $w=(q_1,q'_1)\in\overline{B}\big( (q_0 , 0), \tilde{r} \big)$ is immediate. Furthermore, we may formally differentiate \eqref{toy} and \eqref{geo} to find that the aforementioned differentials satisfy the following ODE systems:
\begin{align}\label{toyder}
  \begin{split}
\mathcal{M}(\partial_w \tilde{p})'&+2 \left( \langle \Gamma,\tilde{p},\partial_{w_i} \tilde{p}\rangle  \right)_{1\leq i\leq 6} = \beta^2_\eta \partial_w v_0 + \tilde{\beta}^2_\eta \partial_w v_1,\\
(\partial_w \tilde{q})'&=\left(\left(D\mathcal{R}(\tilde{q}) \cdot \partial_{w_i} \tilde{q}\right) \tilde{p} \right)_{1\leq i\leq 6} + \mathcal{R}(\tilde{q}) \partial_w \tilde{p},
\end{split}
\end{align}
on $[0,T]$,
with $\tilde{\beta}_\eta=\beta_\eta(T-\cdot)$ and $(\partial_w \tilde{q},\partial_w\tilde{p})(0)=0$; respectively,   \begin{align}\label{geow}
  \begin{split}
\mathcal{M}(\partial_{w}\bar{p})'+2 \left( \langle \Gamma,\bar{p},\partial_{w_i}\bar{p}\rangle  \right)_{1\leq i\leq 6} = 0 ,\
(\partial_{w}\bar{q})'=\left(\left(D\mathcal{R}(\bar{q})\cdot \partial_{w_i} \bar{q} \right)\bar{p} \right)_{1\leq i\leq 6} +\mathcal{R}(\bar{q}) \partial_{w} \bar{p},
\end{split}
\end{align}
$\text{on}\ [0,T],\ \text{with}\ \partial_{q_1}\bar{q}(0)=0,\ \partial_{q_1}\bar{q}(T)=\text{Id},\ \partial_{q'_1}\bar{q}(0)=0,\ \partial_{q'_1}\bar{q}(T)=0$. In particular note that due to $ \partial_{q'_1}\bar{q}(0)=0,\ \partial_{q'_1}\bar{q}(T)=0$, we in fact have that $ \partial_{q'_1}\bar{q}$ and $ \partial_{q'_1}\bar{p}$ vanish on $[0,T]$.
It further follows from \eqref{geo} that $\partial_w \bar{p}(0)= \partial_w c_0(q_1) = (\partial_{q_1} c_0(q_1), 0)$ and 
$\partial_w \bar{p}(T)= \partial_w c_1(q_1) = (\partial_{q_1} c_1(q_1), 0)$, which both depend continuously on $w$.

On the other hand, we may further develop $\partial_w v_i=(\partial_{q_1} v_i,\partial_{q'_1} v_i)$ that appear in \eqref{toyder}, recalling the forms of $v_i$ in the statement of the Lemma. We have
\begin{align}\label{cder}
\begin{split}
\partial_{q_1} v_i=\partial_{q_1}\left((-1)^i\mathcal{M}(c_i(q_1)-q'_i\right)&=(-1)^i\mathcal{M}\partial_{q_1}c_i(q_1),\\
\partial_{q'_1} v_i=\partial_{q'_1}\left((-1)^i\mathcal{M}(c_i(q_1)-q'_i\right)&= i \mathcal{M}.
\end{split}
\end{align}

Before we proceed to comparing the trajectories $ (\partial_w\tilde{q},\partial_w\tilde{p})$ and $(\partial_w\bar{q},\partial_w\bar{p})$, let us explain how one can prove that they are 
uniformly bounded in $L^\infty(0,T)$ with respect to $\eta\in(0,1)$ and $w\in\overline{B}\big( (q_0 , 0), \tilde{r} \big)$.
Through similar methods as in Section 5 of \cite{GKS}, it can be checked that $\|\tilde{p}\|_\infty$ is uniformly bounded with respect to $(\eta,w)$, respectively $\|\bar{p}\|_\infty$ is uniformly bounded with respect to $w$. From \eqref{toyder} we get that
\begin{align*}
|\partial_w \tilde{p}(t)| \leq C(| \partial_w v_0| +| \partial_w v_1|)+ C\|\tilde p\|_\infty \int_0^t |\partial_w \tilde{p}(s)| \, ds ,\\
|\partial_w \tilde{q}(t)| \leq  C\|\tilde p\|_\infty \int_0^t |\partial_w \tilde{q}(s)| \, ds +C \int_0^t |\partial_w \tilde{p}(s)| \, ds,
\end{align*} for $t\in[0,T]$,
with $C>0$ independent of $(\eta,w)$.
Using a Gronwall estimate and \eqref{cder}, we get that $\|\partial_w\tilde{p}\|_\infty$ is uniformly bounded with respect to $(\eta,w)$. Similarly one can deduce from  \eqref{geow}
that $\|\partial_w \bar{p}\|_\infty$ is uniformly bounded with respect to $w$.  

Now we are ready to compare the trajectories $ (\partial_w\tilde{q},\partial_w\tilde{p})$ and $(\partial_w\bar{q},\partial_w\bar{p})$, which we will do on three time intervals, namely $[0,2\eta]$, $[2\eta,T-2\eta]$ and $[T-2\eta,T]$, in order to exploit the supports of the functions $\beta_\eta$ and $\tilde{\beta}_\eta$.

Integrating \eqref{toyder} on $[0,2\eta]$ and taking into account \eqref{cder}, we get that
\begin{align*}
\mathcal{M} \partial_w \tilde{p}(2\eta) =\mathcal{M} \partial_{q_1}c_0(q_1) - 2 \int_0^{2\eta} \left( \langle \Gamma,\tilde{p},\partial_{w_i} \tilde{p}\rangle  \right)_{1\leq i\leq 6} \, dt.
\end{align*}
Using that $\|\tilde{p}\|_\infty$ and $\|\partial_w\tilde{p}\|_\infty$ are uniformly bounded with respect to $(\eta,w)$, it follows that $\partial_w \tilde{p}(2\eta) $ converges to $\partial_{q_1}c_0(q_1) $ as $\eta\to 0^+$, uniformly for $w\in\overline{B}\big( (q_0 , 0), \tilde{r} \big)$, while the position $\partial_w \tilde{q}(2\eta)$ converges to $\partial_w \tilde{q}(0)=0$  as $\eta\to 0^+$.

During the time interval $[2\eta,T-2\eta]$, the right-hand side of the first line in \eqref{toyder} vanishes, so the equation satisfied by $ (\partial_w\tilde{q},\partial_w\tilde{p})$ is close to \eqref{geow}, provided that $\tilde{p}$ is close to $\bar{p}$  (in the $L^\infty$ norm).
We may take the difference of \eqref{toyder} and \eqref{geow}, use a Gronwall argument, \eqref{tglim} and the boundedness of $\|\tilde{p}\|_\infty$, $\|\bar{p}\|_\infty$, $\|\partial_w\tilde{p}\|_\infty$, respectively $\|\partial_w \bar{p}\|_\infty$, to conclude that \eqref{tglimd} holds. Implicitly, we have that 
\begin{align}
(\partial_w \tilde{q},\partial_w\tilde{p})(T-2\eta)\to \left( \begin{array}{ccc}
\text{Id}  & \partial_{q_1}c_1 (q_1)\\
0  &  0 
\end{array} \right), \text{ as }\eta\to 0^+.
\end{align}

Finally, during the interval $[T-2\eta,T]$ we proceed as we did during $[0,2\eta]$. More precisely, integrating \eqref{toyder}, using a Gronwall estimate and the form of $\partial_w v_1$ from \eqref{cder} allows us to 
reorient the final velocity $\partial_w \tilde{p}(T)$ from $(\partial_{q_1}c_1 (q_1) , 0 )$ to $(0, \text{Id})$, while the position $\partial_w \tilde{q}(T)$ converges to $(\text{Id},0)$  as $\eta\to 0^+$, which concludes the proof.
\end{proof}

%%%

\paragraph{Differentiating and estimating the real model}

As we have also noted at the beginning of Section \ref{proof-approx}, a lot of the arguments we will use, in particular those using Gronwall type estimates to prove the boundedness or closeness of certain trajectories, work in the same way as in Section 5 of \cite{GKS}, so for the sake of brevity we will often skip such details, only sketch the main elements of the proof and refer the reader to \cite{GKS}.

For simplicity of notation, we rewrite  \eqref{renform} into the following form, in which we only consider the dependences with respect to $q^0$, $p^0$, $w$ and $\beta_\eta$. We have
\begin{align}\label{renformw}
  \begin{split}
\mathcal{M}(p^0)'+\langle \Gamma,p^0,p^0\rangle &= \beta^2_\eta\hat{F}_{1,0} (q^0,w)+ \tilde{\beta}^2_\eta \hat{F}_{1,1} (q^0,w) + (\beta_\eta+\tilde{\beta}_\eta) \hat{F}_2 (q^0,p^0,w) +(\beta'_\eta+\tilde{\beta}'_\eta) \hat{F}_3(q^0),\\
(q^0)'&=\mathcal{R}(q^0) p^0,
\end{split}
\end{align}
where  $\tilde{\beta}_\eta=\beta_\eta(T-\cdot)$, the maps $\hat{F}_{1,0}, \hat{F}_{1,1}, \hat{F}_{2}, \hat{F}_{3}$ can be deduced from \eqref{renform} and are of class $C^\infty$ in $(q^0,p^0)$ and $C^1$ in $w$. 
Furthermore, note that we have 
\begin{align}\label{newef}
\hat{F}_{1,i}(q,w)=F_{1,a}\left[\mathcal{A}\left[q,\overline{g}^0\left(q_i,-2v_i\right)\right]\right],\ \  i=0,1,
\end{align}
where $\mathcal{A}$ was given in Definition \ref{pot}, $F_{1,a}$ in Definition \ref{ef2}, and $v_i$ in \eqref{vi}.

Proceeding as in Section 6 of \cite{GKS}, it can be shown that $\|p^0\|_\infty$ is uniformly bounded with respect to $0<\eta<<1$ and  $ w\in\overline{B}\big( (q_0 , 0), \tilde{r} \big)$.

The differentiability with respect to $ w $ of the solution $(q^0,p^0)$ to the ODE system \eqref{renformw} can be inferred from the regularity of the maps involved. We deduce the following ODE system for $(\partial_w q^0,\partial_w p^0)$.
\begin{align}\label{renformder}
  \begin{split}
\mathcal{M}(\partial_w p^0)'&+2\left( \langle \Gamma,p^0,\partial_{w_i} p^0\rangle  \right)_{1\leq i\leq 6} = (\beta_\eta+\tilde{\beta}_\eta) \nabla \hat{F}_2 (q^0,p^0,w) \cdot (\partial_w q^0,\partial_w p^0,\text{Id}) +(\beta'_\eta+\tilde{\beta}'_\eta) \partial_q \hat{F}_3(q^0) \partial_w q^0\\
&+\beta^2_\eta\left(\partial_q \hat{F}_{1,0} (q^0,w) \partial_w q^0 +\partial_w \hat{F}_{1,0} (q^0,w)\right)+ \tilde{\beta}^2_\eta\left(\partial_q \hat{F}_{1,1} (q^0,w) \partial_w q^0 +\partial_w \hat{F}_{1,1} (q^0,w)\right) ,\\
(\partial_w q^0)'&=\left(\left(D\mathcal{R}(q^0) \cdot \partial_{w_i} q^0 \right) p^0 \right)_{1\leq i\leq 6} + \mathcal{R}(q^0) \partial_w p^0,
\end{split}
\end{align}
with $(\partial_w q^0,\partial_w p^0)(0)=0$. Let us show that $\|\partial_w p^0\|_\infty$ is uniformly bounded with respect to $0<\eta<<1$ and  $ w\in\overline{B}\big( (q_0 , 0), \tilde{r} \big)$. Indeed, first we may estimate $$|\partial_{w}q^0|\leq \int_{[0,2\eta]\cup[T-2\eta,T]} |\partial_{w} p^0| \leq C \eta$$ on the support of $\beta'_\eta+\tilde{\beta}'_\eta$, so that we can obtain $\left|\int (\beta'_\eta+\tilde{\beta}'_\eta) \partial_q \hat{F}_3(q^0) \partial_w q^0\right| \leq C \sqrt{\eta}$.
Then, using the continuity of the bilinear from associated with $\Gamma$, the uniform boundedness of $\|p^0\|_\infty$, and the continuous differentiability of $\hat{F}_{1,0}, \hat{F}_{1,1}, \hat{F}_{2}, \hat{F}_{3}$, allows us to deduce from \eqref{renformder} that
\begin{align*}
|\partial_w p^0(t)| +|\partial_w q^0(t)|\leq C+ C(1+\| p^0\|_\infty )\int_0^t \left[ |\partial_w p^0(s)| + |\partial_w q^0(s)|  \right]\, ds ,
\end{align*} for $t\in[0,T]$,
with $C>0$ independent of $(\eta,w)$.
We then obtain the boundedness of $\|\partial_w p^0\|_\infty$
by the means of a Gronwall estimate.

Let us show that $(\partial_w q^0,\partial_w p^0)$ is close to $(\partial_w \tilde{q},\partial_w\tilde{p})$ in $C([0,T])$ by the means of a comparison argument. In order to see this, we claim that the only terms in \eqref{renformder} that need further investigation are
 $\partial_w \hat{F}_{1,i} (q^0,w)$, $i=0,1$, since all the other terms are sufficiently regular (and of small order with respect to $\eta>0$) for a Gronwall-type estimate. 
Indeed,
$(p^0,\partial_w p^0,w)$ is uniformly bounded in $(w,\eta)$, the other terms on the right-hand side of \eqref{renformder} are regular with respect to $(q^0,p^0,w)$, we recall that we have $\left|\int (\beta'_\eta+\tilde{\beta}'_\eta) \partial_q \hat{F}_3(q^0) \partial_w q^0\right| \leq C \sqrt{\eta}$, as well as $\int \beta_\eta \leq C \sqrt{\eta}$, and we once again estimate $|\partial_{w}q^0|\leq  C \eta$ on the support of $\beta^2_\eta+\tilde{\beta}^2_\eta$, so that we can obtain $$\left|\int \left[ \beta^2_\eta\partial_q \hat{F}_{1,0} (q^0,w) \partial_w q^0 +\tilde{\beta}^2_\eta\partial_q \hat{F}_{1,1} (q^0,w) \partial_w q^0 \right] \right| \leq C \eta.$$
 
 For the terms $\partial_w \hat{F}_{1,i} (q^0,w)$, let us first observe that, since $\|p^0\|_\infty$ is uniformly bounded with respect to $0<\eta<<1$ and  $ w\in\overline{B}\big( (q_0 , 0), \tilde{r} \big)$, there exists $C>0$ independent of $(w,\eta)$ such that $|q^0(t)-q_0|\leq C\eta$ for $t\in[0,2\eta]$ and  $|q^0(t)-q_1|\leq C\eta$ for $t\in[T-2\eta,T].$ Since $\partial_w \hat{F}_{1,i} $ are Lipschitz in $q$, we get that 
 \begin{align}\label{feder}
 \left|\partial_w \hat{F}_{1,i} (q^0(t),w) -\partial_w \hat{F}_{1,i} (q_i,w)\right|\leq C\eta \text{ for }t\in[i(T-2\eta),i(T-2\eta)+2\eta],\ i=0,1.
 \end{align}
Therefore, it suffices to estimate $\partial_w \hat{F}_{1,i} (q_i,w)$ on $[0,2\eta]\cup[T-2\eta,T]$ instead of $\partial_w \hat{F}_{1,i} (q^0(t),w)$. We use \eqref{newef} and recall that $w=(q_1,q'_1)$, to infer from Proposition \ref{farcontr0} and Definition \ref{ef2} that
\begin{align}\label{vder}
\partial_v F_{1,a}\left[\mathcal{A}\left[q_i,\overline{g}^0\left(q_i,v\right)\right]\right]=-\frac{1}{2} \text{Id},\ \ \partial_{q_1} F_{1,a}\left[\mathcal{A}\left[q_i,\overline{g}^0\left(q_i,v\right)\right]\right]=0.
\end{align}
On the other hand, using \eqref{newef}, we get that
\begin{align}\label{fvider}
\partial_w \hat{F}_{1,i} (q_i,w)= \partial_w v_i.
\end{align}

We may conclude by using \eqref{toyder}, \eqref{renformder}, the respective uniform bounds for the solutions of the two systems, furthermore, \eqref{feder}, \eqref{fvider}, and a Gronwall argument, that 
$\| (\partial_w q^0,\partial_w p^0)-(\partial_w \tilde{q},\partial_w\tilde{p})\|_{C([0,T])} \to 0$ as $\eta \to 0^+,$ uniformly with respect to $ w\in\overline{B}\big( (q_0 , 0), \tilde{r} \big)$,
which exactly gives us \eqref{trkul}.

\subsubsection{Proof of Theorem \ref{approxeuzero} with \eqref{H01}}\label{th5h}

Finally, in order to ensure that \eqref{H01} holds, we will add another impulsive control around time $T/2$ which we will show does not affect \eqref{renform} in any substantial manner.

More precisely, we change the control given by \eqref{gcontrol} to the following
\begin{align} \label{gcontrol01}
\begin{split}
g^0_{\eta} (t,x) := \beta_{\eta}(t) \overline{g}^0(q_0,-2v_0)(x)  + \beta_{\eta} (t-T/2+\eta) \overline{g}^0\left(q^0 (T/2-\eta),0\right)(x) + \beta_{\eta} (T-t) \overline{g}^0(q_1,-2v_1)(x),
\end{split}
\end{align}
where we note that due to Proposition \ref{farcontr0}, the function $\overline{g}^0\left(q^0 (T/2-\eta),0\right)$ is non-trivial and uniformly bounded with respect to $\eta>0$, due to the boundedness of $q^0 (T/2-\eta)$.
By the same considerations as in Section \ref{step1t5}, the associated solution $(q^0,p^0)$ will converge to the solution $(\tilde{q},\tilde{p})$ of \eqref{toy} in $C([0,T])$ as $\eta\to 0^+$, since we deduced $\overline{g}^0\left(q^0 (T/2-\eta),0\right)$ by applying Proposition \ref{farcontr0} with $v=0$, so its effect on the solid equation is neglectable as $\eta\to 0^+$. Furthermore, we have
$$ \partial_{w} F_{1,a}\left[\mathcal{A}\left[q^0 (T/2-\eta),\overline{g}^0\left(q^0(T/2-\eta),0\right)\right]\right]=0,$$
so we may repeat the same arguments as in Section \ref{step1t5} for $(\partial_w q^0,\partial_w p^0)$ to arrive at the same conclusion.

To ensure that  \eqref{H01} holds, we note that from Proposition \ref{farcontr0} and by continuity it follows that
$$\text{span}\left\{(n(x),x^\perp \cdot n(x)),\ x\in\text{supp } \nabla\mathcal{A}\left[q^0 (T/2),\overline{g}^0\left(q^0 (T/2-\eta),0\right)\right] \cap \partial\mathcal{S}_0\right\} = \mathbb{R}^3,$$
for $\eta>0$ small enough. Due to \eqref{gcontrol01}, we have
$$u^0(T/2,\cdot)=\nabla\left( p^0 (T/2) \cdot \Phi \right) + \beta_\eta(\eta) \nabla\mathcal{A}\left[q^0 (T/2),\overline{g}^0\left(q^0 (T/2-\eta),0\right)\right].$$
Since $\beta_\eta(\eta)=\mathcal{O}(1/\sqrt{\eta})$ and $p^0 \in L^\infty$ is uniformly bounded with respect to $\eta>0$,
it follows that for $\eta>0$ small enough \eqref{H01} holds. 

We conclude the proof of Theorem \ref{approxeuzero} by further reducing (if necessary) and fixing $\eta=\eta_0>0$ such that  \eqref{H01} also holds.

\begin{remark}\label{g0smallflux}
The fact that one can guarantee the small flux condition \eqref{g0flux} is a direct consequence of the explicit formula for $g^0_\eta (t,x)$ given in \eqref{gcontrol} and of a change of variables in time. 
 Due to the properties of $\beta_{\eta}$ given at the beginning of Section \ref{step1t5} one obtains that the flux
  is of order  $\sqrt{\eta}$. Hence one can reduce $\eta$ again in order to satisfy \eqref{g0flux}. It can be easily seen that this argument is also invariant for passing from small solid displacement to an arbitrary one.
\end{remark}

%%%%%
%%%%
%%%

\section{The boundary layer profiles}
\label{BL}

In this section we prove the existence of the boundary layer profiles appearing near the solid, and provide some regularity estimates which we will use in the sequel to handle the boundary layers' effect on the linearized and remainder terms in the asymptotic development \eqref{scales}, \eqref{scale}.

\subsection{The physical boundary layer profile $v$}\label{BL1}

At order $\mathcal{O}(\sqrt{\varepsilon})$, we look for $v=v(t,x,z)$ which satisfies the following equations
\begin{align}\label{bv1}
\begin{split}
\partial_t v + v \cdot \nabla u^0 + (u^0-u^0_S) \cdot \nabla_x v+ r^0 v^\perp- \partial_z Q n  - \frac{(u^0-u^0_S)\cdot n}{\varphi} z \partial_z v = \partial^2_{zz} v \text{ in }[0,T] \times\overline{\mathcal{F}_0}\times \mathbb{R}_+,\\
\partial_z v  (\cdot,\cdot,0) _{\text{tan}} = 2 \chi(\cdot) \, (D(u^0)n+\mu (u^0-u^0_S))_{\text{tan}}, \text{ in } [0,T]\times \overline{\mathcal{F}_0},\\
v(0,\cdot,\cdot)=0\text{ in }\overline{\mathcal{F}_0}\times \mathbb{R}_+.
\end{split}
\end{align}
Here $\varphi:\mathbb{R}^2\to\mathbb{R}$ denotes a smooth function which is introduced
in the spirit of \cite{CMS} and \cite{IS} such that $\varphi=0$ on $\partial\mathcal{S}_0$, $\varphi>0$ in $\mathcal{F}_0$, $\varphi<0$ in the interior of $\mathcal{S}_0$, and $|\varphi(x)|=\text{dist}(x,\partial\mathcal{S}_0)$ in a small neighbourhood $\mathcal{V}$ of $\partial\mathcal{S}_0$, such that the normal $n$ can be computed as $n=-\nabla\varphi$ close to the boundary and extended smoothly in $\mathcal{F}_0$. Furthermore, we consider $\chi\in C^\infty_0(\overline{\mathcal{F}_0};[0,1])$ such that $\text{supp }\chi\subset \mathcal{V}$ and $\chi=1$ in a neighbourhood of $\partial\mathcal{S}_0$, constructed as in Section 3.4 of \cite{CMS} such that the $x$-support of $v$ does not include points where $\nabla\varphi$ vanishes.

Note that the introduction of such a boundary layer profile $v$ in \eqref{scale} is due to the fact that the PDE satisfied by $u^0$ is of first order, hence it can only satisfy a single scalar boundary condition on $\partial\mathcal{S}_0$, namely $u^0\cdot n=(l^0+r^0 x^\perp)\cdot n$, and therefore the full Navier slip-with-friction condition does not hold for $u^0$. This issue is corrected by $v$ at order $\mathcal{O}(\sqrt{\varepsilon})$ satisfying system \eqref{bv1}. In particular, due to the definition of $\chi$ in the paragraph above, $v$ is compactly supported in a neighbourhood of $\partial\mathcal{S}_0$.

We further introduce the spaces in which we will look for a solution $v$ to \eqref{bv1}, namely the weighted anisotropic Sobolev spaces $H^{k,m,p}$, $k,m,p\in\mathbb{N}$, defined by their norms
$$\|v\|^2_{k,m,p}=\sum_{|\alpha|\leq m,\ j\leq p} \int_{0}^{+\infty}\int_{\mathcal{F}_0}(1+z^{2})^k|\partial^\alpha_x\partial^j_z v|^2\, dx\, dz.$$

We have the following existence and regularity result for $v$, which can be seen as a generalization of Proposition 5 from \cite{IS}.
\begin{proposition}\label{vureg}
Let $(h^0,\vartheta^0,u^0)\in C^{\infty}([0,T];\mathcal{Q}_\delta)\times C^{\infty}([0,T]\times\overline{\mathcal{F}_0};\mathbb{R}^{2})$ as given by Theorem \ref{maineu0}.
There exist a unique solution $v\in L^2((0,T);H^{k,2,1})\cap L^\infty((0,T);H^{k,2,0})$ to \eqref{bv1} such that $\partial_z v \in L^\infty([0,T]\times\mathcal{F}_0\times\mathbb{R}_+)$
and $v$ satisfies the orthogonality property $v\cdot n=0$, for any $z\geq 0$.
Furthermore,
 we have that $v$ is bounded in $L^2((0,T);H^{k,5,3})$, for any $k\in\mathbb{N}$,  and one may define a continuous map $(h_1,\vartheta_1)\mapsto v\in L^2((0,T);H^{k,5,3})$.
\end{proposition}

\begin{proof}[Scheme of proof]

We follow the same strategy as in the proof of Proposition 5 from \cite{IS}.

We first show that the pressure $Q$ in \eqref{bv1} can be separated from the first equation, so we may simply consider $v$ as the solution of the linear PDE
\begin{align}\label{bv2}
\begin{split}
\partial_t v + [v \cdot \nabla u^0 + (u^0-u^0_S) \cdot \nabla_x v + r^0 v^\perp]_{\text{tan}} - \frac{(u^0-u^0_S)\cdot n}{\varphi} z \partial_z v = \partial^2_{zz} v \text{ in } [0,T]\times \overline{\mathcal{F}_0}\times \mathbb{R}_+,\\
\partial_z v  (\cdot,\cdot,0) _{\text{tan}} = 2 \chi(\cdot) \, (D(u^0)n+\mu (u^0-u^0_S))_{\text{tan}}, \text{ in }[0,T]\times \overline{\mathcal{F}_0},\\
v(0,\cdot,\cdot)=0\text{ in }\overline{\mathcal{F}_0}\times \mathbb{R}_+.
\end{split}
\end{align}
Therefore it suffices to prove the conclusions of Proposition \ref{vureg} for $v$ satisfying \eqref{bv2}. Indeed, once the existence of such a $v$ is determined, we may
associate $Q$ as the unique function that vanishes as $z\to+\infty$ and satisfies
\begin{align}\label{bvq}
\partial_z Q = [v \cdot \nabla u^0 + (u^0-u^0_S) \cdot \nabla_x v + r^0 v^\perp]\cdot n,\text{ in }[0,T]\times \overline{\mathcal{F}_0}\times \mathbb{R}_+.
\end{align}
Since $v$ vanishes for $x$ outside of $\mathcal{V}$, so does $Q$. 
Furthermore, it is easy to check that we have
\begin{align}\label{vqreg}
\|Q(t,\cdot,\cdot) \|_{1,1,1} \lesssim \| v(t,\cdot,\cdot)\|_{3,2,0}, \text{ for any } t\in[0,T],
\end{align}
where $``\lesssim"$ denotes an estimation with a constant which does not depend on the parameter $(h_1,\vartheta_1)$.

To deduce the existence of $v$, we note that the fact that $\mathcal{F}_0$ is unbounded does not interfere with the methods used in \cite{IS} to prove the existence of such boundary layers, since as mentioned before, $v$ is compactly supported in $\mathcal V$, due to the definition of $\chi$. Therefore, the same methods as used in the proof of Proposition 5 from \cite{IS} can be applied to deduce that there exists a unique solution $v$ of \eqref{bv2} which satisfies $v\in L^2((0,T);H^{k,2,1})\cap L^\infty((0,T);H^{k,2,0})$, $\partial_z v \in L^\infty([0,T]\times\mathcal{F}_0\times\mathbb{R}_+)$ and $v\cdot n=0$, for any $z\geq 0$.

To obtain the higher regularity estimates in Proposition \ref{vureg}, we bootstrap the methods from the proof of Proposition 5 from \cite{IS} in the following manner. 

First, we note that having higher regularity with respect to $x$ follows simply from the regularity of $u^0$ and repeating the same methods used to obtain $v\in L^2((0,T);H^{k,2,1})$. This allows us to prove $v\in L^2((0,T);H^{k,5,1})$.

Next, we differentiate \eqref{bv2} with respect to $t$ to obtain similar estimates on $\partial_t v$, by using the higher regularity of $(u^0,p^0)$. This in turn allows us to estimate $\partial_{zz}^2 v$, using \eqref{bv2}. 
We iterate the process once more to achieve the desired regularity $v\in L^2((0,T);H^{k,5,3})$.

The continuity part of the result is also straightforward due to Theorem \ref{maineu0}.
The details are left to the reader.

\end{proof}

\subsection{The second boundary corrector $w$}

As in \cite{CMS},  $v$ generates a non-vanishing slow divergence and tangential boundary flux, i.e. $\{ \div v\}$ in $\mathcal{F}_0$, and $(D(v(\cdot,\cdot,0))n+\mu v)_{\text{tan}}$ on $\mathcal{S}_0$ do not vanish a priori. This is undesirable if we take into account \eqref{scales}, \eqref{scale}, since we would like to have that the boundary layers do not contribute to the divergence of $u^\varepsilon$ and that the Navier condition holds for $u^\varepsilon$.

To address this we introduce $w$ given by
\begin{align}\label{bw}
w(t,x,z):=-2e^{-z} (D(v(t,x,0))n(x)+\mu v(t,x,0))_{\text{tan}} -n(x) \int_z^{+\infty} \text{div } v(t,x,z') \, dz'.
\end{align}
It is easy to check that we have
\begin{align}\label{bcorr}
(D(v(\cdot,\cdot,0))n+\mu v)_{\text{tan}}-\frac{1}{2}(\partial_z w)_{\text{tan}}=0\text{ on }\mathcal{S}_0,\text{ and }\{ \div v\}-n\cdot \{\partial_z w \}=0\text{ in }\mathcal{F}_0.
\end{align}
Furthermore, $w$ also vanishes for $x$ outside of $\mathcal{V}$, and for any $k,m,p\in\mathbb{N}$ we have
\begin{align}\label{vwreg}
\|w(t,\cdot,\cdot)\|_{k,m,p} \lesssim \| v(t,\cdot,\cdot) \|_{k+2,m+1,p+1},\text{ for any } t\in[0,T],
\end{align}
where once again we recall that we use $``\lesssim"$ to denote an estimation with a constant which does not depend on the parameter $(h_1,\vartheta_1)$.

\subsection{The inner domain corrector $\theta^\varepsilon$}

Finally, we note that $w$ generates a non-vanishing boundary flux $w\cdot n$ on $\partial\mathcal{S}_0$ and a slow divergence. To address this, for fixed time $t\in[0,T]$, we define $\theta^\varepsilon$ as the solution to
\begin{align}\label{bt1}
\Delta \theta^\varepsilon = -\{ \div w\}\text{ in }\mathcal{F}_0,\ \partial_n \theta^\varepsilon=-w(t,\cdot,0)\cdot n \text{ on }\partial\mathcal{S}_0,\ \lim_{|x|\to+\infty}|\nabla\theta^\varepsilon|=0,
\end{align}
where we recall that for $f=f(t,x,z)$, we denote $\{f\}$ its evaluation at $z=\frac{\varphi(x)}{\sqrt{\varepsilon}}$.

Proceeding as in \cite{CMS}, we see that $\nabla\theta^\varepsilon$ satisfies
\begin{align}\label{bt2}
\begin{split}
\partial_t \nabla\theta^\varepsilon + (u^0-u^0_S)\cdot \nabla(\nabla\theta^\varepsilon) + \nabla\theta^\varepsilon \cdot \nabla u^0 +r^0 \nabla^\perp\theta^\varepsilon+\nabla\rho^\varepsilon=0\text{ for }x\in\mathcal{F}_0,\\
\div \theta^\varepsilon = -\{ \div w\}\text{ for }x\in\mathcal{F}_0,\\
\nabla\theta^\varepsilon \cdot n=-w(t,\cdot,0)\cdot n \text{ for }x\in\partial\mathcal{S}_0,\  \lim_{|x|\to+\infty}|\nabla\theta^\varepsilon|=0.
\end{split}
\end{align}
with pressure term $\rho^\varepsilon:= -\partial_t \theta^\varepsilon -(u^0-u^0_S) \cdot \nabla \theta^\varepsilon$. Furthermore, we have
\begin{align}\label{vtreg}
\| \theta^\varepsilon (t,\cdot)\|_{H^{m+2}} \lesssim \varepsilon^{\frac{1-2m}{4}}\| w(t,\cdot)\|_{0,m+2,m} + \| v(t,\cdot)\|_{1,m+1,0},
\end{align}
\text{ for any }$t\in[0,T],\ m=0,1,2$.

Gathering \eqref{bcorr}, \eqref{bt1}, and the fact that $v\cdot n=0$ for any $(t,x,z)$, we may conclude that we have
\begin{align*}
\div \left(\sqrt{\varepsilon}\{v\}+\varepsilon \{w\} +\varepsilon\nabla\theta^\varepsilon\right)=0\text{ in }\mathcal{F}_0,\\
(D(v(\cdot,\cdot,0))n+\mu v)_{\text{tan}}-\frac{1}{2}(\partial_z w)_{\text{tan}}=0\text{ on }\mathcal{S}_0,\\
(\sqrt{\varepsilon}\{v\}+\varepsilon \{w\} +\varepsilon\nabla\theta^\varepsilon)\cdot n=0\text{ on }\mathcal{S}_0.
\end{align*}
Taking into account  \eqref{eu0}, \eqref{bv1}, this allows us to look for a linearized term $u^1$ (with divergence $g^1$) and a divergence-free remainder $u_R^\varepsilon$ satisfying the boundary conditions
\begin{align*}
u^1 \cdot n& = u^1_S \cdot n,\\ u_R^\varepsilon \cdot n& = u_{R,S}^\varepsilon \cdot n,\\
(D(u_R^\varepsilon)n)_{\text{tan}}& = -\mu (u_R^\varepsilon-u_{R,S}^\varepsilon)_{\text{tan}} - \left(D\left(u^1+\nabla\theta^\varepsilon+w(\cdot,\cdot,0)\right)n\right) + \mu \left(u^1-u^1_S+\nabla\theta^\varepsilon+w(\cdot,\cdot,0)\right)_{\text{tan}},
\end{align*}
on $\mathcal{S}_0$,
so that $u^\varepsilon$ defined via \eqref{scales}, \eqref{scale} satisfies the Navier boundary conditions in \eqref{eucs} and has divergence $g^0+\varepsilon g^1$.
%%%%%
%%%%
%%%
\section{The first order term $u^1$}\label{seu1}

In this section we construct a controlled solution to the linearized terms appearing in the asymptotic expansion \eqref{scales}, \eqref{scale}.

At order $\mathcal{O}(\varepsilon)$, we look for $(u^1,\pi^1,l^1,r^1,g^1)$ satisfying the following system.
\begin{eqnarray}\label{eu1}
\begin{split}
\frac{\partial u^1}{\partial t}+(u^0&-u^0_S)\cdot\nabla u^1 +(u^1-u^1_S)\cdot\nabla u^0 + r^0 (u^{1})^\perp+ r^1(u^{0})^\perp+\nabla \pi^1  = \nabla g^0, \\     \div u^1 &= g^1
\ \text{ for } x \in \mathcal{F}_0,\   
u^1\cdot n = u^1_S \cdot n\ \text{ for } x \in \partial\mathcal{S}_0,\ \lim_{|x|\to+\infty}|u^1|=0,\\
m   (l^1)'  &=  \int_{ \partial \mathcal{S}_0} (\pi^{1}+Q(\cdot,\cdot,0)+\rho^\varepsilon ) \, n \, d\sigma -\int_{ \partial \mathcal{S}_0} (2D(u^0)n-\partial_z v(\cdot,\cdot,0))  \, d\sigma- m  r^{0} ( l^{1})^\perp- m  r^{1}  (l^{0})^ \perp
, \\
\mathcal{J}  \ (r^{1})' &=    \int_{ \partial   \mathcal{S}_0}  (\pi^{1}+Q(\cdot,\cdot,0)+\rho^\varepsilon )  \, x^{\perp}  \cdot n \, d \sigma-\int_{ \partial \mathcal{S}_0} x^{\perp}  \cdot (2D(u^0)n-\partial_z v(\cdot,\cdot,0))  \, d\sigma, \end{split}
\end{eqnarray}
where $u^1_S (t,x)= l^1(t) + r^1(t) x^\perp,$
for  $t\in[0,T],$ with
\begin{align}\label{eu1ic}
u^1(0,\cdot)=u^*, \quad (l^1,r^1) (0) =(h'_0,\vartheta'_0).
\end{align}
Note once again the subtle dependence on $\varepsilon>0$ due to $\rho^\varepsilon$ and $u^*$, which we have ommited from the notation, and recall that $u^*$ was constructed as per \eqref{newid} and is uniformly bounded in $C^2$ with respect to $\varepsilon>0$. Furthermore, $\nabla g^0=\Delta u^0$, since $u^0$ is irrotational. 

The position of the solid $q^1=(h^1,\vartheta^1)$ can be associated analogously with \eqref{frame2}, 
however, we remark that the solid position $q^1$ does not play an important role in our control strategy, as long as it is bounded independently of $\varepsilon$, due to the scaling in \eqref{scales}.
Once again, observe that using Lamb's form from \eqref{lamb}, the gradient of the pressure $\pi^1$ in \eqref{eu1} 
can be expressed as
\begin{align}\label{pres1}
\nabla\pi^1=-\partial_t u^1 - \nabla( u^0 \cdot u^1)+\nabla( u^1 \cdot u^0_S)+\nabla( u^0 \cdot u^1_S)+\nabla g^0 - \omega^1 (u^0-u^0_S)^\perp,
\end{align}
where $\omega^1=\curl u^1$.

We have the following approximate controllability result regarding the solid velocity $(l^1,r^1)$.
\begin{theorem} \label{maineu1}
Let  $T>0$, $\mathcal{S}_0\subset\mathbb{R}^2$ bounded, closed, simply connected with smooth boundary, which is not a disk, and $u^*\in C^\infty(\overline{\mathcal{F}_0})\cap L^2(\mathcal{F}_0)$, $\gamma\in\mathbb{R}$ constructed as per \eqref{newid}, 
$p_0=(h'_0,\vartheta'_0),p_1=(l_1,r_1)\in\mathbb{R}^3$, such that
\begin{gather*}
\div u^*=0 \text{ in }\mathcal{F}_0, \ \curl u^* \in C^\infty_0(\overline{\mathcal{F}_0}),  
\lim_{|x|\to+\infty}|u^*(x)|=0,\\
u^* \cdot n = (h'_0+\vartheta'_0 x^{\perp}) \cdot n \text{ on } \partial \mathcal{S}_0, \ 
\int_{\partial\mathcal{S}_0} u^* \cdot \tau \, d\sigma =\gamma.
\end{gather*}
Given $\nu>0$, $\delta>0$, $\varepsilon>0$ and a solution $(u^0,h^0,\vartheta^0,g^0)$ as described in Theorem \ref{maineu0}, there exists a control $g^1\in C_{0}^{\infty}((0,T)\times\mathcal{F}_0)$, and a solution
 $(l^1,r^1,u^1)\in C^{\infty}([0,T];\mathbb{R}^3\times C^{\infty}(\overline{\mathcal{F}_0};\mathbb{R}^{2}))$ 
 to  
 \eqref{eu1}, \eqref{eu1ic}, which satisfies $u^1\in C([0,T];L^2(\mathcal{F}_0))$ and
$$|(l^1,r^1)(T)-(l_1,r_1)|\leq \nu \text{, } \text{supp }g^1(t,\cdot) \subset R(\vartheta^0(t))^T (B_c -h^0(t)), \text{ for all } t\in [0,T].$$
Furthermore, $(l^1,r^1,u^1)$ is uniformly bounded in $L^\infty(\mathbb{R}^3)\times L^\infty(C^2_\infty(\overline{\mathcal{F}_0}))$ with respect to  
$\varepsilon\in(0,1]$; and one may define a map
$(h_1,\vartheta_1,l_1,r_1)\mapsto (l^1,r^1,u^1)\in L^\infty(\mathbb{R}^3)\times L^\infty(C^2_\infty(\overline{\mathcal{F}_0}))$ which
is continuous, uniformly  with respect to 
$\varepsilon\in(0,1]$.
\end{theorem}

The rest of the section is dedicated to proving Theorem \ref{maineu1}.
The proof will be based on a similar strategy as that of Theorem \ref{maineu0} before, with a slight simplification due to the fact that we do not want to control the solid position, but only the velocity. Hence one impulsive control will be enough.

\subsection{Vorticity}

First we consider $\omega^1$ as the regular solution of the vorticity equation
\begin{align}\label{vorteu1}
\begin{split}
\frac{\partial \omega^1}{\partial t}+(u^0-u^0_S)\cdot\nabla \omega^1 + \omega^1 g^0  = 0
\ \text{ for } x \in \mathcal{F}_0, \\
\omega^1(0,\cdot)=\curl u^*(\cdot).
\end{split}
\end{align}
Note that for $\omega^1$ we have once more omitted the dependence with respect to $\varepsilon>0$ from the notation.

We recall the definition of the hydrodynamic Biot-Savart operator on $\mathcal{F}_0$, which can be seen as the ``inverse'' of the $\curl$ operator, namely for $\omega\in C_0^\infty ( \overline{\mathcal{F}_0})$, we consider $K_H[\omega]\in C^\infty ( \overline{\mathcal{F}_0})$ given as the unique solution of the system
\begin{equation}
\label{zozoFormalBS}
  \left\{
      \begin{aligned}
&\div K_H[\omega] = 0,\ \, \curl K_H[\omega]= \omega \ \text{ in } \mathcal{F}_0 ,\\
& K_H[\omega] \cdot n = 0\ \text{on}\ \partial \mathcal{S}_0, \lim_{|x|\to+\infty}|K_H[\omega]|=0,\\
&\int_{\partial\mathcal{S}_0} K_H[\omega] \cdot \tau \, d\sigma =0.
\end{aligned} \right.
\end{equation}

We have the following existence and regularity result regarding the vorticity, in particular to investigate the dependence with respect to $(h_1,\vartheta_1)\in\mathcal{Q}_{\delta}$ and $\varepsilon\in(0,1]$.
\begin{proposition}\label{vortreg}
Let $(h_1,\vartheta_1)\in\mathcal{Q}_{\delta}$, $(q^0,p^0,u^0)$ as in Theorem \ref{maineu0}.
There exists a unique smooth solution
$\omega^1\in C^\infty([0,T];C_0^\infty ( \overline{\mathcal{F}_0}))$ of \eqref{vorteu1} such that we have the following.
\begin{itemize}
\item[(i)]
 $K_H[\omega^1]$ is bounded in $L^\infty(C^2_\infty(\overline{\mathcal{F}_0}))$, $\omega^1$ is bounded in $C([0,T]\times\overline{\mathcal{F}_0})$, both uniformly with respect to $(h_1,\vartheta_1)\in\mathcal{Q}_{\delta}$ and $\varepsilon\in(0,1]$, furthermore, the maps $(h_1,\vartheta_1)\mapsto K_H[\omega^1]\in L^\infty(C^2_\infty(\overline{\mathcal{F}_0}))$, $(h_1,\vartheta_1)\mapsto\omega^1\in C([0,T]\times\overline{\mathcal{F}_0})$ are continuous, uniformly with respect to $\varepsilon\in(0,1]$.
 \item[(ii)] There exists $R>0$, which does not depend on $(h_1,\vartheta_1)\in\mathcal{Q}_{\delta}$ or $\varepsilon\in(0,1]$, such that $\text{supp }\omega^1(t,\cdot) \subset\text{supp }\omega^1(0,\cdot) +\overline{B}(0,R),$ for all $t\in[0,T]$.
 \item[(iii)]  $K_H[\partial_t \omega^1]$ is uniformly bounded in $L^2(L^3(\mathcal{F}_0))$ with respect to $(h_1,\vartheta_1)\in\mathcal{Q}_{\delta}$ and $\varepsilon\in(0,1]$, and the map $(h_1,\vartheta_1)\mapsto  K_H[\partial_t \omega^1] \in L^2(L^3(\mathcal{F}_0))$ is continuous, uniformly with respect to $\varepsilon\in(0,1]$.
 \end{itemize}
\end{proposition}
\begin{proof}

Note that the existence and regularity of $\omega^1$ is not in question here, it is classical that if the initial vorticity is smooth and compactly supported, then there exists a unique smooth solution to \eqref{vorteu1} which is also compactly supported at any time. However, tracking the effect of the control $g^0$ and initial data $u^*$ in \eqref{vorteu1} to get uniform estimates with respect to $(h_1,\vartheta_1)$ and $\varepsilon\in(0,1]$ is not trivial. Recall that $u^*$ is uniformly bounded in $C^2$ with respect to $\varepsilon\in(0,1]$.

(i)
It is classical that, for any $\alpha\in(0,1)$ and for any $\omega\in C^{1,\alpha}_0$, we have
\begin{align}\label{vr1}
\|K_H[\omega]\|_{2,\alpha}\lesssim \|\omega\|_{1,\alpha}.
\end{align}
On the other hand, from \eqref{vorteu1} it follows (c.f. \cite{Bar}) that
\begin{align*}
\frac{d}{dt_{+}}  \|\omega^1(t,\cdot)\|_{0,\alpha} \leq \|g^0(t,\cdot)\, \omega^1(t,\cdot)\|_{0,\alpha}+\alpha \|\nabla (u^0-u^0_S)(t,\cdot)\|_{\infty}  \|\omega^1(t,\cdot)\|_{0,\alpha},
\end{align*}
where $\frac{d}{dt_{+}} $ denotes the right derivative.
A Gronwall estimate gives us
\begin{align}\label{vr2}
\|\omega^1(t,\cdot)\|_{0,\alpha} \leq C \|\omega^1(0,\cdot)\|_{0,\alpha} \, e^{\|u^0\|_{L^1(C^2_\infty)}+ \|p^0\|_\infty} \leq C,
\end{align}
due to the regularity given by Theorem \ref{maineu0}. Iterating this argument once again by differentiating \eqref{vorteu1} with respect to $x$ gives us
\begin{align*}
\frac{d}{dt_{+}}  \|\nabla\omega^1(t,\cdot)\|_{0,\alpha} \leq  \|\nabla g^0(t,\cdot)\, \omega^1(t,\cdot) \|_{0,\alpha}+  \|g^0(t,\cdot) \nabla \omega^1(t,\cdot)  \|_{0,\alpha}+\|\nabla\omega^1(t,\cdot) \cdot \nabla u^0(t,\cdot) \|_{0,\alpha}\\+2|r^0(t)| \|\nabla\omega^1(t,\cdot)\|_{0,\alpha} +\alpha \|\nabla (u^0-u^0_S)(t,\cdot)\|_{\infty} \|\nabla\omega^1(t,\cdot)\|_{0,\alpha},
\end{align*}
from where we deduce once more by a Gronwall estimate and \eqref{vr2} that
\begin{align}\label{w1alf}
\|\omega^1(t,\cdot)\|_{1,\alpha} \leq C \|\omega^1(0,\cdot)\|_{1,\alpha} \, e^{\|u^0\|_{L^1(C^3_\infty)}+ \|p^0\|_\infty} \leq C,
\end{align}
thanks to Theorem \ref{maineu0}.
Therefore, (i) follows.

(ii)
We define $\phi_t (\cdot)\in C^1(\overline{\mathcal{F}_0})$, $t\in[0,T]$, as the flow of $u^0-u^0_S$,
\begin{align}\label{vrf}
\frac{d}{dt}\phi_t (x) =( u^0-u^0_S)(t,\phi_t (x) ),\ \ \phi_0(x)=x.
\end{align}
It is classical that, for any $t\in[0,T]$, $\phi_t:\mathbb{R}^2 \to\mathbb{R}^2$ is a $C^1$-diffeomorphism, and \eqref{vorteu1} can be solved by the method of characteristics, namely we have
\begin{align}\label{vro}
\omega^1(t,x)=\omega^1(0,\phi^{-1}_t (x)) \, e^{\displaystyle -\int_0^t g^0(s,\phi_s\circ \phi^{-1}_t (x) )\, ds}
\end{align}
Therefore, the support of $\omega^1$ is transported by $\phi_t$, and for any $x\in\text{supp }\omega(0,\cdot)$ we have
\begin{align*}
|\phi_t (x)-x| \leq  \|u^0\|_{L^1(C^0_\infty)} +T \|l^0\|_\infty +T \|r^0\|_\infty \max_{x\in\text{supp }\omega(0,\cdot)}|x|  + \int_0^t |r^0(s)||\phi_s (x)-x| \, ds,
\end{align*}
once again a Gronwall estimate gives us
\begin{align*}
|\phi_t (x)-x| \leq R,
\end{align*}
with $R>0$ uniform in $(h_1,\vartheta_1)\in\mathcal{Q}_\delta$ and $\varepsilon\in(0,1]$, due to the regularity given by Theorem \ref{maineu0}. This concludes the proof of (ii).

(iii) From (ii), \eqref{vorteu1} and \eqref{w1alf} we deduce that
\begin{align*}
\| \partial_t \omega^1(t,\cdot)\|_\infty \leq C(|p^0(t)|+\|u^0(t,\cdot)\|_\infty+ \|g^0(t,\cdot)\|_\infty),
\end{align*}
for some $C>0$ which is uniform  in $(h_1,\vartheta_1)\in\mathcal{Q}_\delta$ and $\varepsilon\in(0,1]$.
Therefore,  using the fact that $\omega^1$ is supported in $\text{supp }\omega^1(0,\cdot)+\overline{B}(0,R)$, we have
\begin{align}\label{dtom}
\| \partial_t \omega^1\|_{L^2(L^{6/5})} \leq C.
\end{align}
Finally, 
we apply Lemma 1 from \cite{GS-Low} 
with $p=\frac{6}{5}$ and $\gamma=0$ to get that
$$\| K_H[\partial_t \omega^1]\|_{L^2(L^3)} \leq C \|\partial_t \omega^1 \|_{L^2(L^{6/5})} \leq C ,$$
with $C>0$ uniform in $(h_1,\vartheta_1)\in\mathcal{Q}_\delta$ and $\varepsilon\in(0,1]$. The continuity part of the result follows once again from Theorem \ref{maineu0}.

\end{proof}

\subsection{An ODE reformulation in the linearized case}

In this section we will give a different ODE reformulation of the solid equations in \eqref{eu1} compared to Section \ref{NODE}, due to the fact that the equations are now linear.

We introduce the stream function ${\psi}\in C^\infty(\overline{\mathcal{F}_0})$ for the circulation term in the following way. First we consider the solution $\tilde{\psi}$ of the Dirichlet problem
$\Delta\tilde{\psi}=0$ in $ \mathcal{F}_0$, 
$\tilde{\psi}=1$ on $ \partial\mathcal{S}_0$, $\displaystyle\lim_{|x|\to+\infty}|\nabla\psi(x)|=0$.
Then we set 
\begin{align}\label{str}
\psi=-\left(\int_{\partial\mathcal{S}_0} \partial_n \tilde{\psi} \, d\sigma \right)^{-1}\tilde{\psi},
\end{align}
such that we have
$$\int_{\partial\mathcal{S}_0} \partial_n \psi\, d\sigma = -1,$$
noting that the strong maximum principle gives us $\partial_n \tilde{\psi}<0$ on $\partial\mathcal{S}_0$. Note that we have $\nabla^\perp\psi\in C^2_{\infty}(\overline{\mathcal{F}_0}).$

We remind the reader of the definition of the Kirchhoff potentials given in \eqref{kir},  as well as $\mathcal M$ from Definition \ref{christ0}.
We further introduce the following notations for some new force terms appearing in the solid equations in \eqref{eu1}, which do not depend on the control.
\begin{definition}\label{christ1}
Given a solution $p^0=(l^0,r^0) \in C^{\infty}([0,T];\mathbb{R}^3)$, $u^0\in C^{\infty}([0,T]\times\overline{\mathcal{F}_0};\mathbb{R}^{2})$
to  (\ref{eu0})  with zero initial conditions, 
we define for any 
 $p=(l,r)\in\mathbb{R}^3$, $\omega,\tilde{\omega}\in C_0^\infty ( \overline{\mathcal{F}_0})$, $\varepsilon\in(0,1]$  and $t\in[0,T]$, the functions
 \begin{align*}
 A(t):= &\mathcal{M}^{-1}\int_{ \partial \mathcal{S}_0} \left( \begin{array}{c}
 Q(t,\cdot,0)\, n -2D(u^0(t,\cdot))n + \partial_z v(t,\cdot,0)\\
x^\perp \cdot \left(Q(t,\cdot,0)\, n -2D(u^0(t,\cdot))n + \partial_z v(t,\cdot,0)\right) \end{array} \right) \, d\sigma 
, \\
B_\varepsilon (t) := &   \mathcal{M}^{-1}\int_{ \partial   \mathcal{S}_0}  \rho^\varepsilon (t,\cdot) \, \partial_n \Phi  \, d\sigma,\\
C [\omega,\tilde{\omega}](t):=& -\mathcal{M}^{-1}\int_{\mathcal{F}_0} \left(K_H[\tilde{\omega}] + \omega \left( u^0(t,\cdot)-u^0_S(t,\cdot) \right)^\perp \right) \nabla \Phi \, dx\\&-\mathcal{M}^{-1}\int_{\partial\mathcal{S}_0} (u^0(t,\cdot)-u^0_S(t,\cdot)) \cdot \left( \gamma \nabla^\perp \psi (x)+K_H[\omega](x) \right) \partial_n \Phi \, d\sigma,\\
\langle L(t),p\rangle :=&\mathcal{M}^{-1} \int_{\partial\mathcal{S}_0} \left((l+r x^\perp) \cdot u^0(t,\cdot)-(u^0(t,\cdot)-u^0_S(t,\cdot)) \cdot \nabla \left( p \cdot \Phi \right)\right)\, \partial_n \Phi \, d\sigma\\
&-\mathcal{M}^{-1}\left( m  r^{0} (t) l^\perp+ m  r  (l^{0}(t))^ \perp,0\right).
\end{align*}
\end{definition}

The following Lemma gives us the regularity of the terms introduced above with respect to $(h_1,\vartheta_1)\in\mathcal{Q}_{\delta}$.
\begin{lemma}\label{blr}
Let $(h_1,\vartheta_1)\in\mathcal{Q}_{\delta}$, $(q^0,p^0,u^0)$ as in Theorem \ref{maineu0} and $\omega^1$ be the solution of \eqref{vorteu1}. We have $A,B_\varepsilon,C [\omega^1,\partial_t \omega^1]$ and $L$ are uniformly bounded in $L^2(0,T)$ with respect to $\varepsilon\in(0,1]$ and $(h_1,\vartheta_1)\in\mathcal{Q}_{\delta}$. Furthermore, the map $(h_1,\vartheta_1)\in\mathcal{Q}_{\delta}\mapsto (A,B_\varepsilon,C [\omega^1,\partial_t \omega^1], L )\in \left(L^2(0,T)\right)^4$ is uniformly continuous with respect  to $\varepsilon\in(0,1]$.
\end{lemma}
\begin{proof}

We have
\begin{align*}
|A(t)|\lesssim \|Q(t)\|_{1,1,1} + \| \nabla u^0(t) \|_\infty + \| \partial_z v(t) \|_{1,1,1} \lesssim \|v(t) \|_{3,2,1} + \| \nabla u^0(t) \|_\infty + \| v(t) \|_{1,1,2}.
\end{align*}
due to \eqref{vqreg}. The boundedness result for $A$ follows from Theorem \ref{maineu0} and Proposition \ref{vureg}, so does the continuity in $(h_1,\vartheta_1)$.

Note that we have $\rho^\varepsilon=-\partial_t \theta^\varepsilon - (u^0-u^0_S)\cdot \nabla \theta^\varepsilon$, so the result for $B_\varepsilon$ holds as soon as $\partial_t \theta^\varepsilon$ and $\nabla \theta^\varepsilon$ are uniformly bounded in $L^2(H^1)$ with respect to $\varepsilon>0$ and $(h_1,\vartheta_1)\in\mathcal{Q}_\delta$. 

Since $\theta^\varepsilon$ depends linearly on $w$, which in turn depends linearly on $v$, similarly to \eqref{vwreg} and \eqref{vtreg}, we may estimate 
\begin{align*}
\|\partial_t \theta^\varepsilon(t)\|_{H^1} \lesssim \|\partial_t v(t) \|_{2,3,1}+ \|\partial_t v(t) \|_{1,0,1}.
\end{align*}
We estimate $\partial_t v$ by using equation \eqref{bv2}, to get that
\begin{align*}
\|\partial_t v(t) \|_{k,m,p} \lesssim \left(1+|p^0(t)|+\|u^0(t)\|_{C^{m+1}_\infty}\right)\|v(t)\|_{k+1,m+1,p+2},
\end{align*}
therefore, using Theorem \ref{maineu0} and Proposition \ref{vureg}, we get that $\partial_t \theta^\varepsilon$ is bounded in $L^2(H^1)$, uniformly with respect to $\varepsilon>0$ and $(h_1,\vartheta_1)\in\mathcal{Q}_\delta$. 

Finally, using \eqref{vwreg} and \eqref{vtreg}, we estimate
\begin{align*}
\| \nabla \theta^\varepsilon(t)\|_{H^1} \lesssim\|v(t) \|_{2,3,1},
\end{align*}
and conclude the boundedness and continuity in $(h_1,\vartheta_1)$ of $B_\varepsilon$ by once again using Proposition \ref{vureg} and Theorem \ref{maineu0}.

The results for $C[\omega^1,\partial_t \omega^1]$ and $L$ follow from Theorem \ref{maineu0} and Proposition \ref{vortreg}. In particular, for the term $\int_{\mathcal{F}_0} K_H[\partial_t \omega^1 ]\nabla \Phi \, dx$ recall that $\nabla \Phi(x)= \mathcal{O}(1/|x|^2)$ as $|x|\to+\infty$, therefore $\nabla\Phi \in L^{3/2}(\mathcal{F}_0)$.
\end{proof}

In order to treat the effect of the control, we recall $F_2$ from Definition \ref{def-forces}, and complete it with the following.
\begin{definition}\label{def-forces1}
Given a solution $p^0=(l^0,r^0) \in C^{\infty}([0,T];\mathbb{R}^3)$, $u^0\in C^{\infty}([0,T]\times\overline{\mathcal{F}_0};\mathbb{R}^{2})$
to  (\ref{eu0})  with zero initial conditions, we define for 
any $\alpha\in C^{\infty}(\overline{\mathcal{F}_0}; \mathbb R)$ the map 
\begin{align*}
\tilde{F}_{1}[\alpha](t)=-\int_{\partial\mathcal{S}_0} \nabla\alpha\cdot (u^0(t,\cdot)-u^0_S(t,\cdot)) \, \partial_n \Phi \, d\sigma.
\end{align*}
\end{definition}

We define our notion of controlled solution of the linearized fluid-solid system as follows. 
\begin{definition}\label{CS1}
Given a controlled solution $(q^0,p^0,g^0)$ associated with  \eqref{eu0}, \eqref{frameeu0},
we say that $(p^1,g^1)$ in $C^{\infty} ([0,T]; \mathbb{R}^3) \times C^{\infty}_0 ((0,T); \mathcal C(q^0(t)))$ is a controlled solution associated with  \eqref{eu1}, \eqref{eu1ic} if the following ODE holds true on $[0,T]$:
  \begin{align}\label{renform1}
(p^1)' = S_\varepsilon+\langle L,p^1\rangle+\mathcal{M}^{-1}\tilde{F}_1[\alpha^1]+\mathcal{M}^{-1}F_2 [\partial_t \alpha^1],
\end{align}
with $p^1(0)=(h'_0,\vartheta'_0)$,
 where $\alpha^1(t,\cdot) := \mathcal A[q^0(t),g^1(t,\cdot)] $ with $\mathcal A$ given in \eqref{pot}, and $S_\varepsilon := A+B_\varepsilon+C[\omega^1,\partial_t \omega^1]$.
\end{definition}

We have the following reformulation result for the linearized fluid-solid system.
\begin{proposition} \label{reformC1}
Given a controlled solution $(q^0,p^0,g^0)$ associated with \eqref{eu0}, \eqref{frameeu0}, 
$$p^1 \in C^{\infty}([0,T];\mathbb{R}^3)  ,\ u^1\in C^{\infty}([0,T]\times\overline{\mathcal{F}_0};\mathbb{R}^{2})   \quad    \text{ and  } \  g^1\in  C^{\infty}_0 ((0,T); \mathcal C(q^0(t)))  ,$$
 we have that $(u^1,p^1,g^1)$
 is a solution to \eqref{eu1}, \eqref{eu1ic} if and only if $(p^1,g^1)$
 is a controlled solution associated with  \eqref{eu1}, \eqref{eu1ic}
 and $u^1$ is the unique smooth solution to the div/curl type problem
\begin{equation}\label{zozo1}
  \left\{
      \begin{aligned}
&\div u^1 = g^1\mathbbm{1}_{R(\vartheta^0)^T (B_c -h^0)},\ \, \curl u^1= \omega^1 \ \text{ in } \mathcal{F}_0 ,\\
& u^1 \cdot n =  \left(l^1+ r^1 x^\perp \right)\cdot n\ \text{on}\ \partial \mathcal{S}_0,  \lim_{|x|\to+\infty}|u^1|=0,\\
&\int_{\partial\mathcal{S}_0} u^1 \cdot \tau \, d\sigma =\displaystyle\int_{\partial\mathcal{S}_0} u^* \cdot \tau \, d\sigma=:\gamma,
\end{aligned} \right.
\end{equation}
for $t\in[0,T].$
\end{proposition}

\begin{remark}
The last equation in \eqref{zozo1} can be easily checked as a generalization of Kelvin's theorem of conservation of the circulation around the body for the linearized Euler equation \eqref{eu1}, \eqref{eu1ic}.
\end{remark}

\begin{remark}
Integrating \eqref{vorteu1} on $\mathcal F _0$, using the fact that $g^0=\div (u^0-u^0_S)$ and the divergence theorem, it follows that $\int \omega^1(t,\cdot)$ is conserved for $t\in[0,T]$. Since we have $u_0\in L^2(\mathcal{F}_0)$ and that the circulation around the solid is conserved in \eqref{zozo1}, this implies that we have $u^1(t,\cdot)\in L^2(\mathcal{F}_0)$ for $t\in[0,T]$.
\end{remark}

\begin{proof}[Scheme of the proof of Proposition \ref{reformC1}]

We have the following result for solving div/curl systems of type \eqref{zozo1}, which is a simple consequence of \eqref{kir}, \eqref{pot}, \eqref{zozoFormalBS} and \eqref{str}.
\begin{lemma}\label{ldecomp1}
For any $q=(h,\vartheta)$ in $\mathcal{Q}_\delta$, $\gamma\in\mathbb{R}$,
$p=(l,r)$ in $\mathbb R^2 \times \mathbb R$, $g \in \mathcal C(q)$ and $\omega\in C_0^\infty ( \overline{\mathcal{F}_0})$, the unique solution $u$ in $C^\infty ( \overline{\mathcal{F}_0})$ to the following system:
\begin{equation}
\label{zozoFormal1}
  \left\{
      \begin{aligned}
&\div u = g\mathbbm{1}_{R(\vartheta)^T (B_c -h)},\ \, \curl u= \omega \ \text{ in } \mathcal{F}_0 ,\\
& u \cdot n =  \left(l+ r x^\perp \right)\cdot n\ \text{on}\ \partial \mathcal{S}_0, \lim_{|x|\to+\infty}|u|=0,\\
&\int_{\partial\mathcal{S}_0} u \cdot \tau \, d\sigma =\gamma 
\end{aligned} \right.
\end{equation}
 is given by the following formula, for $x$ in $\overline{\mathcal{F}_0}$, 
\begin{equation}
  \label{praud1}
u(x)=\nabla (p \cdot \Phi (x))+\nabla  \mathcal A[q,g](x)+\gamma \nabla^\perp \psi(x) + K_H[\omega](x).
\end{equation}
\end{lemma}

From here on it is straightforward to conclude the proof of Proposition \ref{reformC1} in the same manner as for the proof of Proposition \ref{reform}, by using \eqref{pres1} to express the pressure appearing in the solid equations in \eqref{eu1}, expressing $u^1$ appearing in \eqref{pres1} with the help of Lemma \ref{ldecomp1}, and using integration by parts (by noting the regularity at infinity of the functions involved) to obtain the terms given in Definitions \ref{christ1} and \ref{CS1}. The details are left to the reader.

\end{proof}

\subsection{An impulsive control strategy to control the final velocity $p^1(T)$ - Proof of Theorem \ref{maineu1}}\label{impctrlu1}

Similarly to the case of $(q^0,p^0)$, we will control the ODE \eqref{renform1} by the means of impulsive control, however, due to the fact that we only want to control $p^1(T)$ and that equation \eqref{renform1} is linear, the geodesic argument can be omitted and it is sufficient to use one impulsive control at for instance time $T/2$ (hence condition \eqref{H01}). The form of the control will also change from \eqref{gcontrol} due to the linearity of \eqref{renform1} with respect  to the control. 

Let $(h_1,\vartheta_1)\in \mathcal{Q}_\delta$ and $(q^0,p^0,u^0)$ be given as in Theorem \ref{maineu0}.
We consider controls of the form
\begin{equation} \label{gcontrol1}
g^1_{\eta} (t,x) :=  \beta^2_{\eta} (t-T/2+\eta) \overline{g}^1[u^0(T/2,\cdot)-u^0_S(T/2,\cdot)]\left(q^0 (T/2-\eta),v\right)(x),
\end{equation}
where $v\in\mathbb{R}^3$, $( \beta_{\eta}^2 )_\eta $  is supported in $[0,2\eta]$ and is an approximation of the unity when $\eta \to 0^{+}$, and $\overline{g}^1$ is deduced from the following result.

\begin{proposition} \label{farcontr1}
Let $\mathcal{K}$ be a compact subset of $C^1(\partial\mathcal{S}_0;\mathbb{R}^2)$ such that for any $V\in\mathcal{K}$  we have $V \cdot n=0$ on $\partial\mathcal{S}_0$ and
\begin{align}\label{fH01}
\text{span}\left\{\partial_n \Phi (x),\ x\in\partial\mathcal{S}_0 \cap \text{supp } V \right\}=\mathbb{R}^3.
\end{align}
For any $V\in\mathcal{K}$, there exists a continuous mapping $\overline{g}^1[V]:\mathcal{Q}_\delta \times \mathbb R^3 \rightarrow\tilde{ \mathcal C}(\mathcal{Q}_\delta)$ 
such that  for any $q\in\mathcal{Q}_\delta$ we have $\text{Range}(\overline{g}^1[V](q,\cdot))\subset \mathcal C(q)  \cap \tilde{\mathcal{C}}(q)$, and for any $(q,v) $ in $\mathcal{Q}_\delta \times \mathbb R^3$, 
the function $\overline{\alpha}^1 := \mathcal A [q,\overline{g}^1[V] (q,v)]$ in $C^\infty (\overline{\mathcal{F}_0};\mathbb{R})\cap C^3_{\infty}(\overline{\mathcal{F}_0};\mathbb{R})$  satisfies:
\begin{gather}
\Delta \overline{\alpha}^1 =0\ \text{in}\  \mathcal{F}_0 \setminus R(\vartheta)^T (\tilde{B}_c -h), \lim_{|x|\to+\infty}|\nabla\overline{\alpha}^1|=0   \text{ and } 
\partial_{n} \, \overline{\alpha}^1=0\ \text{on}\ \partial\mathcal{S}_0,\\
\int_{ \partial \mathcal{S}_0}   \nabla\overline{\alpha}^1 \cdot V  \, \partial_n \Phi \, d \sigma = v .
\end{gather}
Furthermore, the map $V\in \mathcal{K} \mapsto \overline{g}^1[V]\in C(\mathcal{Q}_\delta \times \mathbb R^3 ; \mathcal C(\mathcal{Q}_\delta))$ is also continuous.
\end{proposition}
Proposition \ref{farcontr1} will be proved in Appendix \ref{comproof2}. Note that to deduce $\overline{g}^1$ in \eqref{gcontrol1} we apply Proposition \ref{farcontr1} with $\mathcal{K}:=\left\{ u^0(T/2,\cdot)-u^0_S(T/2,\cdot),\ (h_1,\vartheta_1)\in\mathcal{Q}_\delta\right\}$. Therefore, we have that
the map $(h_1,\vartheta_1)\in\mathcal{Q}_\delta \mapsto  \overline{g}^1[u^0(T/2,\cdot)-u^0_S(T/2,\cdot)]\left(q^0 (T/2-\eta),\cdot\right)\in C( \mathbb R^3 ; \mathcal C(q^0 (T/2-\eta)))$ is also continuous, thanks to Theorem \ref{maineu0}.

We recall that $\tilde{B}_c \subset B_c$ and $\tilde {\mathcal C}$ were introduced in \eqref{tilmc} to reduce the size of the activ control zone $B_c$ enough such that for $\eta>0$ small enough, as long as the solid displacement during the control phase $[T/2-\eta,T/2+\eta]$ is small, we have that $g^1(t,\cdot)$ given by \eqref{gcontrol1} is supported in $R(\vartheta^0(t))^T (B_c -h^0(t))$, for any $t\in[0,T]$.

We set
\begin{align}\label{impctrl1v}
\begin{split}
v:=&-\mathcal{M}\left(e^{-\displaystyle\int_{\frac{T}{2}+\eta}^T L(t)\, dt}p_1-\int_{\frac{T}{2}+\eta}^T e^{-\displaystyle\int_{\frac{T}{2}+\eta}^{t}L(s)\, ds}\, S_\varepsilon(t)\, dt \right)
\\
&+\mathcal{M}e^{\displaystyle\int_0^{\frac{T}{2}-\eta}L(t)\, dt}\left(p_0+\int_0^{\frac{T}{2}-\eta} e^{-\displaystyle\int_0^{t}L(s)\, ds}\, S_\varepsilon(t)\, dt \right)
\end{split}
\end{align}
in \eqref{gcontrol1}, which is uniformly bounded with respect to $\eta,\varepsilon\in(0,1]$ and $(h_1,\vartheta_1)\in \mathcal{Q}_\delta$ due to Lemma \ref{blr}. We define $p^1$ to be the smooth solution of \eqref{renform1} with control \eqref{gcontrol1} and initial data $(h'_0,\vartheta'_0)$, dropping once more the dependence on $\eta>0$ from the notation for the sake of simplicity.

We will prove that $p^1(T)\to p_1$ as $\eta\to0^+$, uniformly with respect to $\varepsilon\in(0,1]$ and $(h_1,\vartheta_1)\in \mathcal{Q}_\delta$. First let us observe that on the time interval $[0,T/2-\eta]$ the control vanishes, so 
using Duhamel's principle on \eqref{renform1}, we may express
\begin{align}\label{p11}
p^1(T/2-\eta)=\displaystyle e^{\displaystyle\int_0^{\frac{T}{2}-\eta}L(t)\, dt}\left(p_0+\displaystyle\int_0^{\frac{T}{2}-\eta}  e^{-\displaystyle\int_0^{t}L(s)\, ds}\, S_\varepsilon(t) \,dt \right).
\end{align}

On the other hand, integrating \eqref{renform1} on $[T/2-\eta,T/2+\eta]$ and using that $g^1_\eta$ vanishes at the endpoints of this interval in order to eliminate the term $\mathcal{M}^{-1}F_2[\partial_t \overline{\alpha}^1]$, we get
\begin{align*}
p^1(T/2+\eta)&=p^1(T/2-\eta)+\int_{\frac{T}{2}-\eta}^{\frac{T}{2}+\eta}  S_\varepsilon(t)\, dt+
\int_{\frac{T}{2}-\eta}^{\frac{T}{2}+\eta}   \beta^2_{\eta} (t-T/2+\eta)\mathcal{M}^{-1} \tilde{F}_1 [\mathcal{A}(q^0(t),\overline{g}^1(q^0 (T/2-\eta),v)] (t) \, dt.
\end{align*} 
Using again Lemma \ref{blr} we get that the first integral term in the right-hand side above converges to zero as $\eta\to 0^+$, uniformly with respect to $\varepsilon\in(0,1]$ and $(h_1,\vartheta_1)\in \mathcal{Q}_\delta$. On the other hand, from Theorem \ref{maineu0}, Proposition \ref{farcontr1} and the uniform boundedness of $v$ it follows that the second integral term converges to $-\mathcal{M}^{-1}v$ as $\eta\to 0^+$, uniformly with respect to $\varepsilon\in(0,1]$ and $(h_1,\vartheta_1)\in \mathcal{Q}_\delta$. We deduce that
\begin{align}\label{p12}
p^1(T/2+\eta)-p^1(T/2-\eta)\to -\mathcal{M}^{-1}v \text{ as }\eta\to 0^+,
\end{align}
uniformly with respect to $\varepsilon\in(0,1]$ and $(h_1,\vartheta_1)\in \mathcal{Q}_\delta$.

Finally, applying Duhamel's principle on $[T/2+\eta,T]$ gives us
\begin{align}\label{p13}
\begin{split}
p^1(T)= e^{\displaystyle\int_{\frac{T}{2}+\eta}^T L(t)\, dt}\left(p^1(T/2+\eta)+ \int_{\frac{T}{2}+\eta}^T e^{-\displaystyle\int_{\frac{T}{2}+\eta}^{t}L(s)\, ds}\,  S_\varepsilon(t)\, dt\right).
\end{split}
\end{align}

We may therefore deduce from \eqref{impctrl1v}, \eqref{p11}, \eqref{p12} and \eqref{p13} that
$$p^1(T)\to p_1 \text{ as }\eta\to 0^+,$$
uniformly with respect to $\varepsilon\in(0,1]$ and $(h_1,\vartheta_1)\in \mathcal{Q}_\delta$.

We conclude the proof of Theorem \ref{maineu1} by 
fixing $\eta=\eta_1>0$ in function of $\nu>0$,
associating an appropriate fluid velocity $u^1$ with $(p^1,g^1_\eta)$ as in Proposition \ref{reformC1},
 and observing that the regularity with respect to $\varepsilon\in(0,1]$ and $(h_1,\vartheta_1)\in \mathcal{Q}_\delta$ stated in Theorem \ref{maineu1} is guaranteed by our construction (by using Theorem \ref{maineu0}, Proposition \ref{vortreg}, Lemma \ref{blr} and Proposition \ref{farcontr1}). In particular we note that since $\eta>0$ is now fixed, there is no issue with establishing the regularity with respect to $\varepsilon$ and $(h_1,\vartheta_1)$ of the term $\mathcal{M}^{-1}F_2[\partial_t \overline{\alpha}^1]$ via Proposition \ref{farcontr1}, which we did not need to investigate previously in our strategy. 

\section{Estimating the remainder}\label{remain}

In this section we establish the existence (in a weak sense) of the remainder terms in \eqref{scales}, \eqref{scale}, such that they satisfy an appropriate energy estimate and continuity property with respect to $(h_1,\vartheta_1,l_1,r_1)$.

\subsection{The equation of the remainder and weak solutions}

As noted before, at this point we have not yet proven the existence of the terms $(u^\varepsilon,\pi^\varepsilon,l^\varepsilon,r^\varepsilon)$, however our goal is to define a weak solution $(u^\varepsilon_R,l^\varepsilon_R,r^\varepsilon_R)$ for the remainder in an appropriate way such that defining $(u^\varepsilon,l^\varepsilon,r^\varepsilon)$ by the right-hand sides of the asymptotic development \eqref{scales}, \eqref{scale} gives us a very weak solution to \eqref{eucs} in the sense of Definition \ref{eweak}.

Therefore, formally we can look
for the equations of the remainder $(u^\varepsilon_R,\pi^\varepsilon_R,l^\varepsilon_R,r^\varepsilon_R)$ by replacing $(u^\varepsilon,\pi^\varepsilon,l^\varepsilon,r^\varepsilon)$ in \eqref{eucs} with the terms given in the expansion \eqref{scales}, \eqref{scale}, and by simplifying the equations using the systems \eqref{eu0}, \eqref{bv1}, \eqref{bw}, \eqref{bt2} and \eqref{eu1} satisfied by the respective terms in the expansion, which we have constructed in the previous sections. 
We obtain the following.
\begin{eqnarray}\label{rem}
\begin{split}
\frac{\partial u_R^\varepsilon}{\partial t}+\varepsilon(u_R^\varepsilon&-u^\varepsilon_{R,S})\cdot\nabla u_R^\varepsilon +\varepsilon r_R^\varepsilon (u_R^{\varepsilon})^\perp+ \nabla \pi_R^\varepsilon -\varepsilon \Delta u_R^\varepsilon =\{ f^\varepsilon\}-A^\varepsilon(u_R^\varepsilon,p_R^\varepsilon),   \\ \div u_R^\varepsilon &= 0
\ \text{ for } x \in \mathcal{F}_0,  \  \lim_{|x|\to+\infty}|u_R^\varepsilon|=0,\\
(D(u_R^\varepsilon)n)_{\text{tan}} &= -\mu (u_R^\varepsilon-u_{R,S}^\varepsilon)_{\text{tan}} + (N^\varepsilon)_{\text{tan}},\
u_R^\varepsilon\cdot n = u_{R,S}^{\varepsilon} \cdot n\ \text{ for } x \in \partial\mathcal{S}_0,\\
m   (l_R^\varepsilon)'  &=  \int_{ \partial \mathcal{S}_0} \pi_R^{\varepsilon} \, n \, \, d\sigma 
 - 2\sqrt{\varepsilon}\int_{ \partial \mathcal{S}_0} \Sigma^\varepsilon  \, d\sigma - 2\varepsilon\int_{ \partial \mathcal{S}_0}  D\left(u_R^\varepsilon\right)n  \, d\sigma- m \varepsilon  r_R^{\varepsilon}  (l_R^{\varepsilon})^\perp-F_C^\varepsilon(p^\varepsilon_R), \\
\mathcal{J}  \ (r_R^{\varepsilon})'  &=    \int_{ \partial   \mathcal{S}_0} \pi_R^{\varepsilon} \, x^{\perp}  \cdot n \, d \sigma - 2\sqrt{\varepsilon}\int_{ \partial \mathcal{S}_0} x^{\perp}  \cdot \Sigma^\varepsilon  \, d\sigma - 2\varepsilon\int_{ \partial \mathcal{S}_0} x^{\perp}  \cdot D\left(u_R^\varepsilon\right)n  \, d\sigma, \end{split}
\end{eqnarray}
where $u_{R,S}^{\varepsilon }(t,x)= l_R^\varepsilon(t) + r_R^\varepsilon(t) x^\perp,$
for  $t\in[0,T],$ with
$u_R^\varepsilon(0,\cdot)=u_0-u^*, (l_R^\varepsilon,r_R^\varepsilon) (0) = 0,$ and recall that as per \eqref{newid}, we have
$$\|u_R^\varepsilon(0,\cdot)\|_2 \leq \varepsilon^{1/8}.$$
Furthermore, we have
\begin{align}\label{nie}
N^\varepsilon = - \left(D\left(u^1+\nabla\theta^\varepsilon+w(\cdot,\cdot,0)\right)n\right) + \mu \left(u^1-u^1_S+\nabla\theta^\varepsilon+w(\cdot,\cdot,0)\right),
\end{align}
\begin{align}\label{sie}
\Sigma^\varepsilon= D\left( v(\cdot,\cdot,0)\right)n-\frac{1}{2} \partial_z w(\cdot,\cdot,0) +\sqrt{\varepsilon} D\left( u^1+\nabla\theta^\varepsilon+w(\cdot,\cdot,0)\right)n ,
\end{align}
\begin{align}\label{eeA}
\begin{split}
A^\varepsilon (u_R^\varepsilon,p_R^\varepsilon) &=(u_R^\varepsilon-u^\varepsilon_{R,S})\cdot\nabla (u^0+\sqrt{\varepsilon}\{v\} + \varepsilon u^1 +\varepsilon \{ w\} + \varepsilon \nabla \theta^\varepsilon)
+r^\varepsilon_R(u^0+\sqrt{\varepsilon}\{v\} + \varepsilon u^1 +\varepsilon \{ w\} + \varepsilon \nabla \theta^\varepsilon)^\perp
\\&+ (r^0+\varepsilon r^1) (u_R^{\varepsilon})^\perp
+(u^0+\sqrt{\varepsilon}\{v\} + \varepsilon u^1 +\varepsilon \{ w\} + \varepsilon \nabla \theta^\varepsilon-u^0_S-\varepsilon u^1_S)\cdot\nabla u_R^\varepsilon ,
\end{split}
\end{align}
\begin{align}\label{corio}
F_C^\varepsilon(p^\varepsilon_R)=m  ((r^0+\varepsilon r^1)  (l_R^{\varepsilon})^\perp+\varepsilon r^1 (l^{1})^ \perp+ r^\varepsilon_R(l^0+\varepsilon l^1)^\perp),
\end{align}
and
\begin{align}
f^\varepsilon=f_\Delta^\varepsilon+f_\nabla^\varepsilon+\tilde{f}^\varepsilon,
\end{align}
where
\begin{align}\label{eef}
\begin{split}
f^\varepsilon_\Delta &=\left( \Delta\varphi \partial_z v-2(n\cdot \nabla)\partial_z v + \partial_{zz}^2 w 
\right) +\sqrt{\varepsilon}\left(\Delta v +\Delta \varphi \partial_z w-2(n\cdot \nabla)\partial_z w \right) +\varepsilon\left(\Delta w + \Delta u^1 +\Delta \nabla\theta^\varepsilon \right),\\
f^\varepsilon_\nabla &= -\left( v+\sqrt{\varepsilon} \left(u^1-u^1_S+w+\nabla \theta^\varepsilon \right)\right)\cdot \nabla\left( v+\sqrt{\varepsilon} \left(u^1+w+\nabla \theta^\varepsilon \right) \right)-(u^0-u^0_S) \cdot \nabla w - w \cdot \nabla u^0 \\&+\left(u^1-u^1_S+w+\nabla \theta^\varepsilon \right)\cdot n \partial_z(v+\sqrt{\varepsilon} w)
-\sqrt{\varepsilon} r^1 v^\perp -\varepsilon r^1 (u^1+w+\nabla\theta^\varepsilon)^\perp- r^0 w^\perp,
\\
\tilde{f}^\varepsilon &=-\nabla{Q}-\partial_t w.
\end{split}
\end{align}
Note that the solid position $q^\varepsilon_R=(h^\varepsilon_R,\vartheta^\varepsilon_R)$ can be associated analogously with \eqref{frame2}, but in order to get the desired convergence in \eqref{scales}, we do not need to work with it explicitly.

Recall the definition of the spaces $\mathcal{H}(0)$, $\mathcal{V}(0)$ from \eqref{lerayspaces}. We define a notion of weak solution to \eqref{rem} in the following manner. 

\begin{definition}\label{rweak}
We say that
$$u^\varepsilon_R \in C([0,T]; \mathcal{H}(0)) \cap L^2((0,T);\mathcal{V}(0))$$
is a weak solution to the system \eqref{rem} if  for all $\phi \in C^\infty([0,T];\mathcal{H}(0))$ such that $\phi|_{\overline{\mathcal{F}_0}} \in C^{\infty}([0,T]; C_0^{\infty}(\overline{\mathcal{F}_0};\mathbb{R}^{2}))$ the following holds on $[0,T]$,
\begin{align*}
(u^\varepsilon_R(t,\cdot),\phi(t,\cdot))_{\mathcal{H}(0)}-(u^\varepsilon_R(0,\cdot),\phi(0,\cdot))_{\mathcal{H}(0)} = \int_0^t (u^\varepsilon_R(s,\cdot),\partial_t\phi(s,\cdot))_{\mathcal{H}(0)}\, ds\\ + \varepsilon\int_0^t \int_{\mathcal{F}_0}\left[ (u^\varepsilon_R -u^\varepsilon_{R,S})\cdot \nabla  \phi \cdot u^\varepsilon_R  -r^\varepsilon_R (u^\varepsilon_R)^\perp \cdot \phi \right]\, dx \, ds\\- 2 \varepsilon \int_0^t \int_{\mathcal{F}_0} D(u^\varepsilon_R) : D(\phi) \, dx \, ds 
- 2\varepsilon \mu \int_0^t \int_{\partial\mathcal{S}_0}(u^\varepsilon_R-u^\varepsilon_{R,S}) \cdot  (\phi-\phi_S)\, d\sigma \, ds \\
- \int_0^t \left(m \varepsilon r^\varepsilon_R (l^\varepsilon_R)^\perp  \cdot l_\phi + F_C^\varepsilon(p^\varepsilon_R) \cdot l_\phi \right) \, ds
+\int_0^t \int_{\mathcal{F}_0} (\{ f^\varepsilon\}-A^\varepsilon(u_R^\varepsilon,p_R^\varepsilon) )\cdot \phi \, dx \, ds
\\
- 2 \sqrt{\varepsilon} \int_0^t  \int_{ \partial \mathcal{S}_0} \left( \begin{array}{c}
\Sigma^\varepsilon\\
x^\perp \cdot \Sigma^\varepsilon \end{array} \right) \, d\sigma \cdot (l_\phi,r_\phi) \, ds
+ 2 \varepsilon \int_0^t \int_{ \partial \mathcal{S}_0} (N^\varepsilon \cdot \tau ) \left( (\phi-\phi_S)\cdot \tau\right) \, d\sigma \, ds .
\end{align*}
\end{definition}

In order to see that the above definition of a weak solution is justified, 
we suppose that $(u^\varepsilon_R,\pi^\varepsilon_R,l^\varepsilon_R,r^\varepsilon_R)$ is a sufficiently regular strong solution to \eqref{rem} and that $\phi$ is a test function as in Definition \ref{rweak}.
In this case one may multiply the PDE in \eqref{rem} by $\phi$ and perform the usual operations for deriving the variational formulation for Leray-type weak solutions for viscous fluid-solid systems.
The terms on the right-hand side of the PDE in \eqref{rem} can be treated as a source term and as such pose no problem.

However, let us explain how to treat the terms containing $N^\varepsilon$ and $\Sigma^\varepsilon$.  
One may integrate by parts and proceed as in Lemma 1 from \cite{PS} to obtain
\begin{align}\label{ee11}
\begin{split}
\int_{\mathcal{F}_0}\Delta u_R^\varepsilon \cdot \phi =& -2\int_{\mathcal{F}_0} D(u_R^\varepsilon) : D(\phi) +2 \int_{ \partial \mathcal{S}_0} \left( \begin{array}{c}
 D(u_R^\varepsilon)n\\
x^\perp \cdot \left(D(u_R^\varepsilon)n \right) \end{array} \right) \, d\sigma \cdot (l_\phi,r_\phi)\\
&+2\int_{ \partial \mathcal{S}_0} \left((D(u_R^\varepsilon)n)\cdot \tau \right) \left( (\phi-\phi_S)\cdot \tau\right) \, d\sigma.
\end{split}
\end{align}
Combined with $(D(u_R^\varepsilon)n)\cdot \tau = -\mu (u_R^\varepsilon-u_{R,S}^\varepsilon)\cdot \tau + N^\varepsilon \cdot \tau$, the solid equations from \eqref{rem}, and the usual
formula for the pressure term
\begin{align}\label{ee1}
\int_{\mathcal{F}_0}\nabla\pi_R^\varepsilon \cdot \phi = \int_{\partial\mathcal{S}_0}\pi_R^\varepsilon \, \left( \begin{array}{c}
n\\
x^\perp \cdot n \end{array} \right) \, d\sigma \cdot (l_\phi,r_\phi),
\end{align}
one can easily check that the terms with $N^\varepsilon$ and $\Sigma^\varepsilon$ in Definition \ref{rweak} are justified.

Finally, note once more that if $u^\varepsilon_R \in C([0,T]; \mathcal{H}(0)) \cap L^2((0,T);\mathcal{V}(0))$ is a weak solution in the sense of Definition \ref{rweak}, then $u^\varepsilon \in C([0,T]; \mathcal{H}) \cap L^2((0,T);\mathcal{V})$ constructed via \eqref{scales}, \eqref{scale} is a very weak solution in the sense of Defintion \ref{eweak}.

\subsection{Existence and continuity}

We have the following result for the existence of the remainder.
\begin{proposition}\label{EEfinal}
There exists $\bar{\varepsilon}_0>0$, independent of $(h_1,\vartheta_1,l_1,r_1)$, such that for any $\varepsilon\in(0,\bar{\varepsilon}_0]$,
there exists a unique weak solution $u^\varepsilon_R \in C([0,T]; \mathcal{H}(0)) \cap L^2((0,T);\mathcal{V}(0))$ to \eqref{rem} in the sense of Definition \ref{rweak} with initial data given by $u_R^\varepsilon(0,\cdot)=u_0-u^*, (l_R^\varepsilon,r_R^\varepsilon) (0) = 0$.
Furthermore, we have
the following:
\begin{itemize}
\item[(i)] This unique solution satisfies the following energy inequality at time $T$:
\begin{align}\label{ee}
 |p^\varepsilon_R (T)|^2+\|u_R^\varepsilon(T,\cdot)\|_2^2 + \varepsilon \| D(u_R^\varepsilon) \|_{L^2(L^2)}^2 +\varepsilon \mu \int_0^T \int_{ \partial \mathcal{S}_0}\left |u_R^\varepsilon-u_{R,S}^\varepsilon\right|^2 \, d\sigma\, ds \leq C \varepsilon^{1/4},
\end{align}
where $C=C(\nu)>0$ is independent of $\varepsilon>0$, and depends continuously on $(h_1,\vartheta_1,l_1,r_1)$.
\item[(ii)] The map $(h_1,\vartheta_1,l_1,r_1) \mapsto p^\varepsilon_R \in L^\infty(0,T)$ is continuous.
\end{itemize}
\end{proposition}

Note that to prove Theorem \ref{mainsc} we only need the estimate $|(l^\varepsilon_R,r^\varepsilon_R)(T)| \leq C \varepsilon^{1/8}$. However, the PDE satisfied by $u^\varepsilon_R$ in \eqref{rem} is no longer of Euler- or linearized Euler-type, but truly of Navier-Stokes-type. That is why we estimate the fluid and solid velocities of the remainder together via an energy estimate, instead of examining only the solid equations and using some well-constructed decomposition of the fluid velocity as we have done in the previous sections, which would not work in this case.

We split the proof of Proposition \ref{EEfinal} into three parts: in Section \ref{EEpr1} we give an a priori estimate for the added source terms in Definition \ref{rweak}; in Section \ref{EEpr2} we prove the existence result with the energy estimate from (i); and finally in Section \ref{EEpr3} we explain how to obtain the continuity result from (ii).

\subsubsection{Estimation of the source terms}\label{EEpr1}

The proof of the existence relies on the classical Faedo-Galerkin method for Navier-Stokes-type problems, therefore it requires an energy estimate on the whole time interval $[0,T]$ for weak solutions in the sense of Definition \ref{rweak}.

In preparation for such an energy estimate, we will bound the terms containing $N^\varepsilon$, $\Sigma^\varepsilon$, $A^\varepsilon$ and $f^\varepsilon$ in Definition \ref{rweak} using the following Lemma  (recall the definition of the space $\mathcal{H}(0)$ from \eqref{lerayspaces}). 

\begin{lemma}\label{blremreg}
There exist $C>0$ and $b\in C([0,T];\mathbb{R}_+)$ which
depend continuously on $(h_1,\vartheta_1,l_1,r_1)$, such that
\begin{itemize}
\item[(i)] $\displaystyle \int_0^T \int_{ \partial \mathcal{S}_0}\left( |\Sigma^\varepsilon|^2+ |N^\varepsilon|^2  \right) \, d\sigma \, dt \leq C(1+\varepsilon^{-1/4}),$
\item[(ii)]  $\displaystyle\left|\int_{\mathcal{F}_0}A^\varepsilon(\phi,p_\phi)\cdot \phi \right|  \leq b(t) \left(|p_\phi(t)|^2+\|\phi(t,\cdot)\|_{2}^2 \right),\ \forall t\in[0,T],$
\item[(iii)] $\displaystyle \int_0^t \left|\int_{\mathcal{F}_0}\{ f^\varepsilon\} \cdot \phi \right| \, ds \leq C \varepsilon^{1/4} \left(1+\displaystyle\max_{s\in[0,t]}\left( |p_\phi(s)|^2+\|\phi(s,\cdot)\|_{2}^2\right)\right),$ $\forall t\in[0,T]$,
\end{itemize}
uniformly with respect to $\varepsilon\in (0,1]$,
for all $\phi \in C([0,T];\mathcal{H}(0))$ such that $\phi(t,\cdot)$ is compactly supported in $\mathbb{R}^2$ for every $t\in [0,T]$.
\end{lemma}
\begin{proof}

(i) We have the following, by using \eqref{vwreg} and \eqref{vtreg},
\begin{align*}
\| D(v(t,\cdot,0)) \|_{H^1} &\lesssim \| v(t) \|_{1,2,1},\\
\| \partial_z w(t,\cdot,0) \|_{H^1}& \lesssim \| w(t) \|_{1,1,2} \lesssim \| v(t) \|_{3,2,3},\\
\| D(w(t,\cdot,0)) \|_{H^1}& \lesssim \| w(t) \|_{1,2,1}\lesssim \| v(t) \|_{3,3,2},\\
\| D(\nabla \theta^\varepsilon(t)) \|_{H^1}& \lesssim \| \theta^\varepsilon(t) \|_{H^3} \lesssim \varepsilon^{-1/4}  \| v(t) \|_{2,4,2}+ \| v(t) \|_{1,2,0},\\
\int_0^T \int_{\partial\mathcal{S}_0} | D(u^1) |^2 \, d\sigma \, dt &\lesssim \|u^1\|_{L^2(C^1_\infty)}^2.
\end{align*}
Similar estimates hold for the lower order terms in \eqref{nie}, noting that we have $p^1 \in L^\infty$.
The result follows from Theorem \ref{maineu0} and Proposition \ref{vureg}.

%%%%

(ii) Since $\phi$ is compactly supported in the space variable, we have
\begin{align*}
\int_{\mathcal{F}_0} (u^0+\sqrt{\varepsilon}\{v\} + \varepsilon u^1 +\varepsilon \{ w\} + \varepsilon \nabla \theta^\varepsilon-u^0_S-\varepsilon u^1_S) \cdot \nabla \phi \cdot \phi \, dx = -\frac{1}{2} \int_{\mathcal{F}_0} (g^0+\varepsilon g^1) |\phi|^2 \, dx,
\end{align*}
which can be estimated by $ C\left(\| u^0(t)\|_{C^1_\infty}+ \| u^1(t) \|_{C^1_\infty}\right) \| \phi(t,\cdot) \|^2_2.$

Next we set
$$V^\varepsilon:=u^0+\varepsilon u^1 + \varepsilon \nabla \theta^\varepsilon.$$
We may estimate
\begin{align*}
&\left|\int_{\mathcal{F}_0} \phi\cdot\nabla V^\varepsilon \cdot \phi \right| \leq C (\| \nabla u^0(t)\|_\infty+ \| \nabla u^1(t) \|_{\infty} +\varepsilon \|\theta^\varepsilon(t) \|_{H^4})\|\phi(t)\|_2^2,
\end{align*}
where $C>0$ does not depend on $\phi$.

Recall that \eqref{vwreg} and \eqref{vtreg} give us
\begin{align*}
\varepsilon \|\theta^\varepsilon(t) \|_{H^4} \lesssim \varepsilon^{1/4}  \| v(t) \|_{2,5,3}+ \varepsilon \| v(t) \|_{1,3,0}.
\end{align*}
We use Lamb's form from \eqref{lamb} to get 
\begin{align}
-\nabla \left(\phi_S\cdot V^\varepsilon\right)=-\phi_S\cdot\nabla V^\varepsilon+r_\phi(V^\varepsilon)^\perp +\varepsilon \omega^1 (\phi_S)^\perp.
\end{align}
On the other hand,
\begin{align}
\int_{\mathcal{F}_0}\nabla \left(\phi_S\cdot V^\varepsilon \right)\cdot \phi= \int_{\partial\mathcal{S}_0}\phi_S\cdot V^\varepsilon \, \partial_n \Phi \, d\sigma \cdot p_\phi .
\end{align}
Therefore, we may conclude by using once again that $\omega^1$ is smooth and compactly supported, that
\begin{align}\label{ee3}
\left|\int_{\mathcal{F}_0} \left((\phi-\phi_S)\cdot\nabla (u^0+\varepsilon u^1)+r_\phi(u^0+\varepsilon u^1)^\perp \right) \cdot \phi \right|    \leq C\left(\| u^0(t)\|_{C^1_\infty}+ \| u^1(t) \|_{C^1_\infty}+\| v(t) \|_{2,5,3}\right) \left( |p_\phi(t) |^2+\|\phi(t)\|_2^2\right).
\end{align}

Finally, all the remaining terms in \eqref{eeA} contain some derivatives of either $\{v\}$ or $\{w\}$, which are compactly supported. Therefore, we may estimate the terms containing $\phi_S$ more straightforwardly, since the procedure above was done only because we do not have $\phi_S \in L^2(\mathcal{F}_0)$ and we wanted estimates independent of the support of $\phi$. Taking the $L^\infty(\mathcal{F}_0)$ norm of all the derivatives of $\{v\}$ and $\{w\}$ in the remaining terms in \eqref{eeA}, then using once again \eqref{vwreg} and \eqref{vtreg} with the right Sobolev embeddings, we get
\begin{align}\label{ee33}
\begin{split}
\left|\int_{\mathcal{F}_0}A^\varepsilon(\phi,p_\phi)\cdot \phi \right|  \leq b(t) \left(|p_\phi(t)|^2+\|\phi(t,\cdot)\|_{2}^2 \right),
\end{split}
\end{align}
with 
$$b(t)=C\left(\| u^0(t)\|_{C^1_\infty}+ \| u^1(t) \|_{C^1_\infty}+\| v(t) \|_{2,5,3}\right),$$
where $C>0$ does not depend on any other parameter. We conclude the proof of (ii) by using Theorem \ref{maineu0}, Proposition \ref{vureg} and Theorem \ref{maineu1}.

%%%%

(iii) As in the previous point, let us first estimate the contribution of the terms in $f^\varepsilon_\nabla$ from \eqref{eef} which are not compactly supported in $\mathcal{F}_0$. We set
$$W^\varepsilon:=u^1+\nabla\theta^\varepsilon.$$

Using Lamb's form from \eqref{lamb}, we observe that
\begin{align}
\left(W^\varepsilon-u^1_S \right) \cdot \nabla W^\varepsilon +r^1 \left(W^\varepsilon \right)^\perp  = 
 \frac{1}{2}\nabla \left(|W^\varepsilon|^2\right) -\nabla\left(W^\varepsilon \cdot u^1_S\right) +\omega^1(W^\varepsilon-u^1_S)^\perp,
\end{align}
therefore we may estimate
$$
\left|\int_{\mathcal{F}_0} \nabla \left(|W^\varepsilon|^2\right) \cdot \phi \right|=\left| \int_{\partial\mathcal{S}_0}|W^\varepsilon|^2\, \partial_n \Phi \, d\sigma \cdot p_\phi \right| \leq  C \left(  \|u^1\|_{C^0_\infty(\mathcal{F}_0)}^2 + \|\theta^\varepsilon(t) \|_{H^3} ^2  \right) |p_\phi |.$$
Hence, using \eqref{vwreg} and \eqref{vtreg}, we get
\begin{align*}\varepsilon\int_0^t \left|\int_{\mathcal{F}_0} \nabla \left(|W^\varepsilon|^2\right) \cdot \phi \right| \, ds &\leq C \left(\varepsilon \| u^1\|_{L^2(C^0_\infty)}^2 + \varepsilon^{3/4} \| v \|^2_{L^2(H^{2,4,2})}\right)\max_{s\in[0,t]} |p_\phi(s)| \\ 
&\leq C \varepsilon^{3/4} \left( \| u^1\|_{L^2(C^0_\infty)}^4+ \| v \|^4_{L^2(H^{2,4,2})}+ \max_{s\in[0,t]} |p_\phi(s)|^2\right).
\end{align*}
Proceeding similarly for the term containing $\nabla\left(W^\varepsilon \cdot u^1_S\right)$, and using that $\omega^1$ is smooth and compactly supported, we get that
\begin{multline*}
\varepsilon\int_0^t \left|\int_{\mathcal{F}_0} \left( \left(W^\varepsilon-u^1_S \right) \cdot \nabla W^\varepsilon +r^1 \left(W^\varepsilon \right)^\perp \right) \cdot \phi \right| \, ds \\ \leq C\varepsilon^{3/4} \left(\| u^1\|_{L^2(C^1_\infty)}^4+ \| v \|^4_{L^2(H^{2,4,2})}+  \|p^1\|^4_\infty+  \max_{s\in[0,t]}\left( |p_\phi(s)|^2+\|\phi(s,\cdot)\|_{2}^2\right)\right).
\end{multline*}

For the rest of the terms in $f^\varepsilon_\nabla$ from \eqref{eef}, respectively for $f^\varepsilon_\Delta$ and $\tilde{f}^\varepsilon$, the same estimates can be applied as in Section 4.3 of \cite{CMS}.
More precisely, one combines the estimates \eqref{vqreg}, \eqref{vwreg}, \eqref{vtreg} for the boundary layer profiles in order
 to obtain bounds on the $L^2(\mathcal{F}_0)$ norms of the respective terms, noting that the terms appearing at $\mathcal{O}(1)$ benefit from a fast variable scaling gain of $\varepsilon^{1/4}$ in $L^2(\mathcal{F}_0)$. One can then conclude the proof of Lemma \ref{blremreg} by using Theorem \ref{maineu0}, Proposition \ref{vureg} and Theorem \ref{maineu1}. 
The details are left to the reader.
\end{proof}

\subsubsection{A Faedo-Galerkin method for proving the existence of weak solutions}\label{EEpr2}

Now we are in position to prove the existence result from Proposition \ref{EEfinal}. Since the proof uses classical methods, we will only focus on the parts that are different from the existing literature, for the other details we refer the reader to \cite{PS} and the references therein.

The proof consists of a straightforward generalization of the methods presented in \cite{PS}, where the authors give an extension of Leray's theorem to prove the existence of solutions to a similar fluid-solid system in  the three-dimensional case with no source term, based on the method of Faedo-Galerkin approximations. We will show that being in the two-dimensional case and having some extra terms in \eqref{rem} due to $N^\varepsilon$, $\Sigma^\varepsilon$, $A^\varepsilon$, $F_C^\varepsilon$ and $f^\varepsilon$ will pose no difficulty in adapting the same proof to our case.

To simplify notations we define, for any $u,v,w\in\mathcal{H}_0$ such that  $v,w\in C^\infty_0(\mathbb{R}^2)$, the quantities
\begin{align}\label{varf}
\begin{split}
\mathbbm{a}(u,v)=&- 2  \int_{\mathcal{F}_0} D(u) : D(v) \, dx 
- 2\mu\int_{\partial\mathcal{S}_0}(u-u_S) \cdot  (v-v_S)\, d\sigma,\\
\mathbbm{b}(u,v,w)=&\int_{\mathcal{F}_0}\left[ (u -u_S)\cdot \nabla  w \cdot v -r_u v^\perp \cdot w \right]\, dx 
-m r_u (l_v)^\perp  \cdot l_w,\\
\mathbbm{c}^\varepsilon(u,v)=&
- 2 \sqrt{\varepsilon} \int_{ \partial \mathcal{S}_0} \left( \begin{array}{c}
\Sigma^\varepsilon\\
x^\perp \cdot \Sigma^\varepsilon \end{array} \right) \, d\sigma \cdot p_v
+ 2 \varepsilon  \int_{ \partial \mathcal{S}_0} (N^\varepsilon \cdot \tau ) \left( (v-v_S)\cdot \tau\right) \, d\sigma \\
&+\int_{\mathcal{F}_0} (\{ f^\varepsilon\}-A^\varepsilon(u,p_u) )\cdot v \, dx -F_C^\varepsilon(p_u) \cdot l_v  .
\end{split}
\end{align}
Therefore, the equation satisfied by the weak solution $u_{R}^\varepsilon$ from Definition \ref{rweak} can be reformulated as 
\begin{align*}
(u^\varepsilon_R,\phi)_{\mathcal{H}(0)}-(u^\varepsilon_R(0,\cdot),\phi(0,\cdot))_{\mathcal{H}(0)} = \int_0^t \left[(u^\varepsilon_R,\partial_t\phi)_{\mathcal{H}(0)}+ \varepsilon\mathbbm{a}(u_{R}^\varepsilon,\phi)+ \varepsilon\mathbbm{b}(u_{R}^\varepsilon,u_{R}^\varepsilon,\phi)+ \mathbbm{c}^\varepsilon(u_{R}^\varepsilon,\phi)\right] \, ds.
\end{align*}

Following the same methodology as in the proof of Theorem 1 from \cite{PS}, we first consider a truncated system in which the term $u_S(t,\cdot)$ in $\mathbbm{b}$ above is modified in order that it becomes bounded in $L^\infty(\mathbb{R}^2)$.
More precisely, let $M_0\in\mathbb{R}$ such that $\mathcal{S}_0\subset B\left(0,\frac{M_0}{2}\right)$, and for $M>M_0$ we consider $\chi_M:\mathbb{R}^2\to\mathbb{R}^2$ such that $\chi_M(x)=x$ in $B(0,M)$ and $\chi_M(x)=\frac{M}{\|x\|}x$ in $\mathbb{R}^2\setminus B(0,M)$. We then truncate the term $\mathbbm{b}$ from \eqref{varf} by
$$\mathbbm{b}_M(u,v,w)=\int_{\mathcal{F}_0}\left[ (u -(l_u+r_u\chi_M^\perp))\cdot \nabla  w \cdot v -r_u v^\perp \cdot w \right]\, dx-m r_u (l_v)^\perp  \cdot l_w .$$

We claim that for any $M>M_0$ there exists a solution $u_M  \in C([0,T]; \mathcal{H}(0)) \cap L^2((0,T);\mathcal{V}(0))$ (dropping the dependence on $\varepsilon$ from the notation for simplicity)
to the truncated system, i.e. satisfying for any $\phi \in C^\infty([0,T];\mathcal{H}(0))$ such that $\phi|_{\overline{\mathcal{F}_0}} \in C^{\infty}([0,T]; C_0^{\infty}(\overline{\mathcal{F}_0};\mathbb{R}^{2}))$, the equation below on $[0,T]$,
\begin{align}\label{trunc}
(u_M,\phi)_{\mathcal{H}(0)}-(u_M(0,\cdot),\phi(0,\cdot))_{\mathcal{H}(0)} = \int_0^t \left[(u_M,\partial_t\phi)_{\mathcal{H}(0)}+ \varepsilon\mathbbm{a}(u_M,\phi)+ \varepsilon\mathbbm{b}_M(u_M,u_M,\phi)+ \mathbbm{c}^\varepsilon(u_M,\phi)\right] \, ds.
\end{align}
Indeed, we consider a Hilbert basis $(w_j)_{j\geq 1}$ of $\mathcal{V}(0)$ such that
$$w_j\in \{\phi\in C_0^\infty(\mathbb{R}^2):\ \div \phi =0\text{ in }\mathcal{F}_0,\ D(\phi)=0\text{ in }\mathcal{S}_0 \},\ j\geq 0.$$
Using a Faedo-Galerkin method, 
we may construct a sequence of approximate solutions $(u_N)_{N\geq 1}  \subset C([0,T]; \mathcal{H}(0)) \cap L^2((0,T);\mathcal{V}(0))$ (dropping the dependence on $M$ from the notation for simplicity), $u_N=\displaystyle\sum_{i=1}^N g_{iN}(t) w_i$ satisfying, for any $j\in\{1,\ldots,N\}$,
\begin{align*}
(\partial_t u_N,w_j)_{\mathcal{H}(0)}=   \varepsilon\mathbbm{a}(u_N,w_j)+ \varepsilon\mathbbm{b}_M(u_N,u_N,w_j)+ \mathbbm{c}^\varepsilon(u_N,w_j),
\end{align*}
with initial data $u_{N0}$, which is defined as the projection of $u_R^\varepsilon(0)$ onto the space spanned by $w_1,\ldots,w_N$.
It follows that $u_N(t,\cdot) |_{\overline{\mathcal{F}_0}} \in C^\infty_0(\overline{\mathcal{F}_0})$, for all $t\in [0,T]$, due to regularity of  $w_1,\ldots,w_N$.

It can be checked that $u_N$ will satisfy the following energy equality.
\begin{multline}\label{eneq}
\frac{1}{2} \|u_N(t,\cdot) \|_2^2+\frac{1}{2} m|l_N(t)|^2+\frac{1}{2} \mathcal{J}|r_N(t)|^2+2 \varepsilon \int_0^t \int_{\mathcal{F}_0} |D(u_N) |^2 \, dx \, ds 
+2\varepsilon \mu \int_0^t \int_{\partial\mathcal{S}_0}|u_N-u_{N,S}|^2\, d\sigma \, ds \\=\frac{1}{2} \|u_{N0} \|_2^2+
 \int_0^t  \mathbbm{c}^\varepsilon(u_N,u_N) \, ds.
\end{multline}
Furthermore, one has $\|u_{N0} \|_2 \leq \|u_R^\varepsilon(0,\cdot) \|_2=\|u_0-u^*\|_2$.
Note that as long as we can bound the left-hand side of \eqref{eneq} by a constant that does not depend on the approximation or the truncation, one can straightforwardly adapt the rest of the proof of Theorem 1 from \cite{PS} to establish the existence part of the result.

Let us prove that this is indeed the case. In the rest of the section, 
$C>0$
will denote a generic constant, which we usually deduce by using Theorem \ref{maineu0} and Theorem \ref{maineu1} to estimate the respective norms of $u^0,p^0,u^1,p^1$, and which can depend on $(h_1,\vartheta_1,l_1,r_1)$, but in a continuous manner. Furthermore, $C$ is independent of $\varepsilon\in (0,1]$.

Since $p^0,p^1\in L^\infty$, we may use \eqref{corio} to estimate
\begin{align}\label{pee1}
  \left|F_C^\varepsilon(p_N) \cdot l_N \right| \leq C(\varepsilon^2+|p_N (t)|^2).
\end{align}
On the other hand,
\begin{align}\label{pee2}
\left| \int_{ \partial \mathcal{S}_0} \left( \begin{array}{c}
\Sigma^\varepsilon\\
x^\perp \cdot \Sigma^\varepsilon\end{array} \right) \, d\sigma \cdot p_N \right| \leq C \int_{ \partial \mathcal{S}_0} |\Sigma^\varepsilon(t)|^2 \, d\sigma+ |p_N(t)|^2.
\end{align}
Finally, we have
\begin{align}\label{pee3}
\left| \int_{ \partial \mathcal{S}_0} N^\varepsilon \cdot \tau \left( (u_N-u_{N,S})\cdot \tau\right) \, d\sigma\right| \leq C \int_{ \partial \mathcal{S}_0} |N^\varepsilon(t)|^2 \, d\sigma+ \frac{1}{2 C_K}\| \nabla u_N (t,\cdot) \|_2^2+ C |p_N (t)|^2,
\end{align}
where $C_K>0$ is the constant on the right hand side of the Korn inequality:
$$\|h\|_{H^1}^2 \leq C_K \left(\|h\|_2^2+ \| D(h) \|_2^2 \right), \text{ for any } h\in H^1(\mathcal{F}_0).$$

We sum up \eqref{eneq}, \eqref{pee1}, \eqref{pee2} and \eqref{pee3} and use Lemma \ref{blremreg} to obtain
\begin{multline}\label{enineq}
\left(1-C \varepsilon^{1/4}\right) \max_{s\in[0,t]}\frac{1}{2}\left(\|u_N(s,\cdot) \|_2^2+ m|l_N(s)|^2+ \mathcal{J}|r_N(s)|^2\right)+ \varepsilon \| D(u_N) \|_{L^2(L^2)}^2
+2\varepsilon \mu \int_0^t \int_{\partial\mathcal{S}_0}|u_N-u_{N,S}|^2\, d\sigma \, ds  \\ \leq  C\|u_0-u^* \|_2^2+
C \varepsilon^{1/4} + C\int_0^t  b(s) \left(\|u_N(s,\cdot) \|_2^2+ |p_N(s)|^2\right) \, ds,
\end{multline}
for some $C>0$ and $b\in C([0,T];\mathbb{R}_+)$ which depend continuously on $(h_1,\vartheta_1,l_1,r_1)$, and do not depend on $M>M_0$ used for the truncation or $N\in\mathbb{N}$ used in the Faedo-Galerkin method.
Therefore, there exists $\bar{\varepsilon}_0>0$, uniform for $(h_1,\vartheta_1,l_1,r_1)\in\overline{B}((h_f,\vartheta_f,R(\vartheta_f)^T h'_f,\vartheta'_f),\kappa)$,  such that for any $\varepsilon\in (0,\bar{\varepsilon}_0]$ we have $\frac{1}{2}\leq (1-C \varepsilon^{1/4})$.

Using the fact that $ \|u_0-u^* \|_2\leq \varepsilon^{1/8}$ as per \eqref{newid}, a Gronwall argument for \eqref{enineq}, and further reducing the left-hand side to get rid of any unnecessary constants, we get that 
\begin{align}\label{eegen}
 |p_N (t)|^2+\|u_N (t,\cdot)\|_2^2 +  \varepsilon \| D(u_N) \|_{L^2(L^2)}^2 + \varepsilon \mu \int_0^t \int_{ \partial \mathcal{S}_0}\left |u_N-u_{N,S}\right|^2 \, d\sigma\, ds \leq C \varepsilon^{1/4},
\end{align}
for any $t\in[0,T]$.

Since the bound in the right-hand side above is uniform for $N \geq 1$ and $M>M_0$, we may on one hand conclude the proof of the existence of a weak solution $u^\varepsilon_R$ to \eqref{rem} in the sense of Definition \ref{rweak} by the same methods as used in the proof of Theorem 1 from \cite{PS}.
More precisely, we may first pass to the limit as $N\to+\infty$ to deduce the existence of a solution $u_M$ to \eqref{trunc}, for any $M>M_0$. Then we may extract a convergent subsequence $(u_{M_k})_{k\geq 0}$ with limit $u^\varepsilon_R$ and conclude that we may pass to the limit in \eqref{trunc} as $k\to +\infty$,
 in particular using once more the energy inequality \eqref{enineq} and Lemma \ref{blremreg} when needed in order to bound the time derivative of $u_{M_k}$ in a similar fashion as in the aforementioned result (we refer the reader to \cite{PS} for further details).
Therefore, we obtain that $u^\varepsilon_R$ satisfies \eqref{ee} and also \eqref{eegen}. This concludes the proof of the existence and the energy estimate, the uniqueness of the solution follows from classical theory since we are in the 2D case.

\subsubsection{Proving the continuity result}\label{EEpr3}

The continuity result (ii) from Proposition \ref{EEfinal} follows in a straightforward (but lengthy) manner by first observing that the same arguments used to prove Lemma \ref{blremreg} can be adapted to prove the following result regarding the continuity of the terms in \eqref{rem} which depend on $(h_1,\vartheta_1,l_1,r_1)$.
\begin{lemma}
The elements 
\begin{itemize}
\item $\Sigma^\varepsilon, N^\varepsilon \in L^2 ((0,T);L^2(\partial\mathcal{S}_0)),$
\item $F_C^\varepsilon \in \mathcal{L}(\mathbb{R}^3;\mathbb{R}^2),$
\item $A^\varepsilon(\cdot) \in \mathcal{L}( L^2((0,T);\mathcal{H}(0)_c) ; L^2((0,T)\times \mathcal{F}_0)),$
\item $\displaystyle \int_0^t \langle \{f^\varepsilon\},\cdot \rangle_{L^2(\mathcal{F}_0)} \, ds \in \mathcal{L} (C([0,T];\mathcal{H}(0));\mathbb{R} )$, for any $t\in[0,T]$,
\end{itemize}
depend continuously on $(h_1,\vartheta_1,l_1,r_1)\in\overline{B}((h_f,\vartheta_f,R(\vartheta_f)^T h'_f,\vartheta'_f),\kappa)$, where $\mathcal{H}(0)_c$ denotes all functions in $\mathcal{H}(0)$ which are compactly supported in $\mathbb{R}^2$, respectively $\mathcal{L}(\cdot;\cdot)$ is used to denote the space of linear continuous operators between two given spaces.
\end{lemma}
As mentioned above, the proof is a straightforward adaptation of the proof of Lemma \ref{blremreg}, by using the continuity with respect to $(h_1,\vartheta_1,l_1,r_1)$ of $u^0$, $v$ and $u^1$ given by Theorem \ref{maineu0}, Proposition \ref{vureg} and Theorem \ref{maineu1}, hence we omit it.

From here, the standard method for concluding the continuity of the map $(h_1,\vartheta_1,l_1,r_1) \mapsto p^\varepsilon_R \in L^\infty(0,T)$ is the following.
One may consider two weak solutions associated with two different values of the parameters $(h_1,\vartheta_1,l_1,r_1)$ and compare them using an energy estimate of their difference, proving that the right-hand side of the energy estimate will go to zero as the difference of the parameters goes to zero (as done for instance in standard methods to prove the uniqueness of weak solutions, see eg. \cite{Bra,GS-Uniq} for the - more complicated - bounded case). 
However, one needs to take care when
handling the trilinear convective terms, so one may consider the energy estimate first using the truncated approximate solutions corresponding to the two weak solutions, as explained in the construction above for the proof of the existence result. This allows one to deduce that $p_N$ (as well as $u_N\in L^\infty(L^2)\cap L^2(H^1)$) depends continuously on $(h_1,\vartheta_1,l_1,r_1)$, for any $N\geq 1$. However, since the right-hand side of \eqref{eegen} does not depend on $(h_1,\vartheta_1,l_1,r_1)$, nor on $N$, it follows that the convergence of $(u_N,p_N)$ to $(u^\varepsilon_R,p^\varepsilon_R)$ mentioned in the previous section is uniform with respect to $(h_1,\vartheta_1,l_1,r_1)$, from where the desired continuity result follows for $p^\varepsilon_R$.
The details are left to the reader.

This concludes the proof of Proposition \ref{EEfinal}.

%%%
%%%
%%%

\newpage

\section{Conclusion}\label{end}

It follows from Sections \ref{seu0} to \ref{remain} that we have completed the construction of the asymptotic expansion from Section \ref{fluidassexp}, which gives us the result of Theorem \ref{main} due to the reductions presented in Sections \ref{previsc} and \ref{SAE}. We present below an overview of the amplitude with respect to time of the controls $g^0$ and $g^1$ constructed in Theorem \ref{maineu0}, respectively Theorem \ref{maineu1}.

 \begin{figure}[!h]\centering
\resizebox{0.7\linewidth}{!}{\includegraphics{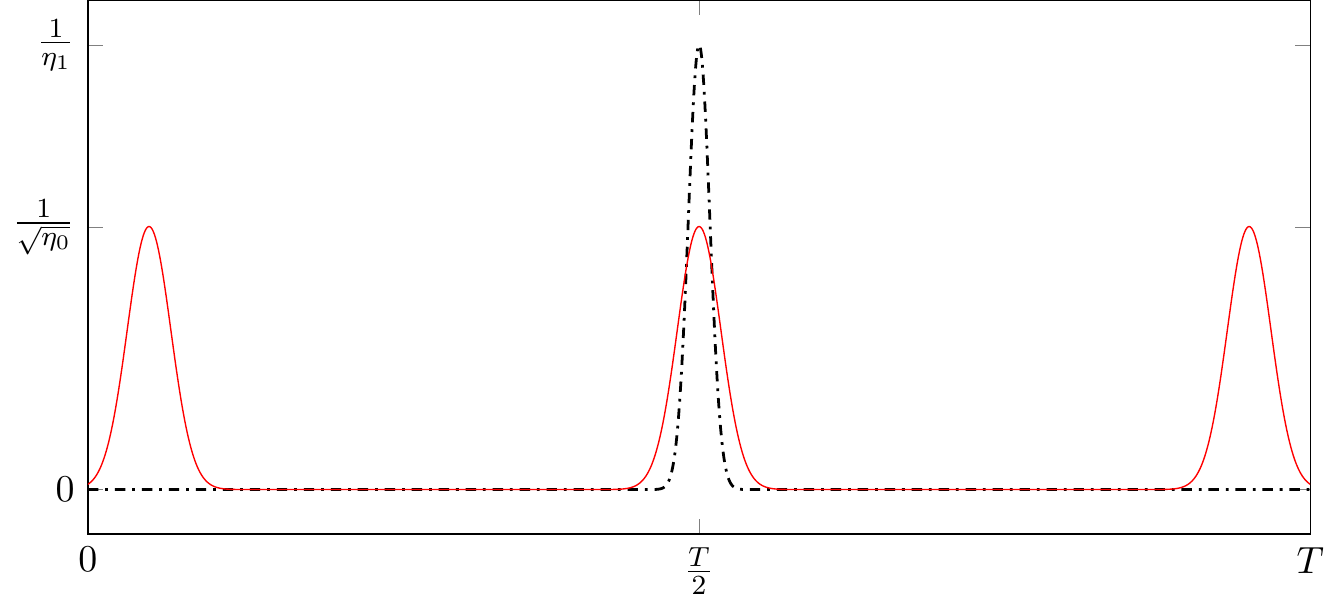}}
\caption{A comparison of the amplitudes of the controls $g^0$ (red continuous line) and $g^1$ (black dotted line) over the time interval $[0,T]$}
\label{controlsfig} 
\end{figure}

We recall that the forms of $g^0=g^0_{\eta_0}$ and $g^1=g^1_{\eta_1}$ were given in \eqref{gcontrol01}, respectively \eqref{gcontrol1}, where $ \beta_{\eta}\geq 0$, $ \|\beta_{\eta}\|_\infty=\mathcal{O}(1/\sqrt{\eta})$, $( \beta_{\eta}^2 )_\eta $  is supported in $[0,2\eta]$ and is an approximation of the unity when $\eta \to 0^{+}$.
The parameter $\eta_0>0$ was fixed at the end of the proof of Theorem \ref{maineu0} in Section \ref{th5h}, while the parameter $\eta_1>0$ was fixed (in function of the parameter $\nu>0$, roughly such that $\eta_1\leq \mathcal{O}(\nu)$) at the end of the proof of Theorem \ref{maineu1} in Section \ref{impctrlu1}. However, for the purpose of illustration in Figure \ref{controlsfig} and without loss of generality, we may assume that $0<\eta_1< \eta_0$.

\bigbreak

\textbf{A possibility of passing to arbitrary time via autoregularization}\label{autotime}

\bigbreak

Let us present a possibility for passing to arbitrary time in Theorem \ref{main}, i.e. deducing a control result in any given time $T>0$. Since we know how to control in small-time, one possibility would be to let the system evolve without control for a long time period and control during a short time interval at the end of $[0,T]$. The main technical issue with this strategy is the following. For our small-time controllability result Theorem \ref{main} to hold, we assume $H^4$ regularity on the initial data $u_0$. Intuitively, one would think that due to the smoothing properties of the Navier-Stokes equations, even with initial data that is only $L^2$, after a short time the solution would become $H^4$ and stay that way for all times until $T$. However, up to our knowledge no such results exist in the literature for fluid-solid interaction problems such as the one considered in this article. We have tried to establish such a result ourselves, but we have ran into some technical difficulties due to the fact that we are in a moving exterior domain with Navier conditions on the solid boundary, which we will precise below. 

So, for the time being let us formulate this autoregularization property as the following Open Problem.
\begin{Op}\label{regularizacio}
Given $u_0 \in \mathcal{H}(0)$, $T>0$, the associated weak solution $\bar{u} \in C([0,T]; \mathcal{H}(t)) \cap L^2((0,T);\mathcal{V}(t))$ in the sense of Definition \ref{weak} satisfies $(\bar{u},\bar{h}',\bar{\vartheta}')\in C((0,T];H^4(\mathcal{F}(t))\times\mathbb{R}^3)$.
\end{Op}

If Open Problem \ref{regularizacio} is proven, then the following result holds.

\begin{theorem} \label{bigmain}
Given $T>0$, $\mathcal{S}_0\subset\mathbb{R}^2$ bounded, closed, simply connected with smooth boundary, which is not a disk, 
and $u_{0}\in   \mathcal{H}(0)$, $\curl u_0 \in L^1(\mathcal{F}_0;\mathbb{R}^{2})$, 
$q_0=0, q_f =(h_{f},\vartheta_f) \in \mathcal{Q}$, 
$h'_0,h'_f\in\mathbb{R}^2,\vartheta'_0,\vartheta'_f\in\mathbb{R}$, such that \begin{gather*}
\div u_0=0 \text{ in }\mathcal{F}_0, \   
\lim_{|x|\to+\infty}|u_0(x)|=0,\\
u_0 \cdot n = (h'_0+\vartheta'_0 x^{\perp}) \cdot n,\ (D(u_0)n)_{\text{tan}} = - \mu (u_0- (h'_0+\vartheta'_0 x^{\perp}))_{\text{tan}} \text{ on } \partial \mathcal{S}_0.
\end{gather*}
Then there exists a control $\xi \in L^2((0,T)\times \Omega_c)$, compactly supported in time, and 
a weak solution $u \in C([0,T]; \mathcal{H}(t)) \cap L^2((0,T);\mathcal{V}(t))$ in the sense of Definition \ref{weak} to the system \eqref{eu}, \eqref{bc1}, \eqref{bc2}, \eqref{newt}, \eqref{ic} with \eqref{czone},
such that we have
$(h,h',\vartheta,\vartheta')(T)=(h_f,h'_f,\vartheta_f,\vartheta'_f)$. 
\end{theorem}

\begin{proof}

First we observe that $\tilde{T}>0$ in Theorem \ref{main} can be made uniform for $(h'_0,\vartheta'_0,u_0)$ in compact sets of $\mathbb{R}^3\times H^4(\mathcal{F}_0;\mathbb{R}^{2})$. Indeed, the initial data only comes into effect on the level of $u^1$ in our construction. However, it is easy to see that $(l^1,r^1,u^1)$ depend continuously on $(h'_0,\vartheta'_0,u_0)$, since they satisfy a linear equation (and therefore depend continuously on the initial data) and the control $g^1$ constructed in Section \ref{seu1} depends continuously on $(h'_0,\vartheta'_0)$. This regularity is carried over to the remainder terms in Section \ref{remain}, since they can be estimated by various norms of $(l^1,r^1,u^1)$, which are all uniformly bounded for  $(h'_0,\vartheta'_0,u_0)$ in compact sets of $\mathbb{R}^3\times H^4(\mathcal{F}_0;\mathbb{R}^{2})$. 

Using  Open Problem \ref{regularizacio} we deduce the existence of a weak solution $\bar{u} \in C([0,T]; \mathcal{H}(t)) \cap L^2((0,T);\mathcal{V}(t))$. For $\lambda>0$ we then consider the following compact set
$$S:=\{(l,r,v) \in \mathbb{R}^3\times H^4(\mathcal{F}_0;\mathbb{R}^{2}),\ |(l,r)-(\bar{h}'(T),\bar{\vartheta}'(T))|+\|v-\bar{u}(T,\cdot)\|_{H^4}\leq \lambda  \}.$$
It follows that we may apply Theorem \ref{main} with initial data in this set (without loss of generality we may assume that the initial position of the solid is once more the origin, since else one may simply apply a translation and a rotation to the system of coordinates in $\mathbb{R}^2$) to deduce $\tilde{T}>0$, which is now uniform for all initial data in $S$.

All that needs to be proven is that there exists $T_\lambda\in(T-\tilde{T},T)$ such that we have $(\bar{h}'(T_\lambda),\bar{\vartheta}'(T_\lambda),\bar{u}(T_\lambda,\cdot))\in S$,
since then we can let the system evolve without any control on $[0,T_\lambda]$ to ensure the regularization property of the equation, then use Theorem \ref{main} to deduce the existence of a control on $[T_\lambda, T]$ which drives the solid to the desired position and velocity. However, the existence of such a $T_\lambda$ follows from the regularity in Open Problem \ref{regularizacio}, therefore the proof of Theorem \ref{bigmain} follows.

\end{proof}

Let us give a few remarks regarding the difficulties in proving Open Problem \ref{regularizacio}. 

A direct possibility for the proof would be to consider $(\bar{u},\bar{h}',\bar{\vartheta}')$ as the solution of a Stokes system and use for instance the regularity of the Stokes operator with Navier conditions in an exterior domain proven in \cite{Far}, in particular the analyticity of the associated semigroup, to conclude the desired regularity. Of course, in order to do this, one transforms the PDE onto a fixed domain (with transformations similar to \eqref{chov}), and then put all the unwanted terms into an inhomogeneous force term acting on the Stokes equation, which one would then estimate. The problem with this approach is that the change of variables creates a term of the type $u_S \cdot \nabla u$ (as seen for instance in the PDE \eqref{eucs}), which is not even $L^2$ when $u\in H^1$. Note that this is not an issue in the Galerkin method presented in Section \ref{remain} or in \cite{PS} for instance, since the trilinear term $\int_{\mathcal{F}_0}(u-u_S) \cdot \nabla u \cdot u \, dx$ can be defined as $0$ as a limit of approximations.

Another, more ``manual'' approach would be to bootstrap the regularity by classical Navier-Stokes methods. One obtains that after arbitrarily small time, the solution becomes $H^1$, and then using some appropriate test functions, the $H^2$ regularity can also be obtained after some manipulation of the weak solutions. However, for the $H^3$ regularity to hold, one usually differentiates the equation with respect to $t\in[0,T]$, proves an energy estimate for $\partial_t \bar{u}$ and then concludes by the means of a stationary Stokes problem (which in theory could be also done on a moving domain, without a change of variables). And it is at this point where this strategy breaks down in our setting. If one differentiates the equation on $\mathcal{F}(t)$, without switching to a fixed domain, then higher order unwanted terms appear in the Navier boundary conditions, which pose a problem when trying to estimate them in the energy inequality (we contrast this with the Dirichlet case, for instance in \cite{BG}, which is more robust to differentiation with respect to time, and where such a strategy worked without issue). On the other hand if one switches to the fixed domain $\mathcal{F}_0$ to avoid this issue, then the term $\partial_t u_S \cdot \nabla u$ will be created, which will pose the same problems as mentioned above.

We plan to investigate this problem in a future article, and perhaps by a more lengthy and technical approach of considering a different change of variables, which behaves like a rigid movement near the solid, but behaves like the identity operator outside of a large enough ball (as done for instance in \cite{GGH} for the Dirichlet case, where the authors only went as far as $H^2$ regularity) some sort of breakthrough could be made. 
This echoes the issues presented in Remark \ref{plane} for the bounded case, since such a change of variables would still create extra terms on the PDE which need to be estimated, but since this step is separate from the asymptotic expansion, perhaps those terms can be handled with sufficiently strong estimates, similarly to \cite{GGH}.

%
%%%
%

%
%
%

\appendix
\renewcommand*{\thesection}{\Alph{section}}
 %%%%%
 %%%

 \section{Design of the controls}
 \label{comproofs}
 
 \subsection{Proof of Proposition \ref{farcontr0}}
 \label{comproof1}
We will use the notation of Section 7.1 of \cite{GKS} for the rest of this proof, in contrast to the previous notations of this paper. Therefore, for the sake of simplicity, we will be working in the case when the solid is a homogeneous disk, however, using the same methods as in Section 7.2 of \cite{GKS}, our construction can be adapted to the general case as well. 
We recall from \eqref{tilmc} that we had $\tilde{B}_c\subset B_c$ and, for $q\in\mathcal Q_\delta$, we defined
$$ \mathcal C(q) = \left\{  g \in    C_{0}^{\infty}( R(\vartheta)^T(B_c-h)  ;\mathbb{R}) : \, \int g =0 \right\},\quad \tilde{\mathcal{C}} (q)= \left\{  g \in    C_{0}^{1}(R(\vartheta)^T(\tilde{B}_c-h)  ;\mathbb{R}): \, \int g =0 \right\}.$$

Let us give the corresponding result in the case of a homogeneous disk.
\begin{proposition} \label{dfarcontr0}
There exists a $C^1$ mapping $\overline{g}^0:\left(\mathcal{Q}_\delta \cap \{q=(h,0)\in\mathbb{R}^3 \} \right) \times \mathbb R^2 \rightarrow \tilde{\mathcal C}(\mathcal{Q}_\delta)$ 
such that for any $q=(h,0)\in\mathcal{Q}_\delta$ we have $\text{Range}(\overline{g}^0(q,\cdot))\subset \mathcal C(q) \cap \tilde{\mathcal{C}}(q)$, and  for any $((h,0),v) $ in $\mathcal{Q}_\delta \times \mathbb R^2$ 
the function $\overline{\alpha}^0 := \mathcal A [q,\overline{g}^0 (q,v)]$ in $C^\infty (\overline{\mathcal{F}_0};\mathbb{R})\cap C^6_{\infty}(\overline{\mathcal{F}_0};\mathbb{R})$  satisfies:
\begin{gather}
\Delta \overline{\alpha}^0 =0\ \text{in}\  \mathcal{F}_0 \setminus (\tilde{B}_c -h) , \lim_{|x|\to+\infty}|\nabla\overline{\alpha}^0|=0  \text{ and } 
\partial_{n} \, \overline{\alpha}^0=0\ \text{on}\ \partial\mathcal{S}_0 ,\\
\int_{ \partial \mathcal{S}_0} \left|  \nabla\overline{\alpha}^0\right|^{2} \, n\, d \sigma = v , \\
\int_{\partial\mathcal{S}_0} \overline{\alpha}^0 \, n \, d\sigma = 0 ,\\
\text{span}\left\{n(x),\ x\in\text{supp } \nabla\overline{\alpha}^0(q,\cdot) \cap \partial\mathcal{S}_0\right\} = \mathbb{R}^2. \label{fcspd}
\end{gather}
\end{proposition}
 
We follow the same construction as in Section 7.1 of \cite{GKS}. We have the following Lemma, which appears as Lemma 9 in \cite{GKS}, and also applies in our case without any modification, we refer the reader to the aforementioned article for the proof.

\begin{lemma} \label{kup1}
There exist three vectors $e_1, e_{2}, e_{3} \in \{n (x):\ x\in\partial\mathcal{S}_0\}$
and positive $C^\infty$ maps $(\mu_i)_{1\leqslant i \leqslant 3} : \R^{2} \rightarrow [1,+\infty)$
such that for any $v \in \R^{2}$,
\begin{equation} \label{SommeDesMu}
\sum_{i=1}^3 \mu_i (v) e_i = v.
\end{equation}
\end{lemma}
Note that 
there exist $x_i\in\partial\mathcal{S}_0$ such that $n(x_i)=e_i$, $i=1,2,3$.
We will now modify Lemma 10 from \cite{GKS} in order to guarantee \eqref{fcspd}.
We have the following result.
\begin{lemma} \label{LemBase}
There exist families of functions $(\tilde{\alpha}_{\varepsilon}^{i,j})_{\varepsilon \in (0,1)}$, $i,j\in\{1,2,3\}$, such that 
 for any  $i,j\in\{1,2,3\}$, for any $\varepsilon \in (0,1)$, $\tilde{\alpha}_{\varepsilon}^{i,j}$ is defined and 
harmonic in a closed neighbourhood $ \mathcal{V}_{\varepsilon}^{i,j}$ of $\partial\mathcal{S}_0$,
satisfies $\partial_{n} \tilde{\alpha}_{\varepsilon}^{i,j} =0$ on $\partial\mathcal{S}_0$,
and moreover one has the following:
\begin{itemize}
\item[(i)] for any  $i,j,k,l$ in $\{1,2,3\}$,
$$
\int_{\partial\mathcal{S}_0} \nabla\tilde{\alpha}_{\varepsilon}^{i,j} \cdot \nabla\tilde{\alpha}_{\varepsilon}^{k,l} \,n\, d\sigma
\rightarrow  \delta_{(i,j),(k,l)} \, e_i   \quad \text{ as } \eps  \rightarrow 0^+ ; $$
\item[(ii)] for any  $i,j,k,l$ in $\{1,2,3\}$, $(i,j)\neq (k,l)$,
$$\| \nabla\tilde{\alpha}_{\varepsilon}^{i,j} \cdot \nabla\tilde{\alpha}_{\varepsilon}^{k,l} \|_{C(\partial\mathcal{S}_0)}\rightarrow 0   \quad \text{ as } \eps  \rightarrow 0^+;$$
\item[(iii)] for any  $i,j$ in $\{1,2,3\}$ there exist $x^{i,j}_\varepsilon\in\partial\mathcal{S}_0$ such that $|\nabla\tilde{\alpha}_{\varepsilon}^{i,j}(x^{i,j}_\varepsilon)|>1$ and $x^{i,j}_\varepsilon \to x_i$ as $\varepsilon\to 0^+$.
\end{itemize}
\end{lemma}
\begin{proof}
The proof is essentially the same as that of Lemma 10 from \cite{GKS}, noting that the manner in which the functions $\beta_\varepsilon^{i,j}$ were constructed in the aforementioned proof allows us to deduce points (ii) and (iii). The details are left to the reader.\end{proof}

Next we adapt Lemma 11 from \cite{GKS} to our setting. We have the following result.
\begin{lemma} \label{LemmeEta}
Let $q=(h,0) \in\mathcal{Q}_\delta.$
There exists a family of functions $(\alpha_{\eta}^{i,j}(q,\cdot))_{\eta \in (0,1)}$, $i,j \in \{1,2,3\}$,  harmonic in $\mathcal{F}_0 \setminus  (\tilde{B}_c -h)$, 
satisfying $\displaystyle\lim_{|x|\to+\infty}|\nabla \alpha_{\eta}^{i,j}(q,\cdot)|=0$,  $\partial_n \alpha_{\eta}^{i,j}(q,\cdot)=0$ on  $\partial\mathcal{S}_0$ and $\alpha_{\eta}^{i,j}(q,\cdot)\in C^6_\infty(\overline{\mathcal{F}_0})\cap C^{\infty}(\overline{\mathcal{F}_0})$, such that, for any $k$ in $\mathbb N$, 
\begin{equation} \label{del2new}
\| \alpha_{\eta}^{i,j}(q,\cdot) - \tilde{\alpha}_{\varepsilon}^{i,j}(\cdot) \|_{C^k(\mathcal{V}_{\varepsilon}^{i,j}\cap\overline{\mathcal{F}_0})} 
\rightarrow   0     \text{ when } \eta \rightarrow  0^+ .
\end{equation}
\end{lemma}
\begin{proof}
We construct $\alpha_{\eta}^{i,j}$ from $\tilde{\alpha}_{\varepsilon}^{i,j}$ using an approximation by rational functions, as mentioned in \cite{GKS} and as used in \cite{Cortona}, pages 147-149.
We use the following generalization of Runge's theorem on the Riemann sphere (see for instance \cite{Gau}, page 238).
\begin{theorem} 
Let $K$ be a compact subset of the Riemann sphere $\overline{\mathbb{C}}$, containing at least two points, and let $P$ be a subset of $\overline{\mathbb{C}}$. A necessary and sufficient condition in order that for each holomorphic function $f$ on $K$ and each $\eta>0$, there be a rational function $r$, whose poles lie in $P$, such that $\|f-r\|_{C^{k}(K)}<\eta$, for  any $k\geq 0$, is that $P$ meet each connected component of $\overline{\mathbb{C}}\setminus K$.
\end{theorem}
We use the above result with $K=\left(\overline{\mathbb{C}}\setminus B(0,R) \right)\cup \mathcal{V}_{\varepsilon}^{i,j}$, for some $R>0$ large enough such that
$\left( (\tilde{B}_c -h)\cup \mathcal{V}_{\varepsilon}^{i,j}\right)\subset B(0,R)$, set $f=0$ on $\overline{\mathbb{C}}\setminus B(0,R)$ and $f=\partial_{x_1} \tilde{\alpha}_{\varepsilon}^{i,j}-i\partial_{x_2}\tilde{\alpha}_{\varepsilon}^{i,j} $ on $\mathcal{V}_{\varepsilon}^{i,j}$, and choose $P$ such that it is made up of a point from $\text{int }(\tilde{B}_c -h)$ and a point from $\left(\text{int }\mathcal{S}_0 \right)\setminus \mathcal{V}_{\varepsilon}^{i,j}$. We conclude as in \cite{Cortona} that there exists some function $\theta$ (ignoring the dependence on the parameters for this notation) which is harmonic on $\mathcal{F}_0\setminus  (\tilde{B}_c -h)$, $\displaystyle\lim_{|x|\to+\infty}\nabla\theta(x)=C_\theta \in \mathbb{R}^2$, $|C_\theta|=\mathcal{O}(\eta)$, and $\partial_n \theta=\mathcal{O}(\eta)$ on $\partial\mathcal{S}_0$ in any $C^k$ norm, such that $\theta\in C^k_\infty(\overline{\mathcal{F}_0})$ and
$$\|\theta - \tilde{\alpha}_{\varepsilon}^{i,j} \|_{C^k(\mathcal{V}_{\varepsilon}^{i,j}\cap\overline{\mathcal{F}_0})} 
\rightarrow   0     \text{ when } \eta \rightarrow  0^+ ,\text{ for any }k\geq 0.$$

Proceeding as in \cite{Cortona}, we introduce a corrector function, to ensure the vanishing boundary conditions of our potential flow, by setting
\begin{align*}
\Delta\Psi = 0 \text{ on }\mathcal{F}_0,\ \lim_{|x|\to+\infty}\nabla\Psi(x)=C_\theta, \text{ and }\partial_n\Psi=\partial_n \theta \text{ on } \partial\mathcal{S}_0.
\end{align*}
Note that the function $\theta-\Psi$ satisfies all the conclusions of Lemma \ref{LemmeEta}, except the smoothness on $\mathcal{F}_0$, since it has a singularity in $\tilde{B}_c -h$ as per Runge's theorem above. To address this issue, we use Whitney's extension theorem (see for instance \cite{Bru}, Chapter 2) to deduce that for any function $F\in C^{\infty}\left(\overline{\mathcal{F}_0\setminus   (\tilde{B}_c -h)}\right)$ there exists a smooth extension $E(F)\in C^{\infty}(\overline{\mathcal{F}_0})$ satisfying $E(F)=F$ on $\overline{\mathcal{F}_0\setminus   (\tilde{B}_c -h)}$.
We conclude the proof by setting
$\alpha_{\eta}^{i,j}(q,\cdot)=E(\theta-\Psi)$.

\end{proof}

We have the following result, which corresponds to the adaptation of Lemma 12 from \cite{GKS} to our setting.
\begin{lemma} \label{Lem2emeCond}
For any $\nu>0$, there exist $C^1$ mappings
$q=(h,0) \in\mathcal{Q}_\delta \mapsto \overline{\alpha}^{i,j} (q,\cdot)\in C^{2}_\infty(\overline{\mathcal{F}_0})$, 
$i,j\in\{1,2,3\}$, such that for any $q=(h,0) \in\mathcal{Q}_\delta$,
$\Delta_{x} \overline{\alpha}^{i,j} (q,\cdot)=0$ in $\mathcal{F}_0 \setminus   (\tilde{B}_c -h)$, $\displaystyle\lim_{|x|\to+\infty}|\nabla \ \overline{\alpha}^{i,j}(q,\cdot)|=0$,  $\partial_n  \overline{\alpha}^{i,j}(q,\cdot)=0$ on  $\partial\mathcal{S}_0$, $ \overline{\alpha}^{i,j}(q,\cdot)\in C^6_\infty(\overline{\mathcal{F}_0})\cap C^{\infty}(\overline{\mathcal{F}_0})$, and the following hold:
\begin{itemize}
\item[(i)] for any  $i,j,k,l$ in $\{1,2,3\}$,
$$\left|  \int_{\partial\mathcal{S}_0} \nabla \overline{\alpha}^{i,j}(q,\cdot) \cdot \nabla \overline{\alpha}^{k,l}(q,\cdot) \,n\, d\sigma
- \delta_{(i,j),(k,l)} \, e_i  \right| \leq \nu;$$
\item[(ii)] for any $i,j,k,l$ in $\{1,2,3\}$,  $(i,j)\neq (k,l)$,
$$\| \nabla \overline{\alpha}^{i,j}(q,\cdot) \cdot \nabla \overline{\alpha}^{k,l}(q,\cdot) \|_{C(\partial\mathcal{S}_0)}\leq \nu;$$
\item[(iii)] for any  $i,j$ in $\{1,2,3\}$ there exist $x^{i,j}\in\partial\mathcal{S}_0$ such that $|\nabla\overline{\alpha}^{i,j}(x^{i,j})|>1$ and $|x^{i,j}-x_i|\leq \nu$.
\end{itemize}
\end{lemma}
\begin{proof}
The proof of Lemma 12 from \cite{GKS} can be easily adapted to our setting, simply by considering an elliptic problem of the type \eqref{pot} instead of the Neumann problem mentioned in its proof. It can be easily checked that the solutions obtained this way are still in $ C^6_\infty(\overline{\mathcal{F}_0})$ and that the map $q=(h,0) \in\mathcal{Q}_\delta \mapsto \overline{\alpha}^{i,j} (q,\cdot)\in C^2_{\infty}(\overline{\mathcal{F}_0})$ is $C^1$. Finally, points (ii) and (iii) follow from Lemma \ref{LemBase} and Lemma \ref{LemmeEta}.
\end{proof}

Next we have the following result, which corresponds to Lemma 13 from \cite{GKS}.
\begin{lemma} \label{Lem3emeCond}
For any $\nu>0$, there exist $C^1$ mappings
$q=(h,0) \in\mathcal{Q}_\delta \mapsto \overline{\alpha}^{i} (q,\cdot)\in C_\infty^{2}(\overline{\mathcal{F}_0})$, 
$i\in\{1,2,3\}$, such that for any $q=(h,0) \in\mathcal{Q}_\delta$,
$\Delta_{x} \overline{\alpha}^{i} (q,\cdot)=0$ in $\mathcal{F}_0 \setminus   (\tilde{B}_c -h)$, $\displaystyle\lim_{|x|\to+\infty}|\nabla \ \overline{\alpha}^{i}(q,\cdot)|=0$,  $\partial_n  \overline{\alpha}^{i}(q,\cdot)=0$ on  $\partial\mathcal{S}_0$, $ \overline{\alpha}^{i}(q,\cdot)\in C^6_\infty(\overline{\mathcal{F}_0})\cap C^{\infty}(\overline{\mathcal{F}_0})$, and the following hold:
\begin{itemize}
\item[(i)] for any  $i,j$ in $\{1,2,3\}$,
$$\left|  \int_{\partial\mathcal{S}_0} \nabla \overline{\alpha}^{i}(q,\cdot) \cdot \nabla \overline{\alpha}^{j}(q,\cdot) \,n\, d\sigma
- \delta_{i,j} \, e_i  \right| \leq \nu , \quad 
\int_{\partial\mathcal{S}_0}   \overline{\alpha}^{i}(q,\cdot) \,n\, d\sigma  = 0;$$
\item[(ii)] for any $i,j$ in $\{1,2,3\}$,  $i\neq j$,
$$\| \nabla \overline{\alpha}^{i}(q,\cdot) \cdot \nabla \overline{\alpha}^{j}(q,\cdot) \|_{C(\partial\mathcal{S}_0)}\leq \nu;$$
\item[(iii)] for any  $i$ in $\{1,2,3\}$ there exist $\overline{x}^i\in\partial\mathcal{S}_0 $ such that $|\nabla \overline{\alpha}^{i}(q,\overline{x}^i)|^2>1-\nu$ and $|\overline{x}^i-x_i|\leq \nu$.
\end{itemize}
\end{lemma}
\begin{proof}
Consider the functions $\overline{\alpha}^{i,j}$ given by Lemma \ref{Lem2emeCond}.
For any $q=(h,0) \in\mathcal{Q}_\delta$, for any  $i\in\{1,2,3\}$, the three vectors 
$ \int_{\partial\mathcal{S}(q)} \overline{\alpha}^{i,j}(q,\cdot) \, n \, d\sigma$, where $j\in\{1,2,3\}$, are 
linearly dependent in  $\mathbb R^2$; therefore there exist $\lambda^{i,j}(q)\in\mathbb{R}$ such that
\begin{equation} \label{fucklasncf}
\sum_{j=1}^{3} \lambda^{i,j}(q) \int_{\partial\mathcal{S}(q)} \overline{\alpha}^{i,j}(q,\cdot) \, n \, d\sigma = 0
\text{ and } \displaystyle\sum_{j=1}^{3} |\lambda^{i,j}(q)|^2 = 3,
\end{equation}
Then one defines $\overline{\alpha}^{i}(q,\cdot):=\sum_{j=1}^{3} \lambda^{i,j}(q)  \overline{\alpha}^{i,j}(q,\cdot)$, and one checks that it satisfies (i) and (ii) with some $C \nu$ in the right hand side. 

On the other hand, for each $i\in\{1,2,3\}$ let $k\in\displaystyle\text{arg}\max_{j} |\lambda^{i,j}(q)|^2$, so that we have $|\lambda^{i,k}(q)|^2\geq 1$.
Using (ii) and (iii) from Lemma \ref{Lem2emeCond} it follows that
\begin{align*}
|\nabla \overline{\alpha}^{i}(q,x^{i,k})|^2  \geq  |\lambda^{i,k}(q)|^2 |\nabla \overline{\alpha}^{i,k}(q,x^{i,k})| ^2 - \tilde{C} \nu > 1  - \tilde{C} \nu,
\end{align*}
for some $\tilde{C}>0$.

Changing $\nu$ in $\min\{\nu/C,\nu/\tilde{C},\nu\}$ allows to conclude.
\end{proof}

Before concluding the proof of Proposition \ref{dfarcontr0}, let us show how condition \eqref{fcspd} can be satisfied.
Let $v\in\mathbb{R}^2$, using Lemma \ref{kup1} we set
$$
\tilde{\alpha}(q,\cdot):=\sum_{i=1}^{3}\sqrt{\mu^{i}(v)} \, \overline{\alpha}^{i}(q,\cdot).
$$
 It follows from (ii) and (iii) in Lemma \ref{Lem3emeCond}, and the fact that $\mu^{i}(v)\geq 1$, $i\in\{1,2,3\}$, that 
\begin{align*}
|\nabla\tilde{\alpha}(q,\overline{x}^i)|^2 \geq |\nabla\overline{\alpha}^i(q,\overline{x}^i)|^2 - C \nu >1 -(C+1) \nu,
\end{align*}
and $|\overline{x}^i-x_i|\leq \nu$,
for any $i\in\{1,2,3\}$. Since $\text{span}\{n(x_i),\ i=1,2,3\}=\mathbb{R}^2$, it follows by continuity that, for $\nu>0$ small enough, we have $\text{span}\{n(\overline{x}^i),\ i=1,2,3\}=\mathbb{R}^2$, and $\overline{x}^i\in\text{supp } \nabla\tilde{\alpha}(q,\cdot)$, for any $i\in\{1,2,3\}$. Therefore, \eqref{fcspd} holds for $\tilde{\alpha}$.

To conclude the proof of Proposition \ref{dfarcontr0}, we proceed as at the end of the proof of Proposition 7 from \cite{GKS}, namely that when $\nu>0$ is small enough, one may apply a local inversion argument. The details are left to the reader. 

\begin{remark}[A remark on the case when the solid is not a disk]
Lemma 15 from \cite{GKS} can be improved in the following way:
there exist vectors $e_i\in \{\partial_n \Phi (x):\ x\in\partial\mathcal{S}_0\}$, $1\leqslant i \leqslant 16$, a constant $M>0$,
and positive $C^\infty$ maps $(\mu_i)_{1\leqslant i \leqslant 16} : \R^{2} \rightarrow [M,+\infty)$
such that for any $v \in \R^{3}$,
$$
\sum_{i=1}^{16} \mu_i (v) e_i = v.
$$
This allows us to prove \eqref{fcsp} similarly as we have done above in the case of a disk.
\end{remark}

%%%%%

 \subsection{Proof of Proposition \ref{farcontr1}}
 \label{comproof2}
 
Again, for simplicity we will present the proof in the case when the solid is assumed to be a homogeneous disk, then one can deduce the general case in a similar manner as in Section 7.2 from \cite{GKS}, which will be explained at the end of the section. Furthermore, the notations of this section will be self-contained. 

We prove the following adaptation of Proposition \ref{farcontr1} to the case of a a homogeneous disk.
\begin{proposition} \label{dfarcontr1}
Let $\mathcal{K}$ be a compact subset of $C^1(\partial\mathcal{S}_0;\mathbb{R}^2)$ such that for any $V\in\mathcal{K}$ we have $V \cdot n=0$ on $\partial\mathcal{S}_0$ and
\begin{align}\label{dH01}
\text{span}\left\{n(x),\ x\in\partial\mathcal{S}_0 \cap \text{supp } V \right\}=\mathbb{R}^2.
\end{align}
For any $V\in\mathcal{K}$, there exists a continuous mapping $\overline{g}^1[V]:\left(\mathcal{Q}_\delta \cap \{q=(h,0)\in\mathbb{R}^3 \} \right) \times \mathbb R^2 \rightarrow \tilde{\mathcal C}(\mathcal{Q}_\delta)$ 
such that  for any $q=(h,0)\in\mathcal{Q}_\delta$ we have $\text{Range}(\overline{g}^1[V](q,\cdot))\subset \mathcal C(q) \cap \tilde{\mathcal{C}}(q)$, and for any $(q,v) $ in $\mathcal{Q}_\delta \times \mathbb R^2$, $q=(h,0)$,
the function $\overline{\alpha}^1 := \mathcal A [q,\overline{g}^1[V] (q,v)]$ in $C^\infty (\overline{\mathcal{F}_0};\mathbb{R})\cap C^3_{\infty}(\overline{\mathcal{F}_0};\mathbb{R})$  satisfies:
\begin{gather}
\Delta \overline{\alpha}^1 =0\ \text{in}\  \mathcal{F}_0 \setminus  (\tilde{B}_c -h), \lim_{|x|\to+\infty}|\nabla\overline{\alpha}^1|=0   \text{ and } 
\partial_{n} \, \overline{\alpha}^1=0\ \text{on}\ \partial\mathcal{S}_0,\\
\int_{ \partial \mathcal{S}_0}   \nabla\overline{\alpha}^1 \cdot V  \, n \, d \sigma = v .\label{clin}
\end{gather}
Furthermore, the map $V\in \mathcal{K} \mapsto \overline{g}^1[V]\in C(\mathcal{Q}_\delta \times \mathbb R^2 ; \mathcal C)$ is also continuous.
\end{proposition}

We may suppose without loss of generality that $\mathcal{S}_0$ is the unit disk, parametrized by $c(s)=(\cos(s),\sin(s))$. We have the following geometrical property.
\begin{lemma}\label{bigeo}
Given $V\in\mathcal{K}$,  there exist $x_i\in \partial\mathcal{S}_0 \cap \text{supp } V$, $s_i\in[0,2\pi]$, such that $n(x_i)=x_i=(\cos(s_i),\sin(s_i))$, $i=1,2$, and that $\text{span}\{b_1,b_2\}=\mathbb{R}^2,$ where 
\begin{align}\label{bi}
b_i= \tau(x_i) \cdot V(x_i)\, n(x_i) -\frac{1}{\pi}\int_{\partial\mathcal{S}_0} \tau(x) \cdot V(x) \,n(x)\, d\sigma. 
\end{align}
\end{lemma}

\begin{proof}
For each $i\in\{1,2\}$, using \eqref{dH01} and $V \cdot n =0$ on $\partial\mathcal{S}_0$, there exist $\bar{x}_i=(\cos(\bar s_i),\sin(\bar s_i))$ such that we have $\tau(\bar x_i) \cdot V( \bar x_i) \neq 0$ and 
\begin{align}\label{bispan}
\text{span}\{\tau(\bar x_i) \cdot V(\bar x_i)\, n( \bar x_i),\ i=1,2\}=\mathbb{R}^2.
\end{align}
Clearly, $\displaystyle \frac{1}{\pi}\int_{\partial\mathcal{S}_0} \tau(x) \cdot V(x) \,n(x)\, d\sigma$ does not depend on the choice of $\bar x_i$, so either $\displaystyle\bar b_i=\tau(\bar x_i) \cdot V(\bar x_i)\, n( \bar x_i)-\frac{1}{\pi}\int_{\partial\mathcal{S}_0} \tau(x) \cdot V(x) \,n(x)\, d\sigma,\ i=1,2$, are collinear or they span $\mathbb{R}^2$. 

If they are collinear, we may modify the one out of the two which has the smaller norm (in order to also handle the case in which one of them is zero) in the following way in order to break the collinearity. Without loss of generality we may suppose that $|\bar b_1|\leq |\bar b_2|$. We observe that \eqref{bispan} is robust to perturbations, due to \eqref{dH01} and the continuity of $V$, i.e. there exists $\eta>0$ such that for any $s\in (\bar s_1-\eta,\bar s_1+\eta)$ we have
$$\text{span}\{\tau(c(s)) \cdot V(c(s))\, n( c(s)),\tau(\bar x_2) \cdot V(\bar x_2)\, n( \bar x_2)\}=\mathbb{R}^2.$$
It follows that there exists some $x_1=(\cos(s_1),\sin(s_1))$ with $s_1\in (\bar s_1-\eta,\bar s_1+\eta)$, such that setting $x_2=\bar x_2$, the vectors $b_i$ given by \eqref{bi} are no longer collinear.
\end{proof}

We will base our construction on the existence of such vectors $b_i$ given by Lemma \ref{bigeo}, in order to satisfy \eqref{clin}.
We have the following result.
\begin{lemma} \label{LemBase1} 
There exist families of functions $(\tilde{\alpha}_{\varepsilon}^{i})_{\varepsilon \in (0,1)}$, $i\in\{1,2\}$, such that 
 for any  $i\in\{1,2\}$, for any $\varepsilon \in (0,1)$, $\tilde{\alpha}_{\varepsilon}^{i}$ is defined and 
harmonic in a closed neighbourhood $ \mathcal{V}_{\varepsilon}^{i}$ of $\partial\mathcal{S}_0$,
satisfies $\partial_{n} \tilde{\alpha}_{\varepsilon}^{i} =0$ on $\partial\mathcal{S}_0$,
and moreover one has the following:
$$
\int_{\partial\mathcal{S}_0} \nabla\tilde{\alpha}_{\varepsilon}^{i} \cdot V \,n\, d\sigma
\rightarrow  b_i   \quad \text{ as } \eps  \rightarrow 0^+ . $$
\end{lemma}
\begin{proof}
For each $i\in\{1,2\}$ we consider families of smooth functions $\delta_{\varepsilon}^{i}\in C_0^\infty( (0,2\pi) ; \mathbb{R})$  $\varepsilon \in (0,1)$, such that $\displaystyle\int \delta^i_\varepsilon = 1$,
$\text{diam}\left(\text{supp }\delta_{\varepsilon}^{i,j}\right)=C \varepsilon$, and
\begin{equation*}
\left| \int_{0}^{2\pi} \delta_{\varepsilon}^{i}(s) \, f(s) ds- f(s_i)\, ds\right| \leq C \|f'\|_\infty \varepsilon,
\end{equation*}
for any $f \in C^1([0,2\pi])$.
Then we define $\tilde{\alpha}_{\varepsilon}^{i,j}$ in polar coordinates as the truncated Laurent series:
\begin{align}\label{laurents}
\tilde{\alpha}_{\varepsilon}^{i}(r,\theta) := \frac12 \sum_{0<k\leq K} \frac{1}{k} \left(r^k+\frac{1}{r^k}\right)( -\hat{b}_{k,\varepsilon}^{i} \cos(k\theta) +\hat{a}_{k,\varepsilon}^{i}  \sin(k\theta)) , 
\end{align}
where $\hat{a}_{k,\varepsilon}^{i}$ and $\hat{b}_{k,\varepsilon}^{i}$ denote the $k$-th Fourier coefficients of the function $\delta_{\varepsilon}^{i} $.
It is elementary to check that the function $\tilde{\alpha}_{\varepsilon}^{i}$ satisfies the required properties for an appropriate choice of $K$. In particular, 
noting that $\nabla_x \tilde{\alpha}_{\varepsilon}^{i}=\partial_\tau \tilde{\alpha}_{\varepsilon}^{i}\, \tau$ on $\partial \mathcal{S}_0$, 
for each $\theta\in[0,2\pi]$ we have
\begin{align*}
\partial_\tau \tilde{\alpha}_{\varepsilon}^{i} (\cos(\theta),\sin(\theta)) = \partial_\theta \tilde{\alpha}_{\varepsilon}^{i} (1,\theta) = \delta^{i}_\varepsilon(\theta) -\frac{1}{\pi} \int_0^{2\pi} \delta^i_\varepsilon(\theta)\, d\theta + \mathcal{O}(\varepsilon),
\end{align*}
and due to the properties of $\delta^i_\varepsilon$, this implies
\begin{align*}
\left| \int_{\partial\mathcal{S}_0} \nabla\tilde{\alpha}_{\varepsilon}^{i} \cdot  V \,n\, d\sigma - b_i \right| \leq C\left(\|V \|_{C^1(\partial\mathcal{S}_0)} +1\right) \varepsilon.
\end{align*}
\end{proof}

We combine the methods used to prove Lemma \ref{LemmeEta} and Lemma \ref{Lem2emeCond} from Section \ref{comproof1} to get 
 the following result.

\begin{lemma} \label{Lem2emeCond1}
There exist continuous mappings
$(q,V)\in\left(\mathcal{Q}_\delta \cap \{q=(h,0)\in\mathbb{R}^3 \} \right)\times\mathcal{K} \mapsto \overline{\alpha}_{i} [V](q,\cdot)\in C^{\infty}(\overline{\mathcal{F}_0})$, such that for any $(q,V)\in\mathcal{Q}_\delta\times\mathcal{K}$, $q=(h,0)$, we have
$\Delta_{x} \overline{\alpha}_{i}[V] (q,\cdot)=0$ in $\mathcal{F}_0 \setminus   (\tilde{B}_c -h)$, $\displaystyle\lim_{|x|\to+\infty}|\nabla \overline{\alpha}_{i}[V](q,\cdot)|=0$,  $\partial_n  \overline{\alpha}_{i}[V](q,\cdot)=0$ on  $\partial\mathcal{S}_0$, $ \overline{\alpha}_{i}[V](q,\cdot)\in C^3_\infty(\overline{\mathcal{F}_0})$, $i\in\{1,2\}$, and 
\begin{align}\label{alphaspan}
\text{span}\left\{ \int_{\partial\mathcal{S}_0} \nabla \overline{\alpha}_{i}[V](q,\cdot)  \cdot V \,n\, d\sigma,\ i=1,2 \right\}=\mathbb{R}^2.
\end{align}

\end{lemma}
\begin{proof}
The proof follows the same principles as the proofs of Lemma \ref{LemmeEta} and Lemma \ref{Lem2emeCond} from Section \ref{comproof1}. The only difference is that
instead of using a compact covering and partition of unity argument for $q\in\mathcal{Q}_\delta$, we use one for the pair $(q,V)\in\mathcal{Q}_\delta \times \mathcal{K}$ by observing that if $\overline{\alpha}_i[V]$ satisfy \eqref{alphaspan} for some fixed $V \in \mathcal{K}$, then there exists $\delta_V>0$ such that
$$\text{span}\left\{ \int_{\partial\mathcal{S}_0} \nabla \overline{\alpha}_{i}[V](q,\cdot)  \cdot \tilde{V} \,n\, d\sigma,\ i=1,2 \right\}=\mathbb{R}^2.$$
holds for any $\tilde{V}\in \mathcal{K}$ with $\|V-\tilde{V}\|_{C^{1}} \leq \delta_V.$

This allows us 
to make our construction continuous with respect to $V$.
 
\end{proof}

We conclude the proof of Proposition \ref{dfarcontr1} by using Lemma \ref{Lem2emeCond1} above 
to deduce  $\overline{\alpha}_{i}(q,\cdot)$, $i=1,2$.
For each $v \in\mathbb{R}^2$ we may obtain $\lambda_i[V] (q,v)\in\mathbb{R}$, $i=1,2,$ such that
$$\overline{\alpha}^1[V](q,\cdot) =\lambda_1[V] (q,v) \overline{\alpha}_{1}[V](q,\cdot)  +\lambda_2[V] (q,v) \overline{\alpha}_{2}[V](q,\cdot),$$
which satisfies the required properties of Proposition \ref{dfarcontr1}, in particular the regularity with respect to $q$, $v$ and $V$ follows from the construction above and the fact that $\lambda_i$, $i=1,2,$ are the solutions of a linear system whose coefficients are regular with respect to the aforementioned parameters.

%%%%%

\begin{remark}[The case when $\mathcal{S}_0$ is not a disk]
We may follow a similar construction as in Section 7.2 of \cite{GKS} to reduce the general case to the case of a disk, by using a conformal mapping $\Psi:\overline{\mathbb{C}}\setminus B(0,1)\to \overline{\mathbb{C}}\setminus\mathcal{S}_0$.  The key observation is that the condition $\tau(x) \cdot V(x)\neq 0$ for $x\in\partial\mathcal{S}_0\cap\text{supp }V$, will be conserved by $\Psi$ since it is a conformal mapping. 
Therefore we can deduce 
$$\text{span}\{\tau(x) \cdot V(x) \partial_n \, \Phi(x)\}=\mathbb{R}^3,$$
and proceed as in Lemma \ref{bigeo} to prove that subtracting the vector $\displaystyle\frac{1}{\pi}\int_{\partial\mathcal{S}_0} \tau(x) \cdot V(x) \, \partial_n \Phi(x)\, d\sigma$ does not change the above span.
The details are left to the reader.
\end{remark}
%%%
%%%%%
%%%%%%
%%%%%
%%%
\section*{Acknowledgements}

The author wishes to thank  the Agence Nationale de la Recherche, Project DYFICOLTI, grant ANR-13-BS01-0003-01 and Project IFSMACS, grant ANR-15-CE40-0010  for their financial support, furthermore, the Fondation Sciences Math\'ematiques de Paris for their support in the form of the PGSM Phd Fellowship.

\section*{References}

\end{document}